\pdfoutput=1
\newif\ifpersonal
\documentclass[11pt,a4paper,dvipsnames,x11names]{amsart}

\usepackage{amsmath,amsthm,amssymb,mathrsfs,mathtools,tensor,eucal,thmtools} 
\usepackage[all,cmtip]{xy}

\usepackage[LGR,T1]{fontenc} 
\usepackage[utf8]{inputenc} 
\usepackage{CJKutf8}

\usepackage{microtype,inconsolata} 
\usepackage{enumerate,comment,braket,xspace,tikz,tikz-cd,csquotes} 
\usetikzlibrary{shapes.geometric}
\usepackage[centering,vscale=0.7,hscale=0.7]{geometry}
\usepackage[hidelinks,colorlinks=true,linkcolor=Red4,citecolor=RoyalBlue]{hyperref}
\usepackage[capitalize]{cleveref}
\usepackage{xcolor}
\usepackage{xpatch}

\usepackage{mathpazo}
\usepackage{euler}

\usepackage{stackrel}
\usepackage{faktor}
\usepackage{appendix}
\usetikzlibrary{decorations.markings}

\numberwithin{equation}{subsection}
\theoremstyle{plain}
\newtheorem{theorem}[equation]{Theorem}
\newtheorem{lemma}[equation]{Lemma}
\newtheorem*{example*}{Example}
\newtheorem*{definition*}{Definition}
\newtheorem{proposition}[equation]{Proposition}

\newtheorem{corollary}[equation]{Corollary}
\newtheorem{assumption}[equation]{Assumption}

\theoremstyle{definition}
\newtheorem{definition}[equation]{Definition}
\newtheorem{notation}[equation]{Notation}
\newtheorem{example}[equation]{Example}
\newtheorem{remark}[equation]{Remark}

\newtheorem{recollection}[equation]{Recollection}

\newtheorem{construction}[equation]{Construction}

\ifpersonal
\newcommand{\personal}[1]{\textcolor[rgb]{0,0,1}{(Personal: #1)}}
\newcommand{\discussion}[1]{\textcolor{violet}{(Discussion: #1)}}
\else
\newcommand{\personal}[1]{\ignorespaces}
\newcommand{\discussion}[1]{\ignorespaces}
\fi

\def\A{\mathbb{A}}

\def\G{\mathbb{G}}
\def\N{\mathbb{N}}

\def\S{\mathbb{S}}
\def\ZZ{\mathbb{Z}}

\def\C{\mathcal{C}}
\def\D{\mathcal{D}}
\def\E{\mathcal{E}}

\def\X{\mathcal{X}}
\def\Y{\mathcal{Y}}

\def\M{\mathcal{M}}
\def\OO{\mathcal{O}}
\def\S{\textsf{S}}

\def\Prlo{\mathrm{Pr}^{\mathrm{L}, \omega}}

\def\ST{\overline{\mathrm{ST}}_R}
\def\Funst{\mathrm{Fun}^{\mathrm{st}}}
\def\FunR{\mathrm{Fun}^{\mathrm{R}}}

\DeclareMathOperator{\Frac}{Frac}
\DeclareMathOperator{\Gr}{Gr}
\DeclareMathOperator{\Rep}{Rep}
\DeclareMathOperator{\QCoh}{QCoh}
\DeclareMathOperator{\dAff}{dAff}
\DeclareMathOperator{\dSt}{dSt}

\DeclareMathOperator{\Aff}{Aff}

\DeclareMathOperator{\fib}{fib}
\DeclareMathOperator{\cof}{cof}
\DeclareMathOperator{\Loc}{Loc}

\DeclareMathOperator{\Hom}{Hom}
\DeclareMathOperator{\Cons}{Cons}

\DeclareMathOperator{\Fun}{Fun}
\DeclareMathOperator{\Sh}{Sh}

\DeclareMathOperator{\Mod}{Mod}
\DeclareMathOperator{\LMod}{LMod}
\DeclareMathOperator{\Ext}{Ext}

\DeclareMathOperator{\St}{St}

\DeclareMathOperator{\Map}{Map}

\DeclareMathOperator{\End}{End}
\DeclareMathOperator{\Cat}{\mathrm{Cat}}
\DeclareMathOperator{\Spc}{Spc}
\DeclareMathOperator{\ev}{ev}
\DeclareMathOperator{\coev}{coev}
\DeclareMathOperator{\Id}{Id}

\DeclareMathOperator{\Perf}{Perf}
\DeclareMathOperator{\Perv}{Perv}
\DeclareMathOperator{\Spec}{Spec}

\newcommand{\colim}{\operatornamewithlimits{colim}}

\makeatletter
\renewenvironment{proof}[1][\relax]{\par
  \pushQED{\qed}%
  \normalfont \topsep6\p@\@plus6\p@\relax
  \trivlist
  \item[\hskip\labelsep\itshape
    \ifx#1\relax \proofname\else\proofname{} of #1\fi\@addpunct{.}]\ignorespaces
}{%
  \popQED\endtrivlist\@endpefalse
}
\makeatother

\makeatletter
\let\save@mathaccent\mathaccent
\newcommand*\if@single[3]{%
	\setbox0\hbox{${\mathaccent"0362{#1}}^H$}%
	\setbox2\hbox{${\mathaccent"0362{\kern0pt#1}}^H$}%
	\ifdim\ht0=\ht2 #3\else #2\fi
}
\newcommand*\rel@kern[1]{\kern#1\dimexpr\macc@kerna}
\newcommand*\widebar[1]{\@ifnextchar^{{\wide@bar{#1}{0}}}{\wide@bar{#1}{1}}}
\newcommand*\wide@bar[2]{\if@single{#1}{\wide@bar@{#1}{#2}{1}}{\wide@bar@{#1}{#2}{2}}}
\newcommand*\wide@bar@[3]{%
	\begingroup
	\def\mathaccent##1##2{%
		\let\mathaccent\save@mathaccent
		\if#32 \let\macc@nucleus\first@char \fi
		\setbox\z@\hbox{$\macc@style{\macc@nucleus}_{}$}%
		\setbox\tw@\hbox{$\macc@style{\macc@nucleus}{}_{}$}%
		\dimen@\wd\tw@
		\advance\dimen@-\wd\z@
		\divide\dimen@ 3
		\@tempdima\wd\tw@
		\advance\@tempdima-\scriptspace
		\divide\@tempdima 10
		\advance\dimen@-\@tempdima
		\ifdim\dimen@>\z@ \dimen@0pt\fi
		\rel@kern{0.6}\kern-\dimen@
		\if#31
		\overline{\rel@kern{-0.6}\kern\dimen@\macc@nucleus\rel@kern{0.4}\kern\dimen@}%
		\advance\dimen@0.4\dimexpr\macc@kerna
		\let\final@kern#2%
		\ifdim\dimen@<\z@ \let\final@kern1\fi
		\if\final@kern1 \kern-\dimen@\fi
		\else
		\overline{\rel@kern{-0.6}\kern\dimen@#1}%
		\fi
	}%
	\macc@depth\@ne
	\let\math@bgroup\@empty \let\math@egroup\macc@set@skewchar
	\mathsurround\z@ \frozen@everymath{\mathgroup\macc@group\relax}%
	\macc@set@skewchar\relax
	\let\mathaccentV\macc@nested@a
	\if#31
	\macc@nested@a\relax111{#1}%
	\else
	\def\gobble@till@marker##1\endmarker{}%
	\futurelet\first@char\gobble@till@marker#1\endmarker
	\ifcat\noexpand\first@char A\else
	\def\first@char{}%
	\fi
	\macc@nested@a\relax111{\first@char}%
	\fi
	\endgroup
}
\makeatother

\usetikzlibrary{decorations.markings} 
\tikzset{
  closed/.style = {decoration = {markings, mark = at position 0.5 with { \node[transform shape, xscale = .8, yscale=.4] {/}; } }, postaction = {decorate} },
  open/.style = {decoration = {markings, mark = at position 0.5 with { \node[transform shape, scale = .7] {$\circ$}; } }, postaction = {decorate} }
}

\makeatletter
\tikzcdset{
  open/.code     = {\tikzcdset{hook, circled};},
  closed/.code   = {\tikzcdset{hook, slashed};},
  open'/.code    = {\tikzcdset{hook', circled};},
  closed'/.code  = {\tikzcdset{hook', slashed};},
  circled/.code  = {\tikzcdset{markwith = {\draw (0,0) circle (.375ex);}};},
  slashed/.code  = {\tikzcdset{markwith = {\draw[-] (-.4ex,-.4ex) -- (.4ex,.4ex);}};},
  markwith/.code ={
    \pgfutil@ifundefined%
    {tikz@library@decorations.markings@loaded}%
    {\pgfutil@packageerror{tikz-cd}{You need to say %
      \string\usetikzlibrary{decorations.markings} to use arrows with markings}{}}{}%
    \pgfkeysalso{/tikz/postaction = {
      /tikz/decorate,
      /tikz/decoration={markings, mark = at position 0.5 with {#1}}}
    }
  },
}
\makeatother

\newcommand{\leftrarrows}{\mathrel{\raise.75ex\hbox{\oalign{%
  $\scriptstyle\leftarrow$\cr
  \vrule width0pt height.5ex$\hfil\scriptstyle\relbar$\cr}}}}
\newcommand{\lrightarrows}{\mathrel{\raise.75ex\hbox{\oalign{%
  $\scriptstyle\relbar$\hfil\cr
  $\scriptstyle\vrule width0pt height.5ex\smash\rightarrow$\cr}}}}
\newcommand{\Rrelbar}{\mathrel{\raise.75ex\hbox{\oalign{%
  $\scriptstyle\relbar$\cr
  \vrule width0pt height.5ex$\scriptstyle\relbar$}}}}

\makeatletter
\def\leftrightarrowsfill@{\arrowfill@\leftrarrows\Rrelbar\lrightarrows}
\newcommand{\xleftrightarrows}[2][]{\ext@arrow 3399\leftrightarrowsfill@{#1}{#2}}
\makeatother

\NewCommandCopy{\notocsection}{\section}
\xpatchcmd{\notocsection}{{1}}{{1001}}{}{}

\begin{document}

\title{Good moduli for moduli of objects}
\author{Enrico Lampetti}
\address{Sorbonne Université and Université Paris Cité, CNRS, IMJ-PRG, F-75005 Paris, France}
\email{enrico.lampetti@imj-prg.fr}

\subjclass[2020]{}

\begin{abstract}
    We construct good moduli spaces from moduli of objects in the sense of Toën-Vaquié.
   As an application, we construct good moduli spaces for perverse sheaves. 
\end{abstract}

\maketitle

\tableofcontents

\section{Introduction}
	Good moduli spaces for algebraic stacks were introduced by Alper in \cite{Alp} and have played a prominent role in moduli theory as generalizations of Mumford's good GIT quotients \cite{GIT} and Abramovich-Olsson-Vistoli's tame stacks \cite{AOV}.  
	For the sake of this introduction, recall that a good moduli space for an algebraic stack $\X$ is a qcqs morphism $q \colon \X \to X$ to an algebraic space such that the pushforward along $q$ is exact on quasi-coherent sheaves and such that the canonical morphism $\OO_X \to q_\ast \OO_\X$ is an equivalence.
	If a good moduli space $q \colon \X \to X$ exists, then it is unique, as $q$ is universal for maps from $\X$ to algebraic spaces (\cite[Theorem 6.6]{Alp}).		         
    In a similar way that the Keel-Mori Theorem \cite{KM} provides an intrinsic way to show that a Deligne-Mumford stack admits a coarse moduli space, a fundamental result of Alper-Halpern Leistner-Heinloth \cite[Theorem A]{AHLH} provides an intrinsic way to show that an algebraic stack admits a good moduli space.
    The existence of a good moduli space for an algebraic stack $\X$ has many pleasant consequences.
	It allows for example to give a local presentation of $\X$ by quotient stacks and prove the compact generation of $\QCoh(\X)$ (\cite{AHR}). 
	Also, the existence of good moduli spaces can be exploited to construct BPS Lie algebras and unlock the study of the cohomology of $\X$ via cohomological Hall algebras, see e.g. \cite{DHM, Davison, BDNIKP, Hennecart}. \medskip
	
	As a consequence of \cite[Theorem A]{AHLH}, given a nice enough cocomplete abelian category $\C$, the authors construct a good moduli space for Artin-Zhang's moduli of objects \cite{AZ} parametrizing compact objects in $\C$. 
	 This good moduli space has been further studied by Fernandez Herrero-Lennen-Makarova in \cite{FHLM} and, when it exists, is necessarily \textit{proper} by \cite[Theorem 7.23]{AHLH}.
	 Still many interesting moduli encountered in algebraic geometry are not proper.
	As an example, for a complex algebraic variety $X$ one can show by hand that the algebraic stack $\mathbf{Loc}(X)$ parametrizing local systems with finite dimensional stalks on $X$ admits a good moduli space, that is the character variety of $X$, which is non proper in general. 
	In particular, $\mathbf{Loc}(X)$ does not fit in the framework of \cite{AZ}.
	The point is that the compact objects of the cocomplete category $\Loc(X; \Mod_\mathbb{C})$ do not need to have finite dimensional stalks, as the following example shows :
	\begin{example}\label{example_intr}
	Consider the local system $L$ on $X = \mathbb{G}_m$ corresponding to the representation of $\pi_1(X) \simeq \ZZ$ defined by $\bigoplus_{\ZZ} \mathbb{C}$ with $\ZZ$-action given by translation.
	The object $L$ is compact in $\Loc(X; \Mod_\mathbb{C})$ as it corepresents the stalk functor, but its stalks are not finite dimensional.
\end{example}
	In particular, the above example shows that to single-out local systems with finite dimensional stalks, one needs to impose a condition stronger than compactness.
	As proved in \cite{Exodromy, Perv_moduli}, it turns out that local systems with finite dimensional stalks correspond to \textit{pseudo-perfect} objects of $\Loc(X; \Mod_\mathbb{C})$ in the sense of Toën-Vaquié \cite{TV}.
	Thus, to construct a good moduli for  $\mathbf{Loc}(X)$ in an intrinsic way as in \cite{AHLH}, one is naturally led to replace Artin-Zhang moduli by Toën-Vaquié moduli of objects parametrizing pseudo-perfect objects. \medskip
	
	Note that pseudo-perfect objects require $\infty$-categories to make sense \cite{HAG-I, HAG-II, HTT,HA}.
	Furthermore in Toën-Vaquié's setting, $\C$ is not an abelian category anymore, but a triangulated category (actually its $\infty$-categorically avatar, that is a stable $\infty$-category).
	Then, the original abelian category of interest appears as the heart $\C^\heartsuit$ of a $t$-structure $\tau$ on $\C$.
	The next issue is then to extract from Toën-Vaquié moduli of objects $\M_\C$ a substack parametrizing objects in $\C^\heartsuit$.
	Such problem has already been encountered by Lieblich \cite{Lieblich} for the stack of perfect complexes and led to the notion of \textit{$\tau$-flat} objects in $\C$.
	Following \cite{AP, Polishchuk, stability_in_families, HDR}, we define 
\[
\M_\C^\heartsuit \subset \M_\C
\]
as the substack parametrizing $\tau$-flat objects in $\C$.
    The stack $\M_\C^\heartsuit$  is the central player of this paper.	
	Before moving to our main results, observe that working at this level of generality gives access to new moduli spaces that cannot be constructed via classical algebraic geometry.
	For example, the derived moduli of perverse sheaves from \cite{Exodromy, exodromyconicality, Perv_moduli} provides a far reaching generalization of the character stack and fits into Toën-Vaquié's framework. \medskip
	
	Let us now introduce the main result of this paper.
	Let $k$ be a noetherian ring of characteristic $0$ and $\C$ be a finite type $k$-linear $\infty$-category equipped with an admissible $t$-structure $\tau$ (\cref{accessible_t_structure}).
	Let $\M_\C$ be Toën-Vaquié moduli of objects parametrizing pseudo-perfect objects in $\C$ and $\M_\C^\heartsuit \subset \M_\C$ be the substack parametrizing $\tau$-flat pseudo-perfect objects in $\C$ (\cref{def_tau_flat}).	
	Under suitable natural assumptions on $\tau$ called openness of flatness (\cref{opennes_flatness}), the substack    
$\M_\C^\heartsuit$ is an open substack of $\M_\C$.	
	The last input necessary to construct good moduli spaces is a technical assumption on $\tau$ loosely stated as 
	
\begin{assumption}\label{Assumption_subobject_intro}
For any DVR $R$ over $k$, for every pseudo-perfect object $F$ with coefficients in $R$ and lying in $\C_R^\heartsuit$, the subobjects of $F$ are also pseudo-perfect.
\end{assumption}	

   \cref{Assumption_subobject_intro} is satisfied in many cases of interest.

\begin{example}\label{example_Perv}
	For local systems on a complex algebraic variety, pseudo-perfect objects over $R$ correspond to local systems of $R$-modules with finitely generated stalks.
     Then \cref{Assumption_subobject_intro} is immediately satisfied.
\end{example}
	
\begin{example}
	For perverse sheaves, pseudo-perfect objects over $R$ correspond to perverse sheaves of $R$-modules with perfect stalks. 
	In that case, \cref{Assumption_subobject_intro} is satisfied in virtue of  \cref{Perv_subobjects}.
\end{example}
	
\begin{example}
	Let $X$ be a smooth scheme over $k$.
	For coherent sheaves, a theorem of Ben Zvi-Nadler-Preygel \cite[Theorem 3.0.2]{BZNP} describes the pseudo-perfect objects over $R$ as the coherent sheaves over $X_R$ with proper support.
    In that case, \cref{Assumption_subobject_intro} is thus obvious.
\end{example}

	In the above setup, let $\Spec(\kappa) \to \Spec(k)$ be a closed point with $\kappa$ algebraically closed. 
	Our main result is the following

\begin{theorem}[{\cref{good_general}}]\label{good_general_intr}	
	Every closed quasi-compact substack $\X \subset \M_\C^\heartsuit$ admits a separated good moduli space $X$. Moreover, the $\kappa$-points of $X$ parametrize pseudo-perfect semisimple objects of $\C_\kappa^\heartsuit$ lying over $\X$.
\end{theorem}
 	
 	Let us point out that a notion of derived good moduli spaces has been introduced by Ahlqvist-Hekking-Pernice-Savvas in \cite{derived_good} and that often derived enhancement of a good moduli space come for free thanks to \cite[Theorem 2.12]{derived_good}. \medskip
 	
	The main tool for constructing good moduli spaces is \cite[Theorem A]{AHLH}.
    It states that an algebraic stack of finite presentation over a noetherian ring of characteristic zero admits a separated good moduli space if and only if it is "$\Theta$-reductive" (\cref{def_Theta_reductive}) and "$\mathrm{S}$-complete" (\cref{def_S_complete}).
	These two properties are of valuative nature and involve the algebraic stacks
\[
\Theta_R \coloneqq \left[\faktor{A_{1,R}}{\mathbb{G}_{m,R}}\right], \qquad \ST \coloneqq \left[\faktor{\Spec(R[s,t]/(st-\pi))}{\G_{m,R}}\right],
\]
for $R$ a DVR and $\pi \in R$ a uniformizer.
	For moduli of objects one can explicitly describe $\Theta$-reductiveness and $S$-completeness (\cref{flat_perf_Theta} and \cref{flat_perf_ST}).
	The explicit description of $\Theta$-reductiveness repose on the work Moulinos \cite{Moulinos}, whose main result implies the existence of a symmetric monoidal equivalence 
\[
\QCoh(\Theta_R) \simeq \Rep(\ZZ; \Mod_R),
\]
where the category on the right hand side is the category of filtered $R$-modules.
	The explicit description of $\mathrm{S}$-completeness repose on Appendix \ref{Appendix_QCoh_ST_Theta}, where we use similar techniques to \cite{Moulinos}. For $R$ a DVR with uniformizer $\pi$, our main results is the existence of a symmetric monoidal equivalence 
\[
\QCoh(\ST) \simeq \Mod_{R[s,t] / (st-\pi)} (\Rep(\ZZ^{ds}; \Mod_R)),
\]
where the category on the right hand side is the category of graded $R$-modules equipped with an action of $s$ of weight 1 and an action of $t$ of weight $-1$, subjects to the relations $st = \pi = ts$. \medskip

    \cref{good_general_intr} reduces the existence of good moduli spaces from moduli of objects to the existence of closed quasi-compact substacks $\M_\C^\heartsuit$.
    We exhibit a large class of such in \cref{subsection_extension} and we give a hint of the construction below.
    For this purpose, let $\Spec(\kappa) \to \Spec(k)$ be a closed point with $\kappa$ algebraically closed, let $n \in \N$ and let $\underline{F} \subset \C_\kappa^\heartsuit$ be a finite set of pseudo-perfect semisimple objects over $\kappa$.
	Then \cref{construction_ext} shows that there exists a quasi-compact closed substack
\[
\M_\C^{\heartsuit, \underline{F}, N} \to \M_\C^\heartsuit
\]
whose closed $\kappa$-points are in bijection with extension of length at most $N$ of objects in $\underline{F}$.
	Hence \cref{good_general_intr} gives
\begin{theorem}[{\cref{good_ext}}]\label{good_ext_intro}
	In the setting of \cref{good_general_intr}, the algebraic stack $\M_\C^{\heartsuit, \underline{F}, N}$ admits a separated good moduli space $\mathrm{M}_\C^{\heartsuit, \underline{F}, N}$. Moreover, the $\kappa$-points of $\mathrm{M}_\C^{\heartsuit, \underline{F}, N}$ are in bijection with direct sum of at most $N$ objects in $\underline{F}$, possibly with repetitions.
\end{theorem}

   Let us now describe particular instances of \cref{good_general_intr}.

\subsection{Perverse sheaves}
	The role of perverse sheaves is central in the study of singularities and they naturally generalize local system. Moreover they are linked to the study of differential equation with regular singularities (i.e., simple poles) via the Riemann-Hilbert correspondence (\cite{Kashiwara, Mebkhout}).
	Perverse sheaves originated form the theory of $\mathcal{D}$-modules via the Riemann-Hilbert correspondence as the shifted solution complex and the shifted de Rham complex of a holonomic $\mathcal{D}$-module on a complex manifold (\cite{Kashiwara, KS}).
	An axiomatic background has been introduced in the influential \cite{BBDG}.
	For a stratified space $(X,P)$, let $\mathfrak{p}\colon P \to \ZZ$ be a function, called \textit{perversity}.
	We can attach to $\mathfrak{p}$ a perverse $t$-structure ${}^{\mathfrak{p}}\tau$ on the derived category of constructible sheaves on $X$ and the category of perverse sheaves is defined as its heart.
	The derived stack ${}^{\mathfrak{p}}\mathbf{Perv}_P(X)$ of perverse sheaves on $(X,P)$ with perversity $\mathfrak{p}$ has been constructed in \cite{Exodromy, exodromyconicality, Perv_moduli} and fits in our framework.
	We denote by $t_0{}^{\mathfrak{p}}\mathbf{Perv}_P(X)$ the truncation of ${}^{\mathfrak{p}}\mathbf{Perv}_P(X)$, which is an algebraic stack locally of finite presentation over $k$.\medskip
	
	We will provide the existence of good moduli spaces for perverse sheaves for a large class of stratified space, including real algebraic varieties equipped and compact analytic varieties equipped with Whitney stratifications. 
	The key result here is \cref{Perv_subobjects}:
	it shows that a subobject of a perverse sheaf with perfect stalks also has perfect stalks as long as we take coefficients in $\Mod_R$ for $R$ a regular noetherian ring.
	As already discussed in \cref{example_Perv}, this is precisely the translation of \cref{Assumption_subobject} to the case of perverse sheaves.
	Hence \cref{good_general_intr} and \cref{good_ext_intro} apply.\medskip
	
	If $(X, P)$ is a Whitney stratified manifold (\cref{Withney_manifold_def}) and $\mathfrak{p}\colon P \to \ZZ$ the middle perversity (\cref{middle_perversity_def}), we can prove the existence of a good moduli space for the entire algebraic stack of perverse sheaves, hence generalizing \cref{good_general_intr}.
	More precisely, in the above setting and for $\Spec(\kappa) \to \Spec(k)$ a closed point with $\kappa$ algebraically closed, our main result is the following
\begin{theorem}[{\cref{good_Perv_algebraic}}]
	The algebraic stack $t_0{}^{\mathfrak{p}}\mathbf{Perv}_P(X)$ admits a separated good moduli space $t_0{}^{\mathfrak{p}}\Perv_P(X)$.
	Moreover, the $\kappa$-points of $t_0{}^{\mathfrak{p}}\Perv_P(X)$ parametrize semisimple perverse sheaves with perfect stalks.
	Furthermore, $t_0{}^{\mathfrak{p}}\Perv_P(X)$ admits a natural derived enhancement ${}^{\mathfrak{p}}\Perv_P(X)$ which is a derived good moduli space for ${}^{\mathfrak{p}}\mathbf{Perv}_P(X)$.
\end{theorem}

\subsection{Constructible sheaves and Stokes data}
	Other possible generalizations for local systems are constructible sheaves and Stokes data.
	Although the corresponding moduli fit in our framework, after writing this paper we realized that these moduli satisfy a form of \textit{Hartogs' principle} \cite[Remark 3.51]{AHLH}, that readily implies $\Theta$-reductiveness and $S$-completeness.
	Moreover, they admit natural quasi-compact open and closed substacks.
    This makes the construction of the good moduli spaces easier and it will appear in a companion paper \cite{good_Stokes}.
	However, in order to characterize closed points of the moduli of constructible sheaves and of Stokes data as well as points of their respective good moduli spaces, we need all the strength of the results contained in the present paper.

\subsection{Linear overview}
	In \cref{generalities} we recall the definitions and results about moduli of objects in the sense of Toën-Vaquié and of good moduli spaces that we will need in the rest of the paper. \medskip
	
	In \cref{general_case} we prove the existence of good moduli spaces for moduli of objects of a finite type category $\C$ equipped with a $t$-structure satisfying \cref{Assumption_subobject}. 
	More precisely, in \cref{section_Theta_reductive} we prove that $\M_\C^\heartsuit$ is $\Theta$-reductive, while in  \cref{section_S_complete} we prove that $\M_\C^\heartsuit$ is $\mathrm{S}$-complete.
	As a corollary, we obtain a characterization of closed points of $\M_\C^\heartsuit$ as those whose corresponding object in $\C^\heartsuit$ is semisimple (\cref{closed_point_is_semisimple}).
	We then move on giving the main result of this paper, namely the existence of good moduli spaces for quasi-compact closed substacks of moduli of objects (\cref{good_general}).
	We give applications of the above result in \cref{good_ext}. \medskip
	
	In \cref{section_good_Perv} we specialize our discussion to the case of perverse sheaves.
	More precisely, we show in \cref{Perv_subobjects} that the perverse $t$-structure satisfy \cref{Assumption_subobject} (see \cref{algebraic_strat}).
	In 
	\medskip
	
	In Appendix \ref{Appendix_QCoh_ST_Theta} we recall results on the derived category $\QCoh(\Theta_R)$ from \cite{Moulinos} and we provide an explicit description of the derived category $\QCoh(\ST)$. These results are used in \cref{general_case}.

\subsection{Notation}
We introduce the following running notations.
\begin{itemize}\itemsep=0.2cm
    \item We fix $k$ a discrete noetherian ring of characteristic $0$, that is, a discrete noetherian ring containing $\mathbb{Q}$;
    \item $\mathrm{Grpd}$ is the category of groupoids;
    \item $\mathrm{Spc}$ is the $\infty$-category of spaces;
    \item $\mathrm{Sp}$ is the $\infty$-category of spectra;
    \item $\Cat_\infty$ is the $\infty$-category of $\infty$-categories;
    \item $(-)^\omega \colon \Cat_\infty \to \Cat_\infty$ is the functor that assign to an $\infty$-category its full subcategory spanned compact objects;
    \item $(-)^\simeq \colon \Cat_\infty \to \Spc$ is the functor that assign to an $\infty$-category its maximal $\infty$-sub-grupoid;
    \item $\Pr^{\mathrm{L}}_k$ is the $\infty$-category of presentable $k$-linear $\infty$-categories with left adjoints functors;
    \item $\Prlo_k$ is the $\infty$-category of presentable compactly generated $k$-linear $\infty$-categories with left adjoints functors;
    \item $\Aff_k$ is the opposite of the category of $k$-algebras;
    \item $\dAff_k$ is the opposite of the $\infty$-categories of simplicial $k$-algebras;
    \item the category of stacks over $k$ is the full subcategory
    \[
    \mathrm{St}_k \subset \Fun(\Aff_k, \mathrm{Grpd})
    \]
    spanned by stacks over $k$;
    \item the $\infty$-category of derived stacks over $k$ is the full subcategory
    \[
    \dSt_k \subset \Fun(\dAff_k, \mathrm{Spc})
    \]
    spanned by derived stacks over $k$;
    \item denote by $\tau_{\geq 1}: \mathrm{Spc} \to \mathrm{Grpd}$ the functor that assign to a space its homotopy category, which is a groupoid.
    Denote by $i: \Aff_k \to \dAff_k$ the canonical inclusion.
    The composition 
    \[
    \Fun(\dAff_k, \mathrm{Spc}) \xrightarrow{- \circ i} \Fun(\Aff_k, \mathrm{Spc}) \xrightarrow{\tau_{\geq 1} \circ -} \Fun(\Aff_k, \mathrm{Grpd})
    \]
    restricts to a functor $t_0 \colon \dSt_k \to \St_k$.
    We will refer to $t_0$ as the truncation functor.
\end{itemize}

\notocsection{Acknowledgments}
	I would like to thank my advisors Mauro Porta and Jean-Baptiste Teyssier for their suggestions and their constant support.
	I own much to Matteo Verni with whom I discussed many of the ideas of this paper.
	I would also like to thank Matteo Montagnani, Drimik Roy and Sasha Zakharov for helpful conversations.
	I would like to warmly thank Andres Fernandez Herrero for many useful comments about a first draft of this paper.

\section{Moduli of objects and $t$-structures}\label{generalities}

	In this section we recall the main definitions, constructions and results needed in the rest of the paper.

\subsection{Families of objects}\label{section_moduli_of_objects}
	In this paragraph we recall properties of the tensor product of presentable $\infty$-categories, which is used to define families of objects. 
	We also discuss some finiteness properties one can impose on compactly generated presentable categories.

\begin{recollection}[{\cite[Section 4.8]{HA}}]\label{tensor_product}
	For $\C, \D \in \mathrm{Pr}^\mathrm{L}_k$, we can consider their tensor product
\[
\C \otimes_k \D \simeq \FunR_k(\C^{op}, \D)
\]
where the right hand side is the $\infty$-category of $k$-linear functors commuting with limits.
	This tensor product endows $\Pr^{\mathrm{L}}_k$ with a symmetric monoidal structure which restricts to $\Prlo_k$. 
	When $\C$ is compactly generated, there is a canonical equivalence
\[
\C \otimes_k \D \simeq \Funst_k((\C^{\omega})^{op}, \D)
\]
where the right hand side is the $\infty$-category of exact $k$-linear functors.
\end{recollection}

\begin{remark}\label{functoriality_tens_prod}
	Let $\C, \D, \E \in \Pr_k^{\mathrm{L}}$.
	Then an adjunction
\[
f \colon \D \leftrightarrows \E \colon g
\]
induces an adjunction 
\[
\Id_\C \otimes_k f \colon \C \otimes_k \D\leftrightarrows \C \otimes_k \E \colon \Id_\C \otimes_k g.
\]
	The functor $\Id_\C \otimes_k g$ corresponds to 
\[
g \circ - \colon \FunR_k(\C^{op}, \E) \to C \otimes_k \E \simeq \FunR_k(\C^{op}, \D)
\]
under the equivalences of \cref{tensor_product}.
	Its left adjoint $\Id_\C \otimes_k f$ has no easy description in general.
	Notice indeed that post-composition with $f$ do not preserve functors that commutes with limits.
	Nevertheless, if $\C$ is compactly generated one can see that the adjunction $\Id_\C \otimes_k f  \dashv \Id_\C \otimes_k g$ corresponds to 
\[
f \circ - \colon \Funst_k((\C^{\omega})^{op}, \D) \leftrightarrows \Funst_k((\C^{\omega})^{op}, \D) \colon g \circ -
\]
under the equivalences of \cref{tensor_product}.
\end{remark}

\begin{notation}
In the setting of \cref{functoriality_tens_prod}, we denote the functors $\Id_\C \otimes_k f, \Id_\C \otimes_k g$ respectively by $f_\C, g_\C$.
\end{notation}

	In the stable setting, the following result is a reformulation of \cite[Lemma 2.2.1]{SS}.

\begin{lemma}[{\cite[Lemma 2.4]{AG}}]\label{generators_stable}
	Let $\E \in \Prlo$ stable and $X \subset \E^\omega$ a set of compact objects.
	Then the following are equivalent:
\begin{enumerate}\itemsep=0.2cm
	\item $\forall y \in \E$,
\[
\mathrm{Map}_\E(x,y) \simeq * \in \mathrm{Sp} \ \  \forall x \in X \Rightarrow y \simeq 0;
\]
	\item the $\infty$-category $\E^\omega$ is the smallest full subcategory of $\E$ containing $X$ and stable under finite colimits, shifts and retracts.
\end{enumerate}
\end{lemma}

\begin{definition}\label{compact_gen_def}
	Let $\E \in \Prlo$ stable and $X \subset \E^\omega$ a set of compact objects.
	We say that $X$ is a set of compact generators if it satisfies the equivalent conditions of \cref{generators_stable}.
\end{definition}

\begin{example}
	Let $A \in \mathrm{Alg}_k$.
	Then $A$ is a compact generator of $\Mod_A$ in the sense of \cref{compact_gen_def}.
\end{example}

\begin{definition}
	For $\C \in \mathrm{Pr}^\mathrm{L}_k$ and $X\in \dSt_k$, define the $\infty$-category of $X$-families of objects in $\C$ as
\[
\C_X \coloneqq  \C \otimes_k \QCoh(X).
\]
\end{definition}

\begin{notation}
	Let $\C \in \Pr^{\mathrm{L}}_k$ and $S = \Spec (A) \in \dAff_k$.
	The $\infty$-category of $S$-family of objects in $\C$ will be denoted equivalently by $\C_S$ or by $\C_A$.
\end{notation}

\begin{definition}\label{pseudo_perfect}
	For $\C \in \Prlo_k$ and $X \in \dSt_k$, define the $\infty$-category of pseudo-perfect families over $X$ in $\C$ as the full subcategory
\[
\Funst((C^\omega)^{op}, \Perf(X)) \subset \C_X
\]
spanned by functors valued in $\Perf(X)$.
\end{definition}

In order to introduce finiteness conditions on presentable categories, let us give the following:

\begin{lemma}[{\cite[Theorem D.7.0.7]{SAG}}]\label{finiteness_conditions}
	Let $\C \in \Prlo_k$.
	Then $\C$ is dualizable in $\Pr_k^{\mathrm{L}}$ with dual given by $\C^\vee \simeq \mathrm{Ind}((\C^\omega)^{op})$.
	In particular, there are up to a contractible space of choices unique maps in $\Pr^{\mathrm{L}}_k$
\[
\mathrm{ev}_\C \colon  \C^\vee \otimes_k \C \to \Mod_k
\]
\[
\mathrm{coev}_\C  \colon  \Mod_k \to \C \otimes_k \C^\vee
\]
satisfying the triangular identities.
\end{lemma}

	We will need the following finiteness condition:

\begin{definition}\label{def_finiteness_conditions}
	Let $\C \in \Prlo_k$.
	We say that $\C$ is 
\begin{enumerate}\itemsep=0.2cm
    \item smooth if $\ev_\C$ preserves compact objects;
    \item proper if $\coev_\C$ preserve compact objects;
    \item of finite type if it is a compact object of $\Prlo_{k}$.
\end{enumerate}
\end{definition}

	The next result compares the different finiteness conditions of \cref{def_finiteness_conditions}.

\begin{proposition}[{\cite[Proposition I.4.16]{HDR}}]\label{relation_finiteness_conditions}
	Let $\C \in \Prlo_k$.
	Then:
\begin{enumerate}\itemsep=0.2cm
    \item if $\C$ is of finite type, it is smooth;
    \item if $\C$ is smooth and proper, it is of finite type.
\end{enumerate}
Moreover, if $\C$ is smooth then it admits a compact generator in the sense of \cref{compact_gen_def}.
\end{proposition}

\begin{lemma}\label{lemma_smoothness}
	Let $\C \in \Prlo_k$.
	The following are equivalent:
\begin{enumerate}
	\item $\C$ is smooth;
	\item there exists an equivalence $\C \simeq \LMod_A$ with $A \in \mathrm{Alg}_k$ smooth in the sense of \cite[Definition 4.6.4.13]{HA}.
\end{enumerate}
\end{lemma}

\begin{proof}
This follows by combining \cite[Proposition 11.3.2.4]{SAG} with \cite[Remark 4.2.1.37 \& Remark 4.6.4.15]{HA}.
\end{proof}

\begin{remark}\label{rem_smooth}
	Let $\C$ be smooth and $E \in \C$ be a compact generator.
	Lurie shows in the proof of \cite[Proposition 11.3.2.4]{SAG} that $A \coloneqq \End(E) \in \mathrm{Alg}_k$ is smooth and $\C \simeq \LMod_{A^{rev}}$, where $A^{rev}$ is the opposite algebra of $A$.
\end{remark}

\begin{example}\label{tens_smooth_example}
	In the setting of \cref{rem_smooth} and for $X \in \dSt_k$, it follows from 
\cite[Theorem 4.8.5.16-(4) \& Remark 4.2.1.37]{HA} that we have an equivalence
\[
\C_X \simeq \LMod_A(\Mod_k) \otimes_k \LMod_{\OO_X}(\QCoh(X)) \simeq \LMod_{A \otimes_k \OO_X}(\QCoh(X)).
\]
\end{example}

\begin{corollary}\label{compact_generator_smooth}
	Let $\C \in \Prlo_k$ be a smooth category and $E \in \C$ be a compact generator.
	Let $X$ be a derived stack over $k$.
	Then $F \in \C_X$ is pseudo-perfect if and only if $F(E) \in \Perf(X)$.
\end{corollary}

\begin{proof}
	Since $E$ generates $\C^\omega$ under shifts, finite colimits and retracts and $\Perf(X)$ is closed in $\QCoh(X)$ under shifts finite (co)limits and retracts, the result follows.
\end{proof}

\begin{lemma}[{\cite[Corollary I.5.7]{HDR}}]\label{pseudo_perfect_vs_compact}
	Let $\C \in \Prlo_k$ and let $\Spec (A) \in \dAff_k$.
\begin{enumerate}\itemsep=0.2cm
    \item If $\C$ is smooth, then any pseudo-perfect family in $\C_A$ is compact.
    \item If $\C$ is proper, then any compact object in $\C_A$ is pseudo-perfect.
\end{enumerate}
\end{lemma}

\begin{remark}\label{cpt_gen_tens_product}
	Let $\C \in \Prlo_k$ smooth and $\Spec (B) \in \dAff_k$.
	Let $E$ be a compact generator of $\C$ (\cref{relation_finiteness_conditions}).
	Then $E \otimes_k A$ is a compact generator of $\C_A$.
	This follows by the functor $\mathrm{Forget}\colon \C_A \to \C$ being conservative and the equivalence
	\[
\Map_{\C_A}(E \otimes_k A, -) \simeq \Map_{\C}(E, \mathrm{Forget}(-)) \colon \C_A \to \Mod_k.
\]
\end{remark}

\begin{lemma}\label{pseudo_perf_cpt_Hom_affine}
	Let $\C \in \Prlo_k$ smooth.
	Let $\Spec (A) \in \dAff_k$ and $F_1, F_2 \in \C_A$ with $F_1$ compact and $F_2$ pseudo-perfect.
	Then $\Hom_{\C_A}(F_1, F_2) \in Perf(A)$.
\end{lemma}

\begin{proof}
	It is enough to prove \cref{pseudo_perf_cpt_Hom_affine} for a set of compact objects generating $\C_A^\omega$ under finite colimits, shifts and retracts. 
	Let a compact generator $E \in \C$ (\cref{relation_finiteness_conditions}).
	By \cref{cpt_gen_tens_product}, the object $E \otimes_k A$ is a compact generator of $\C_A$.
	By \cref{generators_stable}, we are reduced to prove that
\[
\Hom_{\C_A}(E \otimes_k A, F_2) \simeq F_2(E) \in \Perf(A),
\]
which holds by our hypothesis that $F_2$ is pseudo-perfect.
\end{proof}

\begin{corollary}\label{pseudo_perf_Hom_affine}
	Let $\C \in \Prlo_k$ smooth. 
	Let $\Spec (A) \in \dAff_k$ and $F_1, F_2 \in \C_A$ pseudo-perfect.
	Then $\Hom_{\C_A}(F_1, F_2) \in Perf(A)$.
\end{corollary}

\begin{proof}
	By \cref{pseudo_perfect_vs_compact} we know that $F_1$ is compact. The result then follows by \cref{pseudo_perf_cpt_Hom_affine}.
\end{proof}

\begin{lemma}\label{descent_left_dualizable}
	Let $I \to \mathrm{Alg}(\Pr^\mathrm{L}_k)$, $i \mapsto \D_i$, be a diagram and let $\D$ be its limit.
	Let $p_i \colon \D \to \D_i$ be the projection.
	Let $A \in \mathrm{Alg}(\D)$ and put $A_i \coloneqq p_i(A)$.
	Then an object $F \in \LMod_A(\D)$ is left-dualizable if and only if every projection $F_i \in \LMod_{A_i}(\D_i)$ is left-dualizable.
\end{lemma}

\begin{proof}
	The direct implication is obvious since every $p_i \colon \C \to \C_i$ is monoidal.
	Let now $F \in \LMod_A(\C)$ and suppose that every projection $F_i \coloneqq p_i(F)$ is left dualizable in $LMod_{A_i}(\D_i)$. 
	Choose duality data $(F_i, F_i^\vee, \coev_{F_i}, \ev_{F_i})$ for each $i \in I$. 	By \cite[Remark 4.6.2.11]{HA}, these duality data are uniquely determined by the $F_i$'s up to a contractible space of choices.
	Hence we have a compatible system of duality data. This defines a duality datum on $F$, proving that $F$ is left-dualizable.
\end{proof}

\begin{lemma}\label{left_dualizable_lemma}
	Let $\C \in \Prlo_k$ smooth and $X \in \dSt_k$.
	Then every pseudo-perfect family $F \in \C_X$ is left-dualizable.
\end{lemma}

\begin{proof}
	By \cref{descent_left_dualizable}, it is enough to prove that for every $f \colon \Spec (B) \to X$, the object $f^\ast F \in \C_B$ is left-dualizable.
	Hence we may assume that $X = \Spec (B)$.
	By \cref{lemma_smoothness}, we can assume that $\C= \LMod_A$ with $A \in \mathrm{Alg}_k$ smooth.
	By \cref{tens_smooth_example}, we have an equivalence
\[
\C_B \simeq \LMod_{A \otimes_k B}(\Mod_B).
\]
Moreover, the above equivalence makes the diagram
\[\begin{tikzcd}[sep=small]
	{\Funst_k((\C^\omega)^{op}; \Mod_B)} && {\LMod_A(\Mod_k) \otimes_k \Mod_B} && {\LMod_{A \otimes_k B}(\Mod_k)} \\
	\\
	&& {\Mod_B}
	\arrow["{\ev_A}"', from=1-1, to=3-3]
	\arrow["\sim"', from=1-3, to=1-1]
	\arrow["\sim", from=1-3, to=1-5]
	\arrow["{\mathrm{Forget}\otimes_k \Id_B}", from=1-3, to=3-3]
	\arrow[from=1-5, to=3-3]
\end{tikzcd}\]
commute.
	Since left-dualizability can be checked after applying the right diagonal arrow by \cite[Proposition 4.6.4.12-(6)]{HA}, we are left to show that $F(E) \in \Perf(B)$.
	Since $F$ is pseudo-perfect, the conclusion follows.
\end{proof}

\begin{lemma}\label{Hom_pseudo_perf_pullback}
	Let $\C \in \Prlo_k$ smooth.
	Let $f \colon Y \to X$ in $\dSt_k$ and let $F_1, F_2 \in \C_X$ with $F_1$ pseudo-perfect.
	There is a canonical equivalence
\[
	f^\ast \Hom_{\C_X}(F_1, F_2) \simeq \Hom_{\C_B}(f_\C^\ast F_1, f_\C^\ast F_2) .
\]
\end{lemma}

\begin{proof}
	By \cref{left_dualizable_lemma}, the object $F_1 \in \C_X$ is left-dualizable.
	Since $f^\ast_\C \colon \C_X \to \C_Y$ is monoidal, it follows that $f^\ast F_1 \in \C_Y$ is left-dualizable.
	Hence we get
\[
f^\ast\Hom_{\C_X}(F_1, F_2) \simeq
f^\ast(F_2 \otimes_{X} F_1^\vee) \simeq
f_\C^\ast F_2 \otimes_Y (f_\C^\ast F_1)^\vee \simeq
\Hom_{\C_Y}(f_\C^\ast F_1, f_\C^\ast F_2).
\]
\end{proof}

\begin{corollary}\label{pseudo_perf_Hom}
	Let $\C \in \Prlo_k$ smooth.
	Let $X \in \dSt_k$ and $F_1, F_2 \in \C_X$ pseudo-perfect.
	Then $\Hom_{\C_X}(F_1, F_2) \in \Perf(X)$.
\end{corollary}

\begin{proof}
	Since $\Perf(-)$ satisfies faithfully flat descent by \cite[Proposition 2.8.4.2 \& Corollary D.6.3.3]{SAG}, the result follows by combining \cref{Hom_pseudo_perf_pullback} with \cref{pseudo_perf_Hom_affine}.
\end{proof}

\subsection{Moduli of objects and $t$-structures}\label{moduli_of_objects}
	In this paragraph we introduce the moduli of objects of an $\infty$-category of \textit{finite type}.
	This construction firstly appeared in \cite{TV}, but we will follow the presentation given in \cite[Section 1.5]{HDR}.

\begin{recollection}\label{moduli_of_objects_recollection}
	Let $\C \in \Prlo_k$.
	Consider the presheaf
\[
\widehat{\M}_\C \colon \dAff_k^{op} \to \mathrm{Spc}
\]
\[
\Spec A \to \C_A^{\simeq}.
\]
	Since $\C$ is dualizable (\cref{finiteness_conditions}) and $\Mod(-)$ satisfies faithfully flat descent (\cite[Corollary D.6.3.3]{HA}), the presheaf $\widehat{\M}_\C $ defines a derived stack over $k$.
	The stack $\widehat{\M}_\C$ is too big and has no chance to be representable.
	The derived moduli of objects of $\C$ is defined as the presheaf 
\[
\M_\C: \dAff_k^{op} \to \mathrm{Spc}
\]
\[
\Spec A \to \C_A^{pp}
\]
where $\C_A^{pp} \subset \C_A$ is the maximal groupoid spanned by pseudo-perfect families over $A$ in the sense of \cref{pseudo_perfect}.
	Since $\Perf(-)$ satisfies hyper-descent (\cite[Proposition 2.8.4.2-(10)]{SAG}), the sub-presheaf
\[
\M_\C \subset \widehat{\M}_\C
\]
is a derived substack by \cref{functoriality_tens_prod}.
\end{recollection}

	The main theorem of \cite{TV} translates as:
\begin{theorem}
	Let $\C \in \Prlo_{k}$ be of finite type.
	Then $\M_\C$ is a locally geometric derived stack locally of finite presentation over $k$.
	Furthermore, the tangent complex at a point $x \colon  S \to \M_\C$ corresponding to a pseudo-perfect family of objects $M_x \in \C_\S$ is given by
\[
x^\ast\mathbb{T}_{\M_\C} \simeq \Hom(M_x, M_x)[1].
\]
\end{theorem}

	In this paper we are mainly interested in moduli of objects of abelian categories.
	If $\C \in \Prlo_k$ is equipped with a $t$-structure $\tau$, we can introduce a substack 
\[
\M_\C^{[0,0]} \subset \M_\C
\] 
whose classical truncation parametrizes families of objects in the heart $\tau$.
	Under suitable assumption on $\tau$, the stack $\M_\C^{[0,0]}$ is an open substack of $\M_\C$, so that it is also locally geometric locally of finite type when $\C$ is of finite type.

\begin{definition}\label{accessible_t_structure}
	Let $\C \in \Prlo_k$ equipped with a $t$-structure $\tau$. 
	We say that
\begin{enumerate}\itemsep=0.2cm
    \item $\tau$ is $\omega$-accessible if the full subcategory $\C_{\leq 0} \subseteq \C$ is stable under filtered colimits;
    \item $\tau$ is left-complete if $\bigcap \C_{\geq n} \simeq 0$;
    \item $\tau$ is right-complete if $\bigcap \C_{\leq n} \simeq 0$;
    \item $\tau$ is non-degenerate if it is both right-complete and left-complete.
    \item $\tau$ is admissible if it is $\omega$-accessible and non-degenerate.
\end{enumerate}
\end{definition}

\begin{recollection}\label{induced_t_structure}
	Let $\C \in \Prlo_k$.
	Then a morphism $f \colon \Spec (B) \to \Spec (A)$ in $\dAff_k$ corresponding to a map $A \to B$ induces an adjunction
\[
f_\C^\ast \colon \C_{A} \leftrightarrows \C_{B}\colon f_{\C, \ast}
\]
where $f_{\C, \ast}$ is the forgetful functor.
	Denote by $p \colon \Spec(A) \to \Spec(k)$ be the structural morphism.
	As explained in \cite[Section I.5.4]{HDR}, given an $\omega$-accessible right-complete $t$-structure $\tau$ on $\C$, we get an induced $\omega$-accessible right-complete $t$-structure $\tau_A$ on $\C_A$ for any $Spec (A) \in \dAff_k$ by declaring that $F \in \C_A$ is (co)connective if $p_{_\C, \ast} (F) \in \C$ is (co)connective.
	Moreover $\tau_A$ is non-degenerate if $\tau$ is.\\ \indent
	For $X\in \dSt_k$, we can also define an induced $t$-structure $\tau_X$ on $\C_X$ whose connective part is defined by descent.
	That is, $F \in (\C_X)_{\geq 0}$ if and only if $f^\ast F \in (\C_A)_{\geq 0}$ for any $f \colon \Spec (A) \to X$.
\end{recollection}

\begin{definition}[{\cite[Definition I.5.28]{HDR}}]\label{def_tau_flat}
	Let $\C \in \Prlo_k$ equipped with an $\omega$-admissible right-complete $t$-structure $\tau$ and $\Spec (A) \in \dAff_k$.
	We say that a family $F \in \C_A$ is $\tau$-flat over $\Spec (A)$ (or $\tau_A$-flat) if for every $M \in \Mod_A^\heartsuit$ one has $F \otimes_A M \in \C_A^\heartsuit$.
\end{definition}

\begin{remark}\label{tau_flat_in_heart}
	In the setting of \cref{def_tau_flat}, when $A$ is discrete, a $\tau$-flat family $F$ over $A$ lies in $\C_A^\heartsuit$.
	Indeed in this case one has $A \in \Mod_A^\heartsuit$, so that $F \simeq F \otimes_A A \in \C_A^\heartsuit$.
	The converse holds when $A$ is a field.
\end{remark}

We briefly recall why $\tau$-flatness behaves as in the geometric setting.

\begin{lemma}\label{exact_pullback}
	Let $\C \in \Prlo_k$ equipped with an $\omega$-accessible right-complete $t$-structure $\tau$ and let $f \colon \Spec(B) \to \Spec(A)$ be flat.
	Then the functor $f_\C^\ast \colon \C_A \to \C_B$ is $t$-exact.
\end{lemma}

\begin{proof}
	Let $F \in \C_A$.
	By definition of induced $t$-structure (\cref{induced_t_structure}), we have that $f^\ast F$ is (co)connective if and only if $f_{\C, \ast} f_\C^\ast F \simeq F \otimes_B A$ is (co)connective.
	By Lazard's Theorem \cite[Theorem 7.2.2.15]{HA}, we have that $F \otimes_A B$ is filtered colimit of direct sums of $F$.
	Since the $t$-structure is $\omega$-accessible, it then follows that $F \otimes_A B$ is (co)connective if $F$ is.
\end{proof}

\begin{lemma}\label{exact_affine_atlas}
	Let $\C \in \Prlo_k$ equipped with an admissible $t$-structure $\tau$.
	Let $p \colon V \to X$ be a flat effective epimorphism in $\dSt_k$ and assume that $X$ has affine diagonal and $V$ is a derived scheme.
	Then $f^\ast \colon \C_X \to \C_V$ is $t$-exact and conservative.
\end{lemma}

\begin{proof}
	Up to replacing $V$ with a Zariski open cover, we can assume $V = \Spec(A)$.
	Since $p$ is a flat effective epimorphism, we have that $X$ is equivalent to the  \v{C}ech nerve $\Spec(A^\bullet)$ of $p$.
	Hence we have an equivalence 
\[
\QCoh(X) \simeq \lim_{n \in \Delta_s} \Mod_{A^n}.
\]
	Since $\C \in \Prlo_k$ is dualizable by \cref{finiteness_conditions}, the functor
\[
\C \otimes_k - \colon \Prlo_k \to \Prlo_k
\]
commutes with limits. 
	Hence we have an equivalence
\[
\C_X \simeq \lim_{n \in \Delta_s} \C_{A^n}.
\]
	Under the above equivalence, the pullback $f_\C^\ast \colon \C_X \to \C_A$ is identified with the projection $\lim_{n \in \Delta_s} \C_{A^n} \to \C_A$.
	Moreover since $p$ is a flat epimorphism and $X$ has affine diagonal, all the transition functor are pullbacks along faithfully flat maps between derived affines.
	Hence we are reduced to the affine case. 
	Since $\tau$ is admissible, the statement follows by \cref{exact_pullback} and \cite[Proposition D.6.4.6]{SAG}.
\end{proof}

\begin{corollary}\label{exact_pullback_non_affine}
In the setting of \cref{exact_affine_atlas}, a family $F \in \C_X$ is (co)connective if and only if $f^\ast F \in \C_V$ is (co)connective.
\end{corollary}

\begin{proof}
	We prove the claim for connective families, the case of coconnective families being analogous.
	We have $F \in (\C_X)_{\geq 0}$ if and only if $\tau_{\leq -1} F \simeq 0$ in $\C_X$.
	Since $f^\ast$ is conservative and $t$-exact by \cref{exact_affine_atlas}, we have that $\tau_{\leq -1} F \simeq 0$ in $\C_X$ if and only if $\tau_{\leq -1} f^\ast F \simeq 0$ in $\C_V$.
	The latter condition is equivalent to $f^\ast F \in (\C_V)_{\geq 0}$.
\end{proof}

\begin{corollary}\label{exact_pullback_lemma}
	Let $\C \in \Prlo_k$ equipped with an admissible $t$-structure $\tau$.
	Consider a commutative square in $\dSt_k$
\[\begin{tikzcd}[sep=small]
	{U} && {V} \\
	\\
	Y && X
	\arrow["g", from=1-1, to=1-3]
	\arrow["q"', from=1-1, to=3-1]
	\arrow["p", from=1-3, to=3-3]
	\arrow["f"', from=3-1, to=3-3]
\end{tikzcd}\]
	where $X,Y$ have affine diagonal, $U,V$ are derived schemes, $f,g$ are flat and $p,q$ are flat effective epimorphisms.
	Then $f^\ast \colon \C_X \to \C_Y$ is $t$-exact.
\end{corollary}

\begin{proof}
	Let $F \in (\C_Y)_{\geq 0}$.
	By \cref{exact_pullback} and \cref{exact_pullback_non_affine}, we have $g^\ast p^\ast F \in (\C_U)_{\geq 0}$.
	Since $g^\ast p^\ast F \simeq q^\ast f^\ast F$, applying again  \cref{exact_pullback_non_affine} we get that $f^\ast F \in (\C_X)_{\leq 0}$.
	A similar proof shows that $f^\ast$ preserve coconnective families.
\end{proof}

\begin{lemma}\label{pullback_flat_affine}
	Let $\C \in \Prlo_k$ equipped with an $\omega$-admissible right-complete $t$-structure $\tau$.
	Let $f\colon \Spec (B) \to \Spec (A)$ in $\dAff_k$.
	If $F \in \C_A$ is $\tau_B$-flat, then $f^\ast F$ is $\tau_B$-flat.
\end{lemma}

\begin{proof}
	By \cite[Theorem 0.5]{Haine}, since in the affine case projection formula holds for the adjunction
\[
f^\ast \colon \Mod_A \leftrightarrows \Mod_B \colon f_{\ast},
\]
it also holds for the adjunction 
\[
f_\C^\ast \colon \C_A \leftrightarrows \C_B \colon f_{\C, \ast}.
\]
	Since (co)connectivity can be tested after applying $\Id_\C \otimes_k f_\ast$, the conclusion follows.
\end{proof}

\begin{remark}\label{tens_separately_colim}
	Let $A \in \dAff_k$ and $\E_1, \E_2 \in \Prlo_A$.
	The tensor product functor 
\[
- \otimes_A - \colon \E_1 \times \E_2 \to \E_1 \otimes_A \E_2
\]
preserves colimits separately in each variable.
	Indeed since $A$ is commutative, we have
\[
\Prlo_A \coloneqq \Mod_{\Mod_A}(\Prlo) \simeq {}_{\Mod_A}\mathrm{BiMod}_{\Mod_A}(\Prlo).
\]
	By definition of $- \otimes_A - $ via Bar construction (\cite[Section 4.4.2]{HA}), we have a factorization
\[
\E_1 \times \E_2 \xrightarrow{- \otimes -} \E_1 \otimes \E_2 \xrightarrow{p} \E_1 \otimes_A \E_2 .
\]
Since $p$ is in $\Pr^{\mathrm{L}}$ and the absolute tensor product $- \otimes -$ commutes with colimits in each variable by \cite[Remark 4.8.1.24]{HA}, the claim follows.
\end{remark}

\begin{lemma}\label{tens_commute_colim}
	Let $\C \in \Prlo_k$, $\Spec(A) \in \dAff_k$ and $F \in \C_A$.
	The functor
\[
	F \otimes_A - \colon \Mod_A \to \C_A
\]
commutes with colimits.
\end{lemma}

\begin{proof}
	It follows directly from \cref{tens_separately_colim} by taking $\E_1= \C_A$ and $\E_2 = \Mod_A$.
\end{proof}

\begin{corollary}\label{tens_preserve_coconnective}
	Let $\C \in \Prlo_k$ equipped with an $\omega$-accessible and right complete $t$-structure $\tau$.
	Let $A \in \dAff_k$, $M \in (\Mod_A)_{\geq 0}$ and $F \in (\C_A)_{\geq 0}$.
	Then $F \otimes_A M \in (\C_A)_{\geq 0}$.
\end{corollary}

\begin{proof}
	Let
\[
\D \coloneqq \left\lbrace N \in \Mod_A \ \big| \ F \otimes_A N \in (\C_A)_{\geq 0} \right\rbrace \subset \Mod_A.
\]
	Since $(\C_A)_{\geq 0}$ is stable under colimits in $\C_A$, it follows from \cref{tens_commute_colim} that the full subcategory $\D$ is stable under colimits in $\Mod_A$. 
	Since $F$ is connective, we have that $A \in \D$.
	By \cite[Proposition 7.1.1.13]{HA}, we deduce that $(\Mod_A)_{\geq 0} \subset \D$.
	Hence $F \otimes_A M$ is connective.
\end{proof}

\begin{lemma}\label{descent_flat_affine}
	Let $\C \in \Prlo_k$ equipped with an admissible $t$-structure $\tau$.
	Let $f \colon \Spec (B) \to \Spec (A)$ be a faithfully flat morphism.
	Then $F \in \C_A$ is $\tau_A$-flat if and only if $f_\C^\ast F$ is $\tau_B$-flat.
\end{lemma}

\begin{proof}
	The direct implication follows by \cref{pullback_flat_affine}. 
	Suppose now that $f_\C^\ast F$ is $\tau_B$-flat and let $M \in \Mod_A^\heartsuit$.
	By \cref{tens_preserve_coconnective} it is enough to prove that $\tau_{\geq 1} (F \otimes_A M) \simeq 0$.
	Since $f$ is faithfully flat and $\tau$ is admissible, the functor $f_\C^\ast \colon \C_A \to \C_B$ is conservative by \cite[Proposition D.6.4.6]{SAG}.
	Hence it is enough to prove that $f_\C^\ast (\tau_{\geq 1} (F \otimes_A M)) \simeq 0$.
	Since $f_\C^\ast \colon \C_A \to \C_B$ is $t$-exact by \cref{exact_pullback}, we get
\[
f_\C^\ast (\tau_{\geq 1}(F \otimes M)) \simeq \tau_{\geq 1}(f_\C^\ast (F \otimes_B M) )\simeq \tau_{\geq 1}(f_\C^\ast (F) \otimes_A f^\ast(M)).
\]
	Since $f \colon \Spec(A) \to \Spec(B)$ is flat, we have that $f^\ast(M) \in Mod_A^\heartsuit$.
Then by $\tau_A$-flatness of $f_\C^\ast(F)$ we get that $f_\C^\ast (F) \otimes_A f^\ast(M) \in \C_A^\heartsuit$. Hence $\tau_{\geq 1}(f_\C^\ast (F) \otimes_A f^\ast(M)) \simeq 0$.
\end{proof}

\begin{recollection}\label{descent_flat}
	Let $\C \in \Prlo_k$ equipped with an admissible $t$-structure $\tau$, and consider the sub-presheaf of $\M_\C$
\[
\M_\C^{[0,0]}: \dAff_k^{op} \to \mathrm{Spc}
\]
sending $S \in \dAff_k$ to the maximal $\infty$-groupoid of $\tau$-flat pseudo-perfect families of objects over $S$.
	By \cref{moduli_of_objects_recollection}, \cref{pullback_flat_affine} and \cref{descent_flat_affine} the presheaf $\M_\C^{[0,0]}$ is a derived substack of $\M_\C$.
\end{recollection}

\begin{notation}
	We will denote by 
\[
\M_\C^\heartsuit \colon \Aff_k^{op} \to \mathrm{Grpd}
\]
the truncation of the derived stack $\M_\C^{[0,0]}$.
\end{notation}

\begin{remark}
	The above notation is justified by the following: when testing on an underived affine, a $\tau$-flat family lies automatically in the heart (see \cref{tau_flat_in_heart}), while in general such family is only of Tor-amplitude $[0,0]$.
\end{remark}

	The following definition extends the notion of $\tau$-flat families for a general derived stack.
	
\begin{definition}
	Let $\C \in \Prlo_k$ equipped with an admissible $t$-structure $\tau$ and let $X \in \dSt_k$.
	A family $F \in \C_X$ is $\tau_X$-flat if for every $\Spec (B) \to X$, the pullback $f_\C^\ast F \in \C_B$ is $\tau_B$-flat.
\end{definition}

\begin{lemma}\label{stability_flatness}
	Let $\C \in \Prlo_k$ equipped with an admissible $t$-structure $\tau$.
	Let $f\colon X \to Y$ be a morphism in $\dSt_k$.
	If $F \in \C_Y$ is $\tau_Y$-flat, then $f_\C^\ast F$ is $\tau_X$-flat.
\end{lemma}

\begin{proof}
	We need to prove that for every $g: \Spec (B) \to X$, $g_\C^\ast (f_\C^\ast F)$ is $\tau_B$-flat.
	Since we have an equivalence $g_\C^\ast (f_\C^\ast F) \simeq (f\circ g)_\C^\ast F$, this holds since $F$ is $\tau_Y$-flat.
\end{proof}

\begin{lemma}\label{flat_affine_atlas}
	Let $\C \in \Prlo_k$ equipped with an admissible $t$-structure $\tau$.
	Let $f \colon Y \to X$ be a flat effective epimorphism, with $Y$ quasi-geometric in the sense of \cite[Definition 9.1.0.1]{SAG}.
	Then $F \in \C_X$ is $\tau_X$-flat if and only if $f_\C^\ast F$ is $\tau_Y$-flat.
\end{lemma}

\begin{proof}
	Let us notice that by definition of quasi-geometric stack there exists a flat epimorphism $p \colon \Spec(A) \to Y$.
	If we can prove the claim when the source is derived affine, then we deduce the general claim by applying it to $f \circ p \colon \Spec(A) \to X$ and to $p \colon \Spec(A) \to Y$.
	Hence we can assume $Y = \Spec(A)$.
	The direct implication follows from \cref{stability_flatness}.
	For the converse, suppose that $f^\ast F \in \C_A$ is $\tau_A$-flat. 
	We need to show that for every $g \colon \Spec (B) \to X$ we have that $g_\C^\ast F$ is $\tau_B$-flat.
	Since $f$ is an effective epimorphism, there exists an fppf cover $\pi \colon \Spec (B') \to \Spec (B)$ such that there exists a lift $\widetilde{g}: \Spec (B') \to X$ of $g$ as in the following diagram:
\[\begin{tikzcd}[sep=small]
	&&&& {\Spec (A)} \\
	\\
	{\Spec (B')} && {\Spec (B)} && X
	\arrow["f", two heads, from=1-5, to=3-5]
	\arrow["{\widetilde{g}}", dashed, from=3-1, to=1-5]
	\arrow["\pi"', from=3-1, to=3-3]
	\arrow["g"', from=3-3, to=3-5]
\end{tikzcd}\]
	By \cref{descent_flat_affine}, $g_\C^\ast F$ is $\tau_B$-flat if and only if $\pi_\C^\ast g_\C^\ast F \simeq \widetilde{g}_\C^\ast f_\C^\ast F$ is $\tau_{B'}$-flat.
	The result then follows from \cref{pullback_flat_affine}.
\end{proof}


	The key property for a $t$-structure to define an open substack of $\M_\C$ is linked to the following definition:

\begin{definition}
	Let $\C \in \Pr_k^{L, \omega}$ equipped with an admissible $t$-structure $\tau$.
	Let $\S \in \dAff_k$ and let $F \in \C_S$.
	The flat locus of $F$ is the functor 
\[
\Phi_F \colon  (\dAff_k)^{op}_{/_{S}} \to \Spc
\]
\begin{align*}
(T \xrightarrow{f} \S) \mapsto 
\left\{
    \begin{aligned}
        & \ast \ \ \text{if $f^\ast(F)$ is $\tau$-flat relative to $T$,} \\
        & \emptyset \ \ \ \ \text{otherwise.}                  
    \end{aligned}
\right.
\end{align*}
\end{definition}

\begin{definition}\label{opennes_flatness}
	Let $\C \in \Pr_k^{L, \omega}$ equipped with an admissible $t$-structure $\tau$.
	We say that $\tau$ universally satisfies openness of flatness if for every $S \in \dAff_k$ and every $F \in \C_S$ the flat locus $\Phi_F$ of $F$ is an open derived subscheme of $S$.
\end{definition}

\begin{proposition}[{\cite[Proposition I.5.35]{HDR}}]\label{moduli_flat_objects}
	Let $\C \in \Pr_k^{L, \omega}$ be of finite type equipped with an admissible $t$-structure $\tau$.
	If $\tau$ universally satisfies openness of flatness, then the structural morphism
\[
\M_\C^{[0,0]} \to \M_\C
\]
is representable by a Zariski open immersion.
	In particular, $\M_\C^{[0,0]}$ is a locally geometric derived stack, locally of finite presentation.
\end{proposition}

	A priori, \cref{moduli_flat_objects} only guarantees that $\M_\C^{[0,0]}$ and its truncation $\M_\C^\heartsuit$ have an open exhaustion by (derived) algebraic stacks.
	When $\Spec (A) \in \dAff_k$ is discrete, the $\infty$-grupoid $\M_\C^{[0,0]}(\Spec(A))$ takes values in $1$-grupoids by \cref{tau_flat_in_heart}.
	Hence $\M_C^\heartsuit$ is $1$-truncated in the sense of \cite{HAG-II}.
	Then \cite[Lemma 2.19]{TV} implies the following

\begin{corollary}\label{geometric_heart}
	Let $\C \in \Pr_k^{L, \omega}$ be of finite type equipped with an admissible $t$-structure $\tau$.
	If $\tau$ universally satisfies openness of flatness, then $\M_\C^{[0,0]}$ is a $1$-Artin stack.
	In particular its truncation $\M_\C^\heartsuit$ is an algebraic stack locally of finite presentation.
\end{corollary}

\begin{example}
	For an example satisfying universal openness of flatness let us refer to \cref{openness_flatness_perv}.
\end{example}

\subsection{An Hartogs-like lemma for flatness}\label{section_flatness}

	In this paragraph we reproduce a fundamental result of \cite{AHLH} in our setting, namely \cite[Lemma 7.15]{AHLH}.
	The proof will essentially be the same.
	This lemma is crucial in the construction of good moduli spaces for moduli of objects.

\begin{lemma}\label{tau_flat_vs_flat}
	Let $\C \in \Prlo_k$ equipped with an $\omega$-accessible right-complete $t$-structure $\tau$.
	Let $\Spec (A) \in \dAff_k$ and $F \in \C_A$ a $\tau_A$-flat family.
	Then the functor
\[
F \otimes_A -  \colon  \Mod_A \to \C_A
\]
is exact.
	In particular it restricts to an exact functor
\[
F \otimes_A -  \colon  \Mod_A^\heartsuit \to \C_A^\heartsuit .
\]
\end{lemma}

\begin{proof}
	Thus the functor $F \otimes_A -$ is commutes with (finite) colimits by \cref{tens_commute_colim}.
	Hence the first claim follows by \cite[Proposition 1.1.4.1]{HA}.
	The second one follows from the first as exact sequences in the heart of a $t$-structure are (by definition) fiber/cofiber sequences where all the terms are in the heart.
\end{proof}

\begin{remark}\label{flat_objects}
	Assume that $\Spec (A)$ is discrete.
	Differently to our convention, flat families over $\Spec (A)$ are defined in \cite{AHLH} as families $F \in \C_A^\heartsuit$ such that
\[
\pi_0(F \otimes_A -)  \colon  \Mod_A^\heartsuit \to \C_A^\heartsuit
\]
is an exact functor.
	\cref{tau_flat_vs_flat} shows that $\tau$-flat families in the sense of \cref{def_tau_flat} are flat in the sense of \cite{AHLH}. 
	The two notions agree when the $t$-structure is left-complete.
\end{remark}

\begin{proposition}[{\cite[Lemma 6.2.2]{HL}}]\label{tau_flat_vs_flat_bis}
	Let $\C \in \Prlo_k$ equipped with an admissible $t$-structure $\tau$.
	Let $\Spec (A) \in \dAff_k$ discrete. 
	For a family $F \in \C_A$ over $A$ the following are equivalent:
\begin{enumerate}\itemsep=0.2cm
    \item $F$ is $\tau_A$-flat;
    \item $F \in \C_A^\heartsuit$ and $\pi_0(F \otimes_A -) \colon  \Mod_A^\heartsuit \to  \C_A^\heartsuit$ is an exact functor.
\end{enumerate}
\end{proposition}

\begin{proof}
	That (1) implies (2) follows directly from \cref{tau_flat_in_heart} and \cref{tau_flat_vs_flat}.
	Suppose now that 
\[
\pi_0(F \otimes_A -) \colon  \Mod_A^\heartsuit \to  \C_A^\heartsuit
\]
is exact and consider $M \in \Mod_A^\heartsuit$.
	We need to show that $F \otimes_A M \in \C_A^\heartsuit$.
	By \cref{tens_preserve_coconnective}, it is enough to show that $\tau_{\geq 1} (F \otimes_A M) \simeq 0$.
	A choice of generators of $M$ gives a presentation
\begin{equation}\label{presentation_s_e_s}
0 \to K \to A^S \to M \to 0.
\end{equation}
By \cref{tau_flat_vs_flat}, applying $F \otimes_A -$ to \eqref{presentation_s_e_s} yields a fiber/cofiber sequence
\[
F \otimes_A K \to F \otimes_A A^S \to F \otimes_A M .
\]
The associated long exact sequence of homotopy groups yields an exact sequence
\[
\begin{tikzcd}[sep=small]
	\cdots & \pi_1(F\otimes_A A^S) & {\pi_1(F \otimes_A M)} & { \pi_0(F \otimes_A K)} & { \pi_0(F \otimes_A A^S)} & { \pi_0(F \otimes_A M)}
	\arrow[from=1-1, to=1-2]
	\arrow[from=1-2, to=1-3]
	\arrow[from=1-3, to=1-4]
	\arrow[from=1-4, to=1-5]
	\arrow[from=1-5, to=1-6]
\end{tikzcd}
\]
	By hypothesis, the map 
\[
\pi_1(F\otimes_A A^S) \to \pi_1(F \otimes_A M)
\]
is an epimorphism. 
	Moreover since $F \in \C_A^\heartsuit$ and $\C_A^\heartsuit$ is closed under arbitrary direct sums, we have 
\[
F\otimes_A A^S \simeq \bigoplus_S F \in \C_A^\heartsuit.
\]
	Thus we get $ \pi_{i}(F \otimes_A A^S) =0$ for any $i>0$.
	Hence we get $\pi_1(F \otimes_A M) = 0$ and
\[
\pi_{i+1}(F \otimes_A M) \xlongrightarrow{\sim}  \pi_{i}(F \otimes_A K)
\]
for any $i>0$.
	This two facts allow to conclude by induction that $ \pi_{i}(F \otimes_A M)= 0$ for any $i>0$.
	Indeed, we can apply recursively the same argument where we replace $M$ by $K$ until we get $i=1$.
	Since $\tau_A$ is left-complete, we get
\[
\tau_{\geq 1}(F\otimes_A M) \in \bigcap_{n \geq 1} (\C_A)_{\geq n} \simeq 0.
\]
\end{proof}

The following Lemmas are analogues of \cite[Lemma C1.12 \& Proposition C1.13]{AZ} to our setup.

\begin{lemma}[{local criterion for flatness}]\label{local_criterion_flatness}
	Let $\C \in \Prlo_k$ equipped with an $\omega$-accessible right-complete $t$-structure $\tau$ and let $R$ be a discrete noetherian commutative ring.
	For $F \in \C_R$, the following statements are equivalent:
\begin{enumerate}
	\item $F \in \C_R$ is $\tau_R$-flat;
	\item for every prime ideal $\mathfrak{p} \subset R$ we have $F \otimes_R R/\mathfrak{p} \in \C_R^\heartsuit$;
\end{enumerate} 
\end{lemma}

\begin{proof}
	That (1) implies (2) follows by definition of $\tau$-flatness.
	Assume now that (2) holds.
	Let $M \in \Mod_R^\heartsuit$.
	We need to prove that $F \otimes_R M \in \C_R^\heartsuit$.
	Write $M$ as a filtered colimit of its finitely generated submodule.
	Since $\tau$ is $\omega$-accessible and tensor product commutes with filtered colimits by \cref{tens_commute_colim}, we can assume that $M$ is finitely generated.
	By \cite[\href{https://stacks.math.columbia.edu/tag/00L0}{Lemma 00L0}]{stacks-project}, there exists a filtration
	\[
	0 = M_0 \subset M_1 \subset \cdots \subset M_n = M
	\]
such that $M_i / M_{i-1} \simeq R / \mathfrak{p}_i$ for some prime ideal $\mathfrak{p}_i \subset R$.
	We prove that $F\otimes M_i \in \C_R^\heartsuit$ by induction on $0 \leq i \leq n$.
	For $i=0$ is obvious.
	Let now $0 \leq i < n$ and assume that $F \otimes_R M_i \in \C_R^\heartsuit$.
	By \cref{tau_flat_vs_flat}, we have a fiber/cofiber sequence
	\[
	F \otimes_R M_i \to F \otimes_R M_{i+1} \to F \otimes_R R / \mathfrak{p}_{i+1}.
	\]
	The first and third term belong to $\C_R^\heartsuit$ by hypothesis.
	Since $\C_R^\heartsuit$ is closed under extensions, we get that $F \otimes_R M_{i+1} \in \C_R^\heartsuit$.
\end{proof}

\begin{corollary}\label{flatness_DVR}
	Let $\C \in \Prlo_k$ equipped with an $\omega$-accessible right-complete $t$-structure $\tau$ and let $R$ be a DVR over $k$ with uniformizer $\pi$.
	Then $F \in \C_R$ is $\tau_R$-flat if and only if $F \in \C_R^\heartsuit$ and the multiplication by $F \xrightarrow{\pi} F$ is injective.
\end{corollary}

\begin{proof}
	The only prime ideals are  $(0), (\pi) \subset R$.
	We have that $F \in \C_R^\heartsuit$ if and only if $F \otimes_R R/(0) \in \C_R^\heartsuit$.
	Let $\kappa \coloneqq R / (\pi)$.
	By \cref{local_criterion_flatness}, it is enough to prove that for $F \in \C_R^\heartsuit$ we have $F \otimes_R \kappa \in \C_R^\heartsuit$ if and only if $F \xrightarrow{\pi} F$ is injective.
	By \cref{tau_flat_vs_flat}, the short exact sequence 
\[
0 \to R \xrightarrow{\pi} R \to \kappa \to 0
\]
induces a fiber cofiber sequence
\[
F \xrightarrow{\pi} F \to F \otimes_R \kappa .
\]
	Taking the associated long exact sequence of homotopy groups yields an exact sequence
\[
0 \to \pi_1(F \otimes_s \kappa) \to F \xrightarrow{\pi} F \to \pi_0( F \otimes_R \kappa) \to 0
\]
and $\pi_i(F \otimes_R \kappa) \simeq 0$ for every $i \neq 0,1$.
	Hence we get that $F \otimes_R \kappa \in \C_R^\heartsuit$ if and only if $F \xrightarrow{\pi} F$ is injective.
\end{proof}

	In the next paragraph we give a more explicit description of pseudo-perfect $\tau$-flat families in a case we are particularly interested into.

\begin{recollection}\label{descent_data}
	Let $\C \in \Prlo_k$ equipped with an admissible $t$-structure $\tau$.
	Let $X = \Spec (R)  \in \dAff_k$ be discrete local regular noetherian of dimension 2.
	Let $j \colon U \to X$ be the complement of the unique closed point and let $x, y \in R$ be a regular sequence.
	By \cite[Lemma 2.11]{DAG_descent}, a family $F \in \C_U$ corresponds to the following data:
\begin{itemize}\itemsep=0.2cm
    \item[(i)] Families
    \[
    F_x \in \C_{R_x},\ \ F_y \in \C_{R_x}, \ \ F_{xy} \in \C_{R_{xy}};
    \]
    \item[(ii)] descent data 
    \[
    F_{y}\otimes_{R_y}R_{xy} \simeq F_{xy} \simeq F_x \otimes_{R_x}R_{xy}.
    \]
\end{itemize}
	By \cref{flat_affine_atlas}, a $\tau$-flat family on $U$ is given by (i) and (ii) as above, with the requirement that $F_x$, $F_y$ and $F_{xy}$ are $\tau$-flat.
\end{recollection}

\begin{recollection}\label{functor_to_limit}
	Let $A \to \Prlo_k$, $\alpha \mapsto \D_\alpha$, be a diagram in $\Prlo_k$, and let $\D$ be its limit.
	Let $f \colon \C \to \D$ in $\Prlo_k$ and put $f_\alpha \coloneqq p_\alpha \circ f$ for every $\alpha \in A$, where $p_\alpha \colon \D \to \D_\alpha$ is the projection.
	For $\alpha \in A$, let $g_\alpha \colon D_\alpha \to \C$ be the right adjoint to $f_\alpha$.
	One immediately checks that the right adjoint $g \colon \D \to \C$ to $f \colon \C \to \D$ is computed by
\[
g \colon \D \to \C
\]
\[
d \mapsto \lim_\alpha g_\alpha \circ p_\alpha (d)
\]
\end{recollection}

\begin{lemma}\label{formula_push_derived}
	In the setting of \cref{descent_data}, let $F \in \C_U$ be a family over $U$.
	We have
\[
j_{\ast} F \simeq \fib(F_x \oplus F_y \to F_{xy})
\]
where $F_x$, $F_y$ and $F_{xy}$ are considered as families in $\C_R$ via the respective forgetful functors.
\end{lemma}

\begin{proof}
	Since $\C \in \Prlo_k$ is dualizable by \cref{finiteness_conditions}, the functor
\[
\C \otimes_k - \colon \Prlo_k \to \Prlo_k
\]
commutes with limits. Thus $\C_U$ fits in a pullback square in $\Prlo_k$
\[\begin{tikzcd}[sep=small]
	{\C_{U}} && {\C_{R_{y}}} \\
	\\
	{\C_{R_{x}}} && {\C_{R_{xy}} \ .}
	\arrow[from=1-1, to=1-3]
	\arrow[from=1-1, to=3-1]
	\arrow[from=1-3, to=3-3]
	\arrow[from=3-1, to=3-3]
\end{tikzcd}\]
	By \cref{functor_to_limit}, the family $j_{\C, \ast}F$ fits in a pullback square in $\C_R$
\[\begin{tikzcd}[sep=small]
	{j_{\ast}F} && {F_{y}} \\
	\\
	{F_{x}} && {F_{xy} \ .}
	\arrow[from=1-1, to=1-3]
	\arrow[from=1-1, to=3-1]
	\arrow["q", from=1-3, to=3-3]
	\arrow["p"', from=3-1, to=3-3]
\end{tikzcd}\]
Since $\C_R$ is stable, we can rewrite the above as
\[
j_{\C, \ast}F \simeq \fib(F_x \oplus F_y \to F_{xy}).
\]
\end{proof}

\begin{corollary}\label{formula_push}
	In the setting of \cref{descent_data}, suppose that $F \in \C_U$ is $\tau$-flat over $U$. Then
\[
\pi_0 j_{\ast} F \simeq \mathrm{Ker}(F_x \oplus F_y \to F_{xy})
\]
where $F_x$, $F_y$ and $F_{xy}$ are considered as families in $\C_R^\heartsuit$ via the respective forgetful functors.
\end{corollary}

\begin{proof}
	By $\tau_U$-flatness of $F$ and exactness of forgetful functors, we have that $F_x, F_y, F_{xy} \in \C_R^\heartsuit$.
	The statement then follows by \cref{formula_push_derived}.
\end{proof}

\begin{remark}
	Let $\C \in \Prlo_k$ equipped with an admissible $t$-structure $\tau$ and let $f \colon \X \to \Y$ in $\dSt_k$.
	By definition of $\tau_\X$ and $\tau_\Y$, the functor $f_\C^\ast \colon \C_\Y \to \C_\X$ preserves connective objects.
	In particular, the adjunction $f_\C^\ast \dashv (f_\C)_\ast$ restricts to an adjunction
\[
\pi_0 f_\C^\ast \colon \C_Y^\heartsuit \leftrightarrows \C_X^\heartsuit \colon \pi_0 (f_\C)_\ast 
\]
\end{remark}

\begin{lemma}\label{push_fully_faith_derived}
	In the setting of \cref{descent_data}, the functor $j_{\C, \ast} \colon C_U \to \C_X$ is fully faithful.
\end{lemma}

\begin{proof}
	We need to prove that the counit $j_\C^\ast j_{\C, \ast} \to \Id_{\C_U}$ is an equivalence.
	Given $F \in \C_U$, let $F_x, F_y, F_{xy}$ as in \cref{descent_data}.
	By \cref{formula_push_derived}, we have $j_{\C, \ast} F \simeq \fib(F_x \oplus F_y \to F_{xy})$.
	Thus restricting to $R_x$ (respectively, $R_y$, $R_{xy}$), we get back $F_x$ (respectively, $F_{y}$, $F_{xy}$).
	Hence $j_\C^\ast j_{\C, \ast} F \simeq F$.
\end{proof}

\begin{corollary}\label{push_fully_faith}
	In the setting of \cref{descent_data}, the functor $\pi_0 j_{\C, \ast}\colon C_U^\heartsuit \to \C_X^\heartsuit$ is fully faithful.
\end{corollary}

\begin{proof}
	We need to prove that the counit $\pi_0 j_\C^\ast \pi_0 j_{\C, \ast} \to \Id_{\C_U^\heartsuit}$ is an equivalence.
	Since $j_\C^\ast$ is $t$-exact by \cref{exact_pullback_lemma}, we have
\[
\pi_0 j_\C^\ast \pi_0 j_{\C, \ast} \simeq j_\C^\ast \pi_0 j_{\C, \ast} \simeq \pi_0 (j_\C^\ast j_{\C, \ast}).
\]
	The conclusion follows by \cref{push_fully_faith_derived}.
\end{proof}

\begin{lemma}\label{push_preserve_flat}
	In the setting of \cref{descent_data}, the functor $\pi_0 j_{\C, \ast} \colon C_U^\heartsuit \to \C_X^\heartsuit$ preserve $\tau$-flat objects.
\end{lemma}

\begin{proof}
	Let $F \in \C_U^\heartsuit$ be $\tau$-flat over $U$.
	By \cref{local_criterion_flatness}, we need to prove that for any $\mathfrak{q} \in \Spec(R)$, we have $\pi_0j_{\C, \ast} F \otimes_k R /\mathfrak{q} \in \C_R^\heartsuit$.
	Since $j_\C^\ast$ is $t$-exact by \cref{exact_pullback_lemma}, we have $j_\C^\ast \pi_0j_{\C, \ast} F \simeq \pi_0 F \simeq F$.
	Hence the statement is clear for $\mathfrak{q} \in U$ and we are reduced to prove that $\pi_0j_{\C, \ast} F \otimes_k \kappa \in \C_R^\heartsuit$, where $\kappa\coloneqq R / (x,y)$.
	Consider the free resolution
\[
0 \to R \to R \oplus R \to R \to \kappa \to 0.
\]
	It is enough to show that tensoring this with $j_{\C, \ast} F$ yields an exact sequence 
\begin{equation}\label{exact_sequence}
    0 \to \pi_0 j_{\C, \ast} F \to \pi_0 j_{\C, \ast} F \oplus \pi_0 j_{\C, \ast} F \to \pi_0 j_{\C, \ast} F .
\end{equation}
	Consider the exact sequence
\[
0 \to \OO_U \to \OO_U \oplus \OO_U \to \OO_U \to 0.
\]
	Tensoring with $F$ yields an exact sequence over $U$ by \cref{tau_flat_vs_flat_bis}.
	Applying $\pi_0 j_{\C, \ast}$ one gets the sequence in \eqref{exact_sequence}, which is exact since $\pi_0 j_{\C, \ast}$ is exact on the left (because $j_{\C, \ast}$ preserve coconnective objects).
\end{proof}

	We are now ready to prove that \cite[Lemma 7.15]{AHLH} carries on in our context with the same proof.

\begin{lemma}\label{saved_lemma}
	Let $\C \in \Prlo_k$ equipped with an admissible $t$-structure $\tau$.
	Let $X$ be a regular noetherian scheme of dimension 2 over $k$.
	Let $j \colon U \to X$ be complement of a dimension $0$ closed subscheme of $X$.
	Then the adjunction
\[
\pi_0 j_\C^\ast \colon \C_X^\heartsuit \leftrightarrows \C_U^\heartsuit \colon \pi_0 j_{\C, \ast}
\]
induces an equivalence between the full subcategories of $\tau$-flat families over $U$ and $\tau$-flat families over $X$.
\end{lemma}

\begin{proof}
	By descent we reduce to the local situation of \cref{descent_data}.
	By \cref{push_fully_faith} and \cref{push_preserve_flat}, it is enough to prove that for $F \in \C_R$ a $\tau$-flat family over $X = \Spec(R)$, the unit morphism $F \to \pi_0 j_{\C, \ast} \pi_0 j^\ast F$ is an equivalence.
	Localizing $F$ at $x,y$ and $xy$, we get a $\tau$-flat family $j^\ast F$ over $U$.
	By \cref{tau_flat_vs_flat}, tensoring the exact sequence in $Mod_R^\heartsuit$
\[
0 \to R \to R_x \oplus R_y \to R_{xy}
\]
with $F$ yields an exact sequence in $\C_R^\heartsuit$
\[
0 \to F \to F_x \oplus F_y \to F_{xy}.
\]
	Hence $F \simeq \pi_0j_{\C, \ast}j_\C^\ast F \simeq \pi_0 j_{\C, \ast} \pi_0 j_\C^\ast F$, where the second equivalence follows by $t$-exactness of $j_\C^\ast \colon \C_X \to \C_U$ (\cref{exact_pullback_lemma}).
\end{proof}

\begin{corollary}\label{saved_lemma_bis}
	In the setting of \ref{saved_lemma}, let $G$ be a smooth affine algebraic group over $k$ equipped with an action on $X$ which restricts to an action on $U$.
	Set $\X \coloneqq [X/G]$, $\mathfrak{U} \coloneqq [U/G]$ and $\overline{j} \colon \mathfrak{U} \to \X$ the open immersion.
	Then the adjunction
\[
\pi_0 \overline{j}_\C^\ast \colon \C_\X^\heartsuit \leftrightarrows \C_\mathfrak{U}^\heartsuit \colon \pi_0 \overline{j}_{\C, \ast}
\]
induces an equivalence between the full subcategories $\tau$-flat families over $\mathfrak{U}$ and $\tau$-flat families over $\X$.
\end{corollary}

\begin{proof}
	Let $p\colon U \to \mathfrak{U}$ and $q\colon X \to X$ be the atlas map. Then we have a pullback diagram
\[\begin{tikzcd}[sep=small]
	U && X \\
	\\
	{\mathfrak{U}} && \X
	\arrow["j", from=1-1, to=1-3]
	\arrow["p"', from=1-1, to=3-1]
	\arrow["q", from=1-3, to=3-3]
	\arrow["{\overline{j}}"', two heads, from=3-1, to=3-3]
\end{tikzcd}\]
	By \cite[Theorem 0.5]{Haine} (see also\cref{functoriality_tens_prod}), flat base change for quasi-coherent sheaves implies flat base change in our setting by applying $\Id_\C \otimes_k -$, that is, we have an equivalence of functors
\[
 q_\C^\ast \overline{j}_{\C, \ast} \simeq j_{\C, \ast} p_\C^\ast \colon \C_{\mathfrak{U}} \to \C_X.
\]
	By \cref{exact_affine_atlas} and \cref{exact_pullback_lemma}, the functors $\overline{j}_\C^\ast, j_\C^\ast, p_\C^\ast, q_\C^\ast$ are $t$-exact and $p_\C^\ast, q_\C^\ast$ are conservative.
	It then follows formally from \cref{saved_lemma} that the unit and counit of the adjunction
\[
\pi_0 \overline{j}_\C^\ast \colon \C_\X^\heartsuit \leftrightarrows \C_\mathfrak{U}^\heartsuit \colon \pi_0 \overline{j}_{\C, \ast}
\]
are equivalent to the identity functors.
	We are left to show that $\pi_0 \overline{j}_\C$ sends $\tau$-flat families over $\mathfrak{U}$ to $\tau$-flat families over $\X$.
	By \cref{stability_flatness}, the family $p^\ast F$ is $\tau_U$-flat.
	Hence it follows by \cref{saved_lemma} that $\pi_0 j_{\C, \ast} p^\ast F \simeq \pi_0 q_\C^\ast\overline{j}_{\C, \ast} F$ is $\tau_X$-flat.
	Since $q^\ast$ is exact, we have that $\pi_0 q_\C^\ast\overline{j}_{\C, \ast} F \simeq q_\C^\ast (\pi_0\overline{j}_{\C, \ast} F)$.
	The result then follows by applying \cref{pullback_flat_affine}.
\end{proof}

\subsection{Existence criteria for good moduli spaces} \label{existence_criteria}
	In this paragraph we recall the definition of good moduli spaces and an existence criterion for them.
	References are given to \cite{Alp, AHLH}.

\begin{definition}[{\cite[Definition 4.1]{Alp}}]
	Let $q\colon \X \to X$ be a qcqs morphism over an algebraic space $S$ with $\X$ an algebraic stack and $X$ an algebraic space. 
	We say that $q \colon \X \to X$ is a good moduli space if the following properties are satisfied:
\begin{enumerate}\itemsep=0.2cm
    \item The functor $q_{\ast} \colon \QCoh(\X) \to \QCoh(X)$ is $t$-exact;
    \item we have $q_{\ast}\OO_\X \simeq \OO_X$.
\end{enumerate}
\end{definition}

	For completeness, we recall some of the main properties of good moduli spaces:

\begin{theorem}[{\cite[Proposition 4.5 \& Theorem 4.16 \& Theorem 6.6]{Alp}}]\label{good_moduli_properties}
	Let $q \colon \X \to X$ be a good moduli space.
	Then
\begin{enumerate}\itemsep=0.2cm
    \item The functor $q^\ast\colon \QCoh(X)^\heartsuit \to \QCoh(\X)^\heartsuit$ is fully faithful;
    \item the map $q$ is universally closed and surjective;
    \item if $Z_1,Z_2$ are closed substacks of $\X$, then
    \[
    \mathrm{Im}(Z_1) \cap \mathrm{Im}(Z_2) = \mathrm{Im}(Z_1 \cap Z_2)
    \]
    where the intersections and images are scheme-theoretic.
    \item for $\kappa$ an algebraically closed field over $k$, $\kappa$-points of $X$ are $\kappa$-points of $\X$ up to closure equivalence (see \cite[Theorem 4.16-(iv)]{Alp} for a precise statement);
    \item If $\X$ is reduced (resp., quasi-
compact, connected, irreducible), then $X$ is also.
	If $\X$ is locally noetherian
and normal, then $X$ is also.
    \item if $\X$ is locally noetherian, then $X$ is locally noetherian and $q_{\ast}\colon \QCoh(\X)^\heartsuit \to \QCoh(X)^\heartsuit$ preserves coherence;
    \item if $k$ is excellent and $\X$ is of finite type, then $X$ is of finite type;
    \item if $\X$ is locally noetherian, then $q \colon \X \to X$ is universal for maps to algebraic spaces.
\end{enumerate}
\end{theorem}

	The necessary and sufficient conditions of \cite{AHLH} for the existence of a good moduli spaces are given in terms of valuative criteria for two stacks that we now introduce.
	In what follows we fix $R$ a DVR over $k$.
	Let $\pi \in R$ be a uniformizer and $\kappa$ the residue field of $R$.
	The first stack we have to consider is
\[
\Theta_R\coloneqq [\Spec(R[x]) / \G_{m, R}],
\]
which has an unique closed point, denoted by $0$.
	Explicitly, 
\[
0 \coloneqq [\Spec(\kappa) / \G_{m,R}].
\]

\begin{definition}\label{def_Theta_reductive}
	We say that $\X \in \mathrm{St}_k$ is $\Theta$-reductive if for any DVR $R$ and any diagram of solid arrows
\[\begin{tikzcd}[sep=small]
	{\Theta_R \setminus \left\lbrace 0 \right\rbrace} && \X \\
	\\ 
	{\Theta_R}
	\arrow["f", from=1-1, to=1-3]
	\arrow["j"', hook, from=1-1, to=3-1]
	\arrow["{\exists ! \ \widetilde{f}}"', dashed, from=3-1, to=1-3]
\end{tikzcd}\]
there exists a unique extension $\widetilde{f}$ of $f$ making the above triangle commutes.
\end{definition}

	The second stack we have to consider is \[
\ST\coloneqq \left[\faktor{\Spec(R[s,t] / (st-\pi))}{\G_{m, R}}\right],
\]
which has a unique closed point, denoted by $0$.
	Explicitly,
\[
0 \coloneqq [\Spec(\kappa) / \G_{m,R}].
\]

\begin{definition}\label{def_S_complete}
	We say that $\X \in \mathrm{St}_k$ is $\mathrm{S}$-complete if for any DVR $R$ and any diagram of solid arrows
\[\begin{tikzcd}[sep=small]
	{\ST\setminus \left\lbrace 0 \right\rbrace} && \X \\
	\\
	\ST
	\arrow["f", from=1-1, to=1-3]
	\arrow["j"', hook, from=1-1, to=3-1]
	\arrow["{\exists ! \ \widetilde{f}}"', dashed, from=3-1, to=1-3]
\end{tikzcd}\]
there exists a unique extension $\widetilde{f}$ of $f$ making the above triangle commute.
\end{definition}

	Let us now state the main theorem of \cite{AHLH}:

\begin{theorem}[{\cite[Theorem A]{AHLH}}]\label{existence_good_separated}
	Let $S$ be an algebraic space of characteristic $0$.
	Let $\X$ be an algebraic stack of finite presentation with affine stabilizers and separated diagonal over $S$.
	Then $\X$ admits a separated good moduli space $X$ if and only if $\X$ is $\Theta$-reductive and $\mathrm{S}$-complete.
\end{theorem}

	Moduli of $\tau$-flat families have affine diagonal, as shown by the following:

\begin{lemma}[{\cite[Lemma 7.20]{AHLH}}]\label{affine_diagonal}
	Let $\C \in \Prlo_k$ be of finite type equipped with be an admissible $t$-structure $\tau$ universally satisfying openness of flatness.
	Then $\M_\C^\heartsuit$ has affine diagonal.
\end{lemma}

\begin{proof}
	By \cref{geometric_heart} we have that $\M_\C^\heartsuit$ is an algebraic stack locally of finite presentation.
	Then the proof of \cite[Lemma 7.20]{AHLH} carries on in our setup.
\end{proof}

\begin{remark}
	In our setup having affine diagonal (thus affine stabilizers) and being locally of finite presentation come for free (see \cref{affine_diagonal} and \cref{geometric_heart}), so that in \cref{existence_good_separated} one only needs to check quasi-compactness, $\Theta$-reductiveness and $\mathrm{S}$-completeness.
	The stack $\M_\C^{[0,0]}$ will essentially never be quasi-compact, but $\Theta$-reductiveness and $\mathrm{S}$-completeness are stable under taking closed substacks.
\end{remark}

	Finally, let us recall that a notion of derived good moduli spaces has been introduced in \cite{derived_good}.
	The following is the main result of that paper:
\begin{theorem}[{\cite[Theorem 2.12]{derived_good}}]\label{derived_good}
	Let $\X$ be a geometric derived stack.
\begin{enumerate}\itemsep=0.2cm
    \item If $t_0\X$ admits a good moduli space $t_0q \colon  t_0\X \to t_0X$, then $\X$ admits a derived good moduli space $q' \colon  \X \to X$ such that $t_0q' \simeq t_0q$.
    \item If $\X$ admits a derived good moduli space $q \colon  \X \to X$, then $t_0q$ is a good moduli space for $t_0X$.
\end{enumerate}
\end{theorem}

\section{Good moduli spaces for $\M_\C^\heartsuit$}\label{general_case}

	In this section we prove the existence of good moduli spaces for $\M_\C^\heartsuit$. In particular, we show that $\M_\C^\heartsuit$ is $\Theta$-reductive and $\mathrm{S}$-complete under the following

\begin{assumption}\label{Assumption_subobject}
	Let $\C \in Pr^{L, \omega}_k$ smooth equipped with an admissible $t$-structure.
	For any DVR $R$ over $k$, the abelian category $\C_R^\heartsuit$ has the following property: any subobject (equivalently, quotient) of a pseudo-perfect object is also pseudo-perfect.
\end{assumption}

\begin{remark}
	For what follows, it is enough to ask that \cref{Assumption_subobject} holds for any subobject of a $\tau_R$-flat object.
	Moreover in \cite{AHLH} it is shown that it is enough to ask that \cref{Assumption_subobject} holds for any DVR that is essentially of finite type.
\end{remark}

	Let us mention some cases in which the Assumption is satisfied.

\begin{example}
	Standard $t$-structures on functor categories (e.g., quiver representations).
	This is nothing but the fact that a submodule of a finitely generated $R$-module is still finitely generated when $R$ is noetherian.
\end{example}

\begin{example}
	Perverse sheaves on nice enough stratified spaces.
	See \cref{Perv_subobjects}.
\end{example}

\begin{example}
	If $\C\in \Prlo_k$ is smooth and proper, $\C^\heartsuit$ is locally noetherian and $R$ is essentially of finite type. This follows by \cref{pseudo_perfect_vs_compact} and the work of Artin-Zhang \cite{AZ}.
\end{example}

\begin{example}
	Quasi-coherent sheaves on a quasi-compact separated algebraic space locally of finite presentation. This follows by \cite[Theorem 3.0.2]{BZNP}.
\end{example}

\subsection{$\Theta$-reductiveness}\label{section_Theta_reductive}

	We will prove in this paragraph that $\M_\C^\heartsuit$ is $\Theta$-complete.
	In order to proceed with this verification, we will characterize explicitly $\tau_{\Theta_R}$-flat pseudo-perfect families in $\C_{\Theta_R}$ in \cref{flat_perf_Theta}.
\begin{notation}\label{notation_Rep}
	For $\E$ a stable $\infty$-category, put
\[
\Rep(\ZZ; \E) \coloneqq \Fun(\ZZ^{op}, \E) \ .
\]
	Recall from \cref{QCoh_Theta} that
\[
\QCoh(\Theta_R) \simeq \Rep(\ZZ, \Mod_R).
\]
	By \cite[Corollary 3.22]{BZFN} we have
\[
\Perf(\Theta_R) = \QCoh(\Theta_R)^{\omega} \simeq \Rep(\ZZ, \Mod_R)^{\omega} \ .
\]
	Consider the inclusion $i_n \colon \ZZ_{\leq n} \hookrightarrow \ZZ$.
	The pullback $i_n^\ast$ along $i_n$ has a left adjoint $i_{n, !}$ given by left Kan extension.
	Put
\[
(-)_{\leq n}\coloneqq i_{n, !}i_n^\ast \colon \Rep(\ZZ; \E) \to \Rep(\ZZ; \E) \ .
\]
	Explicitly, on an object 
\[
E_\bullet \colon  \ \cdots \xrightarrow{f_{n+1}} E_n \xrightarrow{f_n} E_{n-1} \xrightarrow{f_{n-1}} \cdots
\]
we have
\[
(E_\bullet)_{\leq n} \colon  \ \cdots \rightarrow 0 \rightarrow E_n \xrightarrow{f_n} E_{n-1} \xrightarrow{f_{n-1}} \cdots
\]
\end{notation}

\begin{lemma}\label{filtered_compact_bis}
	Let $\E$ be a stable $\infty$-category.
	Suppose we are given an object 
\[
E_\bullet \colon  \ \cdots \xrightarrow{f_{n+1}} E_n \xrightarrow{f_n} E_{n-1} \xrightarrow{f_{n-1}} \cdots
\]
in $\Rep(\ZZ; \E^{})$.
	Then the canonical comparison
\begin{equation}\label{eq_filtered_compact_bis}
\colim_{n \in \N} (E_\bullet)_{\leq n} \to E_\bullet
\end{equation}
is an equivalence.
\end{lemma}

\begin{proof}
	It is enough to show that $\eqref{eq_filtered_compact_bis}$ is an equivalence after evaluating on each $m \in \ZZ$. But for $n>m$ we have that $(E_\bullet)_{\leq n}(m) = E_m$, so that in the colimit the evaluation at $m$ stabilizes to $E_m$ with transition maps given by $\Id_{E_m}$.
\end{proof}

\begin{corollary}\label{filtered_compact}
	In the setting of \cref{filtered_compact_bis}, if $E_\bullet$ is compact we have $E_n \simeq 0$ for $n\gg 0$.
\end{corollary}

\begin{proof}
	Since $E_\bullet$ is compact, it follows from \cref{filtered_compact_bis} that
\[
\Id_{E_\bullet} \in \Map(E_\bullet, E_\bullet) \simeq \Map(E_\bullet, \colim_{n \in \N} (E_\bullet)_{\leq n}) \simeq \colim_{n \in \N} \Map(E_\bullet, (E_\bullet)_{\leq n}).
\]
	Thus $\Id_{E_\bullet}$ factors through $(E_\bullet)_{\leq k}$ for some $k \in \ZZ$. 
	Hence $E_n \simeq 0$ for every $k>n$.
\end{proof}

\begin{recollection}
	Let us choose a coordinate $x$ on $\A^1_R$.
	Recall the equivalence
\[
\QCoh(\Theta_R) \simeq \Mod_{R[x]}(\Rep(\ZZ^{ds})) \simeq \Rep(\ZZ, \Mod_R).
\]
	from \cref{QCoh_Theta}.
	It follows that we have equivalences
\begin{equation}\label{Rep_C}
\C_{\Theta_R} = \C \otimes_k \QCoh(\Theta_R) \simeq \C \otimes_k \Rep(\ZZ, \Mod_R) \simeq \Rep(\ZZ, \C_R) \ .
\end{equation}
	Hence a family $F_\bullet \in \C_{\Theta_R}$ can be represented by
\[
F_\bullet  \colon  \ \cdots \xrightarrow{x} F_n \xrightarrow{x} F_{n-1} \xrightarrow{x} \cdots.
\]
where $F_n \in \C_R$ for every $n \in \ZZ$.
\end{recollection}

\begin{corollary}\label{vanishing_Theta}
	Let $\C \in \Prlo_k$ smooth and $\Spec(R) \in \dAff_k$ discrete.
	Let $F_\bullet \in \C_{\Theta_R}$ and represent it
\[
F_\bullet  \colon  \ \cdots \xrightarrow{x} F_n \xrightarrow{x} F_{n-1} \xrightarrow{x} \cdots.
\]
in virtue of \eqref{Rep_C}.
	Assume that $F_\bullet$ is pseudo-perfect. 
	Then $F_i \simeq 0$ for every $i\gg 0$.
\end{corollary}

\begin{proof}
	By \cref{relation_finiteness_conditions}, the $\infty$-category $\C$ admits a compact generator $E$. By \cref{compact_generator_smooth} we have
\[
F_\bullet (E) \in \Fun(\ZZ^{op}, \Mod_R)^\omega \ .
\]
	Thus by \cref{filtered_compact} there exists $N \in \ZZ$ such that $F_i(E)\simeq 0$ for every $i>N$.
	To conclude the proof notice that $E$ generates $\C^\omega$ under finite colimits retracts and shifts, so that $F_i(C) \simeq 0$ for every $i>N$ and $C \in \C^\omega$.
\end{proof}

\begin{lemma}\label{flat_Theta}
	Let $\C \in \Prlo_k$ smooth equipped with an admissible $t$-structure $\tau$ and $\Spec (R)  \in \dAff_k$ discrete.
	Let $F \in \C_{\Theta_R}$ and represent it by 
\[
F_\bullet \colon  \ \ \cdots \xrightarrow{x} F_n \xrightarrow{x} F_{n-1} \xrightarrow{x} \cdots
\]
in virtue of \eqref{Rep_C}.
	Assume $F$ is $\tau_{\Theta_R}$-flat.
	Then:
\begin{itemize}\itemsep=0.2cm
	\item[(a)] $F_n \in \C_R^\heartsuit$ for every $n \in \ZZ$;
	\item[(b)] the map $F_n \xrightarrow{x} F_{n-1}$ is a monomorphism for every $n \in \ZZ$;
	\item[(c)] $F_n / xF_{n+1}$ is $\tau_R$-flat for every $n \in \ZZ$.
\end{itemize}
\end{lemma}

\begin{proof}
	Let $p \colon \Spec(R[x]) \to \Theta_R$ be the atlas map.
	By \cref{stability_flatness} we have that $p_\C^\ast F_\bullet$ is $\tau_{R[x]}$-flat.
	In particular $p_\C^\ast F_\bullet \in \C^\heartsuit_{R[x]}$.
	By definition of $\tau_{R[x]}$, we have that $p_\C^\ast F_\bullet \in \C^\heartsuit_{R[x]}$ if and only if the underlying family $F$ over $R$ is in $\C_R^\heartsuit$. 
	By \cref{pullback_atlas_Theta} the family $F$ is identified with $\bigoplus_n F_n$, hence $F_n \in \C_R^\heartsuit$ for every $n$.
	This proves (a).\\ \indent
	By \cref{tau_flat_vs_flat_bis} the exact sequence in $\Mod_{R[x]}^\heartsuit$
\[
0 \to R[x] \xrightarrow{x} R[x]
\]
is sent to an exact sequence in $\C_{R[x]}^\heartsuit$
\[
0 \to p_\C^\ast F_\bullet \xrightarrow{x} p_\C^\ast F_\bullet \ .
\]
	By definition of $\tau_{R[x]}$, the above map induces an exact sequence on the underlying families over $R$
\[
0 \to \bigoplus_{n \in \ZZ} F_n \xrightarrow{x} \bigoplus_{n \in \ZZ} F_n \ .
\]
	Hence every map $x \colon F_n \to F_{n-1}$ is a monomorphism. This proves (b).\\ \indent
	The pullback of $F_\bullet$ along the closed immersion $i \colon \mathrm{B}\G_{m} \to \Theta_R$ is $\tau$-flat by \cref{stability_flatness}.
	The underlying $\tau$-flat family over $R$ is identified with $\bigoplus_n F_{n-1} / xF_{n}$ by (b) and \cref{QCoh_Theta}.
	Since $\tau_R$ is $\omega$-accessible, it follows that the homotopy group functors commute with arbitrary direct sums.
	Hence $F_n / xF_{n+1}$ is $\tau_R$-flat for every $n \in \ZZ$. This proves (c).
\end{proof}

\begin{lemma}\label{perf_Theta}
	Let $\C \in \Prlo_k$ smooth equipped with an admissible $t$-structure $\tau$.
	Let $R \in \Aff_k$.
	Let $F \in \C_{\Theta_R}$ and represent it by 
\[
F_\bullet \colon  \ \ \cdots \xrightarrow{x} F_n \xrightarrow{x} F_{n-1} \xrightarrow{x} \cdots
\]
in virtue of \eqref{Rep_C}.
	Assume that
\begin{enumerate}\itemsep=0.2cm
	\item for every $n \in \ZZ$ we have $F_n \in \C_R^\heartsuit$ and the map $F_n \xrightarrow{x} F_{n-1}$ is a monomorphism;
	\item for $n \gg 0$ we have $F_n \simeq 0$; 
	\item for every $n \in \ZZ$ the family $F_n / xF_{n+1}$ is pseudo-perfect and $\tau_R$-flat.
\end{enumerate}
	Then $F_n$ is $\tau_R$-flat and pseudo-perfect for every $n \in \ZZ$.
\end{lemma}

\begin{proof}
	By (2) there exists a biggest integer $n_0$ for which $F_{n_0} \neq 0$.
	It follows directly by (3) that $F_{n_0}$ is pseudo-perfect and $\tau_R$-flat.
	We claim that $F_{n_{0} -1}$ is also pseudo-perfect and $\tau_{R}$-flat; indeed by (1) it fits if an exact sequence
\[
0 \to F_{n_0} \to F_{{n_0}-1} \to F_{n_{0}-1} / x F_{n_0} \to 0.
\]
	Thus $F_{{n_0}-1}$ is an extension of pseudo-perfect and $\tau_{R}$-flat families, which proves the claim.
	Arguing by induction we get that $F_n$ is pseudo-perfect and $\tau_{R}$-flat for any $n$.
\end{proof}

\begin{notation}
	Let $\E$ be a stable $\infty$-category.
	Let $n \in \ZZ$.
	Consider the isomorphism of poset 
\[
g_{n} \colon \ZZ^{op} \to \ZZ^{op}
\]
\[m \mapsto m-n
\]
	Given $F_\bullet \in \Rep(\ZZ, \E)$ we denote by $F_\bullet \langle n \rangle$ composition
\[
\ZZ^{op} \xrightarrow{g_n} \ZZ^{op} \xrightarrow{F_\bullet} \E \ .
\]
\end{notation}

\begin{proposition}\label{flat_perf_Theta}
	Let $\C \in \Prlo_k$ smooth equipped with an admissible $t$-structure $\tau$.
	Let $R \in \Aff_k$.
	Let $F \in \C_{\Theta_R}$ and represent it by
\[
F_\bullet \colon  \ \ \cdots \xrightarrow{x} F_n \xrightarrow{x} F_{n-1} \xrightarrow{x} \cdots
\]
via \eqref{Rep_C}.
	Then $F$ is $\tau_{\Theta_R}$-flat and pseudo-perfect if and only if the following are satisfied:
\begin{itemize}\itemsep=0.2cm
    \item[(i)] for every $n \in \ZZ$ we have $F_n \in \C_R^\heartsuit$ and the map $F_n \xrightarrow{x} F_{n-1}$ is a monomorphism;
    \item[(ii)] for every $n \in \ZZ$ the map $F_n \xrightarrow{x} F_{n-1}$ is an isomorphism in $\C_R^\heartsuit$ for $n \ll 0$;
    \item[(iii)] for $n \gg 0$ we have $F_n \simeq 0$;
    \item[(iv)] for every $n \in \ZZ$ the family $F_n / xF_{n+1}$ is $\tau$-flat over $R$;
    \item[(v)] for every $n \in \ZZ$ the family $F_n / xF_{n+1}$ is pseudo-perfect over $R$.
\end{itemize}
\end{proposition}

\begin{proof}
	Suppose $F_\bullet \in \C_{\Theta_R}$ is $\tau_{\Theta_R}$-flat and pseudo-perfect.
	Then (i) and (iv) follow by \cref{flat_Theta}.
	By stability under pullbacks of pseudo-perfectness, the associated graded is pseudo-perfect (\cref{QCoh_Theta}-(4)).
	But a direct sum can be pseudo-perfect if and only if it is finite and all the summands are pseudo-perfect.
	It follows immediately that $x \colon  F_n \to F_{n+1}$ is an isomorphism for $n \ll 0$ and that $F_n / xF_{n+1}$ is pseudo-perfect for every $n \in \ZZ$.
	This proves (ii) and (v).
	Condition (iii) follows by \cref{vanishing_Theta}.\\ \indent
	Assume now that $F_\bullet$ satisfies (i), (ii), (iii) and (iv).
	By \cref{perf_Theta} the family $F_n$ is pseudo-perfect and $\tau_R$ flat for every $n \in \ZZ$.
	By (iii) the family $F_n \otimes_R \OO_{\Theta_R}$ is pseudo-perfect and $\tau$-flat over $\Theta_R$ for every $n \in \ZZ$.
	Shifting the grading by $-n$, we get the pseudo-perfect and $\tau$-flat family over $\Theta_R$
\[
F_n \otimes_R \OO_{\Theta_R}\langle -n \rangle  \colon  \ \cdots \to 0 \to F_n \xrightarrow{id} F_n \xrightarrow{id} F_n \xrightarrow{id} \cdots,
\]
where the first non-zero element is in weight $n$.
	Now the result follows since by (i), (iii) and (iv), $F_\bullet$ is obtained as a finite number of extensions of $\tau$-flat and pseudo-perfect families.
	Indeed by (iii) there exists a biggest integer $n_0$ for which $F_{n_0} \neq 0$.
	By (i) we have that the family 
\[
G_\bullet \coloneqq \cdots 0 \to F_{n_0} \xrightarrow{x} F_{n_0 -1} \xrightarrow{id} F_{n_0 -1} \xrightarrow{id} \cdots
\]
is the extension of $F\langle n_0\rangle$ by $F_{n_0 -1} / F_{n_0} \langle n_0 -1 \rangle$, i.e., we have a short exact sequence
\[
0 \to F_{n_0}\langle n_0 \rangle \to F_\bullet' \to F_{n_0 -1} / F_{n_0}\langle n_0 -1 \rangle \to 0.
\]
	By (iv) and (v) both  $F_{n_0}\langle n_0 \rangle$ and $F_{n_0 -1} / F_{n_0}\langle n_0 -1 \rangle$ are pseudo-perfect and $\tau_{\Theta_R}$-flat, so that $G_\bullet$ is pseudo-perfect and $\tau_{\Theta_R}$-flat.
	Iterating this process, which eventually stops by (ii), we get the result.
\end{proof}

\begin{remark}\label{Remark_flat_perf_Theta}
	In the setting of \cref{flat_perf_Theta}, if $F_\bullet$ is $\tau_{\Theta_R}$-flat and $\tau$ satisfies \cref{Assumption_subobject}, we can replace conditions (i)-to-(v) with conditions (i)-to-(iv) and:
\begin{itemize}\itemsep=0.2cm
	\item[(v')] for every $n \in \ZZ$ the family $F_n$ is pseudo-perfect over $R$.
\end{itemize}
	Indeed combining (i) with \cref{Assumption_subobject} we get that (v') implies (v).
	The converse holds by \cref{perf_Theta}.
\end{remark}

\begin{proposition}\label{Theta_reductive}
	Let $\C \in \Prlo_k$ of finite type equipped with an admissible $t$-structure $\tau$ satisfying \cref{Assumption_subobject} and universally satisfying openness of flatness. 
	The algebraic stack $\M_\C^\heartsuit$ is $\Theta$-reductive.
\end{proposition}

\begin{proof}
	Consider a pseudo-perfect $\tau_{\Theta_R}$-flat family $F \in \C_{\Theta_R \setminus \left\lbrace 0 \right\rbrace }$ corresponding to a map
\[
\Theta_R \setminus \left\lbrace 0 \right\rbrace \xrightarrow{F} \M_\C^\heartsuit.
\]
	We need to verify that the unique $\tau_{\Theta_R}$-flat extension $j_{\ast}F \in \C_{\Theta_R}$ provided by \cref{saved_lemma} is pseudo-perfect.
	By \cref{flat_Theta}, it is enough to to verify (ii), (iii) and (v) of \cref{flat_perf_Theta}. 
	By \cref{Remark_flat_perf_Theta}, we can replace (v) by (v') in \textit{loc.cit.}.
	The family $F$ can be described as:
\begin{itemize}\itemsep=0.2cm
    \item a family $F \in \C_{\Frac(R)}$ with a filtration 
\[
F_\bullet  \colon  \ \cdots \subseteq F_n \subseteq F_{n-1} \subseteq \cdots
\]
satisfying the conditions of \cref{flat_perf_Theta};
    \item A pseudo-perfect $\tau$-flat $R$-module subobject $E_1 \subseteq F$ with an isomorphism 
\[
E_1 \otimes_R \Frac(R) \simeq F \ .
\]
\end{itemize}
	Then the unique $\tau_{\Theta_R}$-flat extension can be explicitly described as in \cite[ Lemma 7.17]{AHLH}.
	That is, it is the filtered object
\[
j_{\C, \ast}F \simeq F_\bullet \cap E_1 = (\cdots \xrightarrow{x} F_n \cap E_1 \xrightarrow{x} F_{n-1} \cap E_1 \xrightarrow{x} \cdots) \hookrightarrow (\cdots \xhookrightarrow{x} F_n  \xhookrightarrow{x} F_{n-1}\xhookrightarrow{x} \cdots) = F_\bullet \ .
\]

	Since $E_1$ is pseudo-perfect and this properties is stable under taking subobjects by \cref{Assumption_subobject}, we get that $F_n \cap E_1$ is pseudo-perfect for any $n$.
	Since the filtration of $F$ stabilizes for $n \ll 0$ and $F_n \simeq 0$ for $n \gg 0$, the same holds true for $F_\bullet \cap E_1$.
	Thus $j_{\C, \ast}F$ satisfies (i)-to-(v) of \cref{flat_perf_Theta}, hence defining an unique extension
\[\begin{tikzcd}[sep=small]
	{\Theta_R \setminus \left\lbrace 0 \right\rbrace} && {\M_\C^\heartsuit} \\
	\\ 
	{\Theta_R}
	\arrow["F", from=1-1, to=1-3]
	\arrow["j"', hook, from=1-1, to=3-1]
	\arrow["{\exists ! \ j_{\C, \ast}F}"', dashed, from=3-1, to=1-3]
\end{tikzcd}\]
\end{proof}

	\cref{flat_perf_Theta} also has the following immediate corollary:

\begin{corollary}[{\cite[Lemma 7.19]{AHLH}}]\label{closed_point_is_semisimple_preliminar}
	In the setting of \cref{Theta_reductive}, let $\kappa$ be a field over $k$ and $F \in \C_\kappa^\heartsuit$ be a pseudo-perfect family corresponding to a point $x_F \colon \Spec(\kappa) \to \M_\C^\heartsuit$. 
	If $x_F$ is a closed point, then $F \in \C_\kappa^\heartsuit$ is semisimple.
\end{corollary}

\begin{proof}
	Notice that since $\kappa$ is a field, the family $F$ is $\tau_\kappa$-flat by \cref{tau_flat_in_heart}.
	The stack $\M_\C^\heartsuit$ has affine stabilizer by \cref{affine_diagonal}.
	Since $F$ is pseudo-perfect, it can not be expressed as an infinite direct sum of non-zero objects.
	Then the proof of \cite[Lemma 7.19]{AHLH} goes trough without changes.
\end{proof}

\subsection{$\mathrm{S}$-completeness}\label{section_S_complete}

	We will prove in this section that the moduli of objects $\M_\C^\heartsuit$ is $\mathrm{S}$-complete.
	Recall that $R$ is a DVR with uniformizer $\pi$, residue field $\kappa$ and fraction field $K$.
	In order to proceed with this verification, we will first characterize explicitly $\tau_{\ST}$-flat pseudo-perfect families in $\C_{\ST}$ in \cref{flat_perf_ST}.
	We recall in the next paragraph the main results of Appendix \ref{Appendix_QCoh_ST_Theta}.
\begin{recollection}\label{explicit_description_ST}
	The following derived analogue of \cite[Lemma 7.14]{AHLH} hold:
\begin{itemize}\itemsep=0.2cm
    \item Objects of $\QCoh(\ST)$ are identified with
\[\begin{tikzcd}[sep=small]
	\cdots & {F_{n+1}} & {F_{n}} & {F_{n-1}} & \cdots
	\arrow["t", shift left, from=1-1, to=1-2]
	\arrow["s", shift left, from=1-2, to=1-1]
	\arrow["t", shift left, from=1-2, to=1-3]
	\arrow["s", shift left, from=1-3, to=1-2]
	\arrow["t", shift left, from=1-3, to=1-4]
	\arrow["s", shift left, from=1-4, to=1-3]
	\arrow["t", shift left, from=1-4, to=1-5]
	\arrow["s", shift left, from=1-5, to=1-4]
\end{tikzcd}\]
where the $F_i$'s are $R$-modules and $t,s$ are maps of $R$-modules that satisfy $st= \pi =ts$.
	Moreover under this identification, an object is (co)connective if and only if every $F_n$ is.
    \item Restriction along the open immersion $\Spec(R) \xhookrightarrow{s \neq 0} ST$ induces the functor
\[
\begin{tikzcd}[sep=small]
	\cdots & {F_{n+1}} & {F_{n}} & {F_{n-1}} & \cdots & {\colim(\cdots \xrightarrow{t} F_n \xrightarrow{t} F_{n-1} \xrightarrow{t} \cdots )}
	\arrow["t", shift left, from=1-1, to=1-2]
	\arrow["s", shift left, from=1-2, to=1-1]
	\arrow["t", shift left, from=1-2, to=1-3]
	\arrow["s", shift left, from=1-3, to=1-2]
	\arrow["t", shift left, from=1-3, to=1-4]
	\arrow["s", shift left, from=1-4, to=1-3]
	\arrow["t", shift left, from=1-4, to=1-5]
	\arrow["s", shift left, from=1-5, to=1-4]
	\arrow[maps to, from=1-5, to=1-6]
\end{tikzcd}.
\]
    \item Restriction along the open immersion $\Spec(R) \xhookrightarrow{s \neq 0} ST$ induces the functor
\[
\begin{tikzcd}[sep=small]
	\cdots & {F_{n+1}} & {F_{n}} & {F_{n-1}} & \cdots & {\colim(\cdots \xleftarrow{s} F_n \xleftarrow{s} F_{n-1} \xleftarrow{s} \cdots )}
	\arrow["t", shift left, from=1-1, to=1-2]
	\arrow["s", shift left, from=1-2, to=1-1]
	\arrow["t", shift left, from=1-2, to=1-3]
	\arrow["s", shift left, from=1-3, to=1-2]
	\arrow["t", shift left, from=1-3, to=1-4]
	\arrow["s", shift left, from=1-4, to=1-3]
	\arrow["t", shift left, from=1-4, to=1-5]
	\arrow["s", shift left, from=1-5, to=1-4]
	\arrow[maps to, from=1-5, to=1-6]
\end{tikzcd}.
\]
    \item Restriction along the closed immersion $\Theta_{\kappa} \xhookrightarrow{s=0} \ST$ induces the functor
\[
\begin{tikzcd}[sep=small]
	\cdots & {F_{n+1}} & {F_{n}} & {F_{n-1}} & \cdots & {\cdots \xrightarrow{t} F_n / sF_{n-1} \xrightarrow{t} F_{n-1} / sF_{n-2} \xrightarrow{t} \cdots}
	\arrow["t", shift left, from=1-1, to=1-2]
	\arrow["s", shift left, from=1-2, to=1-1]
	\arrow["t", shift left, from=1-2, to=1-3]
	\arrow["s", shift left, from=1-3, to=1-2]
	\arrow["t", shift left, from=1-3, to=1-4]
	\arrow["s", shift left, from=1-4, to=1-3]
	\arrow["t", shift left, from=1-4, to=1-5]
	\arrow["s", shift left, from=1-5, to=1-4]
	\arrow[maps to, from=1-5, to=1-6]
\end{tikzcd}
\]
where the quotient has to be understood as cofibers along the morphisms $s$.
    \item Restriction along the closed immersion $\Theta_{\kappa} \xhookrightarrow{t=0} \ST$ induces the functor
\[
\begin{tikzcd}[sep=small]
	\cdots & {F_{n+1}} & {F_{n}} & {F_{n-1}} & \cdots & {\cdots \xleftarrow{s} F_n / tF_{n+1} \xleftarrow{s} F_{n-1} / tF_{n} \xleftarrow{s} \cdots}
	\arrow["t", shift left, from=1-1, to=1-2]
	\arrow["s", shift left, from=1-2, to=1-1]
	\arrow["t", shift left, from=1-2, to=1-3]
	\arrow["s", shift left, from=1-3, to=1-2]
	\arrow["t", shift left, from=1-3, to=1-4]
	\arrow["s", shift left, from=1-4, to=1-3]
	\arrow["t", shift left, from=1-4, to=1-5]
	\arrow["s", shift left, from=1-5, to=1-4]
	\arrow[maps to, from=1-5, to=1-6]
\end{tikzcd}
\]
where the quotient has to be understood as cofibers along the morphisms $t$.
\end{itemize}
\end{recollection}

\begin{recollection}\label{adj_ST}
	Via the equivalence of \cref{QCoh_ST} and \cref{explicit_description}, we can represent an object $F_\bullet \in \QCoh(\ST)$ by 
\[\begin{tikzcd}[sep=small]
	\cdots & {F_{n+1}} & {F_{n}} & {F_{n-1}} & \cdots
	\arrow["t", shift left, from=1-1, to=1-2]
	\arrow["s", shift left, from=1-2, to=1-1]
	\arrow["t", shift left, from=1-2, to=1-3]
	\arrow["s", shift left, from=1-3, to=1-2]
	\arrow["t", shift left, from=1-3, to=1-4]
	\arrow["s", shift left, from=1-4, to=1-3]
	\arrow["t", shift left, from=1-4, to=1-5]
	\arrow["s", shift left, from=1-5, to=1-4]
\end{tikzcd}\]
where $F_n \in \C_R$ for every $n \in \ZZ$.
	Let $i_{\left\lbrace s = 0 \right\rbrace} \colon \Theta_\kappa \to \ST$ be the closed immersion corresponding to the locus $\left\lbrace s = 0 \right\rbrace$.
	By \cref{closed_points}, the pullback along $i_{\left\lbrace s = 0 \right\rbrace}$ is identified with
\[
\Gr_s \colon \QCoh(\ST) \to \QCoh(\Theta_\kappa)
\]
\[
[\begin{tikzcd}[sep=small]
	\cdots & {F_{n}} & {F_{n-1}} & \cdots
	\arrow["t", shift left, from=1-1, to=1-2]
	\arrow["s", shift left, from=1-2, to=1-1]
	\arrow["t", shift left, from=1-2, to=1-3]
	\arrow["s", shift left, from=1-3, to=1-2]
	\arrow["t", shift left, from=1-3, to=1-4]
	\arrow["s", shift left, from=1-4, to=1-3]
\end{tikzcd}]
\mapsto
[\cdots \xrightarrow{t} \cof(F_{n} \xrightarrow{s} F_{n+1}) \xrightarrow{t} cof(F_{n-1} \xrightarrow{s} F_{n} \xrightarrow{t} \cdots)]
\]
	Shifting the degree by $-1$ we get a functor
\[
\fib_s \coloneqq \Gr_s [-1] \colon \QCoh(\ST) \to \QCoh(\Theta_\kappa)
\]
\[
[\begin{tikzcd}[sep=small]
	\cdots & {F_{n}} & {F_{n-1}} & \cdots
	\arrow["t", shift left, from=1-1, to=1-2]
	\arrow["s", shift left, from=1-2, to=1-1]
	\arrow["t", shift left, from=1-2, to=1-3]
	\arrow["s", shift left, from=1-3, to=1-2]
	\arrow["t", shift left, from=1-3, to=1-4]
	\arrow["s", shift left, from=1-4, to=1-3]
\end{tikzcd}]
\mapsto
[\cdots \xrightarrow{t} \fib(F_{n} \xrightarrow{s} F_{n+1}) \xrightarrow{t} fib(F_{n-1} \xrightarrow{s} F_{n}) \xrightarrow{t} \cdots]
\]
	The functor $\fib_s$ commutes with limits, thus admitting a left adjoint by \cite[Corollary 5.5.2.9]{HTT} informally given by
\[
\fib_s^\mathrm{L} \colon \QCoh(\Theta_\kappa) \to \QCoh(\ST)
\]
\[
[\cdots \xrightarrow{t} F_n \xrightarrow{t} F_{n-1} \xrightarrow{t} \cdots] \mapsto 
[\begin{tikzcd}[sep=small]
	\cdots & {F_{n}} & {F_{n-1}} & \cdots
	\arrow["t", shift left, from=1-1, to=1-2]
	\arrow["0", shift left, from=1-2, to=1-1]
	\arrow["t", shift left, from=1-2, to=1-3]
	\arrow["0", shift left, from=1-3, to=1-2]
	\arrow["t", shift left, from=1-3, to=1-4]
	\arrow["0", shift left, from=1-4, to=1-3]
\end{tikzcd}]
\]
\end{recollection}

\begin{lemma}\label{truncated_representation}
	Let $M \in \Mod_\kappa$ and $n \in \N$.
	Let $i_n \colon \left\lbrace n \right \rbrace \to \ZZ$ be the inclusion and $\ev_{n, !}$ be the left adjoint to the evaluation functor $\ev_n \coloneqq i_n^\ast \colon \Fun(\ZZ; \Mod_\kappa) \to \Mod_\kappa$.
	In the setting of \cref{adj_ST}, consider the quasi-coherent complex on $\ST$
\[M (\leq n) \coloneqq \fib_s^\mathrm{L} \circ \ev_{n,!} (M) \simeq
\begin{tikzcd}[sep = small]
	\cdots & 0 & M & M & \cdots
	\arrow[shift left, from=1-1, to=1-2]
	\arrow[shift left, from=1-2, to=1-1]
	\arrow[shift left, from=1-2, to=1-3]
	\arrow[shift left, from=1-3, to=1-2]
	\arrow["id", shift left, from=1-3, to=1-4]
	\arrow["0", shift left, from=1-4, to=1-3]
	\arrow["id", shift left, from=1-4, to=1-5]
	\arrow["0", shift left, from=1-5, to=1-4]
\end{tikzcd}\]
where the first non-zero copy of $M$ is placed in weight $n$.
	If $M \in \Perf(\kappa)$, then $M(\leq n) \in \Perf(\ST)$.
\end{lemma}

\begin{proof}
	In the adjunction 
\[
\fib_s^\mathrm{L} \circ \ev_{n,!} \dashv ev_n \circ \fib_s
\]
the right adjoint preserves filtered colimits as both $ev_n$ and $\fib_s$ do.
	It then follows by \cite[Lemma 5.5.14]{HTT} that $\fib_s^\mathrm{L} \circ \ev_{n,!}$ preserves compact objects.
	In particular, if $M \in \Perf(\kappa) = \Mod_\kappa^\omega$, then $M(\leq n) \in \QCoh(\ST)^\omega$.
	Since we are in characteristic zero, we also have $\QCoh(\ST)^\omega = \Perf(\ST)$ by \cite[Corollary 3.22]{BZFN}. Hence $M(\leq n)$ is perfect.
\end{proof}

\begin{lemma}\label{pseudo_perf_ST}
	Let $\C \in \Prlo_k$ smooth equipped with an admissible $t$-structure $\tau$.
	Let $F_{\bullet} \in \C_{\ST}$ and represent it by 
\[\begin{tikzcd}[sep=small]
	\cdots & {F_{n+1}} & {F_{n}} & {F_{n-1}} & \cdots
	\arrow["t", shift left, from=1-1, to=1-2]
	\arrow["s", shift left, from=1-2, to=1-1]
	\arrow["t", shift left, from=1-2, to=1-3]
	\arrow["s", shift left, from=1-3, to=1-2]
	\arrow["t", shift left, from=1-3, to=1-4]
	\arrow["s", shift left, from=1-4, to=1-3]
	\arrow["t", shift left, from=1-4, to=1-5]
	\arrow["s", shift left, from=1-5, to=1-4]
\end{tikzcd}\]
in virtue of \cref{adj_ST}.
	Assume that
\begin{itemize}\itemsep=0.2cm
	\item[(a)] $F_n \in \C_R^\heartsuit$ for every $n \in \ZZ$;
	\item[(b)] $t \colon F_n \to F_{n-1}$ is a monomorphism for every $n \in \ZZ$;
	\item[(c)] $s \colon  F_n \to F_{n+1}$ is an isomorphism for $n \gg 0$ and  $t \colon F_n \to F_{n-1}$ is an isomorphism $n \ll 0$;
	\item[(d)] $F_n $ is pseudo-perfect over $R$ for every $n \in \ZZ$.
\end{itemize}
	Then $F_\bullet$ is pseudo-perfect over $\ST$.
\end{lemma}

\begin{proof}
	By (c) there exists two integers $n_1 \geq n_2$ such that $s \colon  F_n \to F_{n+1}$ is an isomorphism for every $n \geq n_1$ and $t \colon F_n \to F_{n-1}$ is an isomorphism for every $n \leq n_2$. We will proceed by induction on $r = n_1 - n_2$.
	If $r=0$ then $F_\bullet$ is equivalent to 
\[(F_{n_2})_\bullet \langle n_2 \rangle \colon \ \begin{tikzcd}[sep=small]
	\cdots && F_{n_2} && F_{n_2} && F_{n_2} && \cdots
	\arrow["\pi", shift left, from=1-1, to=1-3]
	\arrow["id", shift left, from=1-3, to=1-1]
	\arrow["\pi", shift left, from=1-3, to=1-5]
	\arrow["id", shift left, from=1-5, to=1-3]
	\arrow["id", shift left, from=1-5, to=1-7]
	\arrow["\pi", shift left, from=1-7, to=1-5]
	\arrow["id", shift left, from=1-7, to=1-9]
	\arrow["\pi", shift left, from=1-9, to=1-7]
\end{tikzcd}\]
where the central copy of $F_{n_2}$ sits in weight $n_2$. 
     Let \[
p \colon \Spec(R[s,t] / (st-\pi)) \to \ST
\]
be the atlas map.
    Invoking \cref{flat_affine_atlas} it is enough to show that $(F_{n_2})_\bullet \langle {n_2} \rangle$ is pseudo-perfect over $\ST$.
	By \cref{computation_pullback_atlas_ST}, the pullback $p_\C^\ast(F_{n_2})_\bullet \langle {n_2} \rangle$ is identified with (a shift of) 
\[
\widetilde{F}_{n_2} \coloneqq F_{n_2} \otimes_R (R[s,t] / (st-\pi)).
\]
	Hence, it is enough to show that $\widetilde{F}_{n_2}$ is pseudo-perfect over  $\Spec(R[s,t]/(st-\pi))$.
	This follows immediately from (d) and \cref{stability_flatness}.
	Suppose now that the statement holds for $r-1 = n_1 - n_2-1$. 
	Then $F_\bullet$ is equivalent to
\[\begin{tikzcd}
	{G_\bullet\colon \ \ \ \cdots} & {F_{M+2}} & {F_{M+1}} & {F_{M}} & {F_{M}} & \cdots
	\arrow[shift left, from=1-1, to=1-2]
	\arrow[shift left, from=1-2, to=1-1]
	\arrow["t", shift left, from=1-2, to=1-3]
	\arrow["s", shift left, from=1-3, to=1-2]
	\arrow["t", shift left, from=1-3, to=1-4]
	\arrow["s", shift left, from=1-4, to=1-3]
	\arrow["id", shift left, from=1-4, to=1-5]
	\arrow["\pi", shift left, from=1-5, to=1-4]
	\arrow["id", shift left, from=1-5, to=1-6]
	\arrow["\pi", shift left, from=1-6, to=1-5]
\end{tikzcd}\]
	Consider the following sequence in $\QCoh(\ST)$:
\[\begin{tikzcd}[sep=small]
	{H_\bullet \colon\ \ \ \ \cdots} & {F_{n_2+2}} & {F_{n_2+1}} & {F_{n_2+1}} & {F_{n_2+1}} & \cdots \\
	\\
	{G_\bullet\colon \ \ \ \cdots} & {F_{n_2+2}} & {F_{n_2+1}} & {F_{n_2}} & {F_{n_2}} & \cdots \\
	\\
	{F_{n_2} / tF_{n_2 +1}(\leq n_2)\colon \ \cdots} & 0 & 0 & {F_{n_2} / tF_{n_2 +1}} & {F_{n_2} / tF_{n_2 +1}} & \cdots
	\arrow[shift left, from=1-1, to=1-2]
	\arrow[shift left, from=1-2, to=1-1]
	\arrow["t", shift left, from=1-2, to=1-3]
	\arrow["id"', from=1-2, to=3-2]
	\arrow["s", shift left, from=1-3, to=1-2]
	\arrow["id", shift left, from=1-3, to=1-4]
	\arrow["id"', from=1-3, to=3-3]
	\arrow["\pi", shift left, from=1-4, to=1-3]
	\arrow["id", shift left, from=1-4, to=1-5]
	\arrow["t"', from=1-4, to=3-4]
	\arrow["\pi", shift left, from=1-5, to=1-4]
	\arrow["id", shift left, from=1-5, to=1-6]
	\arrow["t"', from=1-5, to=3-5]
	\arrow["\pi", shift left, from=1-6, to=1-5]
	\arrow[shift left, from=3-1, to=3-2]
	\arrow[shift left, from=3-2, to=3-1]
	\arrow["t", shift left, from=3-2, to=3-3]
	\arrow[two heads, from=3-2, to=5-2]
	\arrow["s", shift left, from=3-3, to=3-2]
	\arrow["t", shift left, from=3-3, to=3-4]
	\arrow[from=3-3, to=5-3]
	\arrow["s", shift left, from=3-4, to=3-3]
	\arrow["id", shift left, from=3-4, to=3-5]
	\arrow[from=3-4, to=5-4]
	\arrow["\pi", shift left, from=3-5, to=3-4]
	\arrow["id", shift left, from=3-5, to=3-6]
	\arrow[from=3-5, to=5-5]
	\arrow["\pi", shift left, from=3-6, to=3-5]
	\arrow[shift left, from=5-1, to=5-2]
	\arrow[shift left, from=5-2, to=5-1]
	\arrow[shift left, from=5-2, to=5-3]
	\arrow[shift left, from=5-3, to=5-2]
	\arrow[shift left, from=5-3, to=5-4]
	\arrow[shift left, from=5-4, to=5-3]
	\arrow["id", shift left, from=5-4, to=5-5]
	\arrow["0", shift left, from=5-5, to=5-4]
	\arrow["id", shift left, from=5-5, to=5-6]
	\arrow["{0}", shift left, from=5-6, to=5-5]
\end{tikzcd}\]
	By (b) the above is a fiber sequence in $\QCoh(\ST)$.
	Hence it is enough to show that $H_\bullet$ and $F_{n_2} / tF_{n_2 +1}(\leq n_2)$ are pseudo-perfect.
	The former is pseudo-perfect by recursion. 
	By (b) and (d), the quotient $F_{n_2} / tF_{n_2 +1}$ is pseudo-perfect over $R$.
	Moreover the action of $\pi$ on $F_{n_2} / tF_{n_2 +1} \in \C_R \simeq \Fun^{ex}((\C^\omega)^{op}, \Mod_R)$ is trivial as $\pi = ts$.
	It follows that $F_{n_2} / tF_{n_2 +1}$ is (the pushforward of an element) in $\C_\kappa \simeq \Fun^{ex}((\C^\omega)^{op}, \Mod_\kappa)$, and it is automatically pseudo-perfect over $\kappa$.
	By \cref{truncated_representation}, the family $F_{n_2} / tF_{n_2 +1}(\leq n_2)$ is pseudo-perfect over $\ST$. 
\end{proof}

\begin{lemma}\label{flat_ST}
	Let $\C \in \Prlo_k$ smooth equipped with an admissible $t$-structure $\tau$.
	Let $F_{\bullet} \in \C_{\ST}$ and represent it by 
\[\begin{tikzcd}[sep=small]
	\cdots & {F_{n+1}} & {F_{n}} & {F_{n-1}} & \cdots
	\arrow["t", shift left, from=1-1, to=1-2]
	\arrow["s", shift left, from=1-2, to=1-1]
	\arrow["t", shift left, from=1-2, to=1-3]
	\arrow["s", shift left, from=1-3, to=1-2]
	\arrow["t", shift left, from=1-3, to=1-4]
	\arrow["s", shift left, from=1-4, to=1-3]
	\arrow["t", shift left, from=1-4, to=1-5]
	\arrow["s", shift left, from=1-5, to=1-4]
\end{tikzcd}\]
in virtue of \cref{adj_ST}.
	Then $F$ is $\tau_{\ST}$-flat if and only if the following are satisfied:
\begin{itemize}\itemsep=0.2cm
    \item[(i)] for every $n \in \ZZ$, we have $F_n \in \C_R^\heartsuit$ and the maps $F_n \xrightarrow{t} F_{n-1}, F_n \xrightarrow{s} F_{n+1}$ are monomorphisms;
    \item[(ii)] for every $n \in \ZZ$, the map $F_{n+1} / sF_{n} \xrightarrow{t} F_{n} / sF_{n-1}$ is a monomorphism.
\end{itemize}
	Also, condition (ii) can be replaced by
\begin{itemize}\itemsep=0.2cm
    \item[(ii')] for every $n \in \ZZ$, the map $F_{n-1} / tF_{n+1} \xrightarrow{s} F_{n+1} / tF_{n}$ is a monomorphism for every $n \in \ZZ$.
\end{itemize}
\end{lemma}

\begin{proof}
	Assume $F \in \C_{\ST}$ is $\tau_{\ST}$-flat.
	Arguing as in the proof of \cref{flat_Theta} and using \cref{pullback_atlas_ST}, item (i) follows.
	Similarly using \cref{explicit_description_ST}, item (ii) and (ii') follow.
	But then by \cref{flat_perf_Theta}-(iii) we have that $F_{n+1}/ sF_n \simeq 0$ for $n\gg 0$ and $F_n / tF_{n-1} \simeq 0$ for $n \ll 0$.\\ \indent
	Conversely, suppose that $F_\bullet \in \C_{\ST}$ satisfies (i) and (ii).
	Then (i) shows that multiplication by $\pi$ is injective on $(F_\bullet)|_{\left\lbrace s \neq 0 \right\rbrace}$ and $(F_\bullet)|_{_{\left\lbrace t \neq 0 \right\rbrace}}$. 
	Hence $(F_\bullet)|_{\left\lbrace s \neq 0 \right\rbrace}$ and $(F_\bullet)|_{_{\left\lbrace t \neq 0 \right\rbrace}}$ are $\tau$-flat over $(\ST)_{s \neq 0} \simeq \Spec(R) \simeq (\ST)_{t \neq 0}$ by \cref{flatness_DVR}.
	By (ii) we have that multiplication by $t$ is injective on $(F_\bullet)|_{\left\lbrace t = 0\right\rbrace}$, which in turn implies that $(F_\bullet)$ is $\tau$-flat at the origin by \cref{local_criterion_ST}. 
	A similar argument shows that (i) and (ii') also characterize flatness.
\end{proof}

\begin{proposition}\label{flat_perf_ST}
	In the setting of \cref{Assumption_subobject}, let $F_{\bullet} \in \C_{\ST}$ and represent it by 
\[\begin{tikzcd}[sep=small]
	\cdots & {F_{n+1}} & {F_{n}} & {F_{n-1}} & \cdots
	\arrow["t", shift left, from=1-1, to=1-2]
	\arrow["s", shift left, from=1-2, to=1-1]
	\arrow["t", shift left, from=1-2, to=1-3]
	\arrow["s", shift left, from=1-3, to=1-2]
	\arrow["t", shift left, from=1-3, to=1-4]
	\arrow["s", shift left, from=1-4, to=1-3]
	\arrow["t", shift left, from=1-4, to=1-5]
	\arrow["s", shift left, from=1-5, to=1-4]
\end{tikzcd}\]
in virtue of \cref{adj_ST}.
	Assume $F_\bullet$ is $\tau_{\ST}$-flat. Then $F_\bullet$ is pseudo-perfect if and only if:
\begin{itemize}\itemsep=0.2cm
    \item[(i)] for every $n \gg 0$, the map $s \colon  F_n \to F_{n+1}$ is an isomorphism and for every $n \ll 0$, the map $t \colon F_n \to F_{n-1}$ is an isomorphism;
    \item[(ii)] for every $n \in \ZZ$, the family $F_n $ is pseudo-perfect over $R$.
\end{itemize}
\end{proposition}

\begin{proof}
	Assume that $F_\bullet$ is pseudo-perfect.
	The restriction to $s=0$ and $t=0$ are pseudo-perfect and $\tau_{\Theta_\kappa}$-flat.
	By \cref{flat_perf_Theta}-(iii) and \cref{explicit_description_ST}, we get $F_{n+1} / sF_n \simeq 0$ for $n\gg 0$ and $F_n / tF_{n-1} \simeq 0$ for $n\ll 0$.
This proves (iii).
	To get (ii) one notices that by (i) we have that there exists $N \ll 0$ such that 
\[
(i_{\left\lbrace t \neq 0\right\rbrace})_\C^\ast F_\bullet \simeq \colim (\cdots \xrightarrow{t} F_n \xrightarrow{t} F_{n-1} \xrightarrow{t} \cdots) \simeq F_N.
\]
	Thus $F_N$ is pseudo-perfect and $\tau$-flat over $R$.
	By \cref{flat_perf_ST}-(i) and \cref{Assumption_subobject}, the family $F_n$ is pseudo-perfect for every $n \in \ZZ$.
    Conversely, assume that $F_\bullet$ satisfies (i) and (ii). 
	Since $F_\bullet$ is $\tau_{\ST}$-flat, \cref{flat_ST} implies  that the conditions of \cref{pseudo_perf_ST} are satisfied. 
	The conclusion follows.
\end{proof}

\begin{proposition}\label{S_complete}
	Let $\C \in \Prlo_k$ of finite type equipped with an admissible $t$-structure $\tau$ satisfying \cref{Assumption_subobject} and universally satisfying openness of flatness. 
	The algebraic stack $\M_\C^\heartsuit$ is $\mathrm{S}$-complete.
\end{proposition}

\begin{proof}
	Consider a $\tau$-flat pseudo-perfect family $F_\bullet \in \C_{\ST \setminus \left\lbrace 0 \right\rbrace }$ corresponding to a map
\[
\ST \setminus \left\lbrace 0 \right\rbrace \xrightarrow{F} \M_\C^\heartsuit.
\]
	We need to verify that the unique $\tau_{\ST}$-flat extension $\pi_0j_{\C, \ast}F \in \C_{\ST}$ provided by \cref{saved_lemma} is pseudo-perfect. By \cref{descent_ST}, the family $F_\bullet$ can be described as:
\begin{itemize}\itemsep=0.2cm
    \item a family $F \in \C_{\Frac(R)}^\heartsuit$;
    \item pseudo-perfect $\tau$-flat $R$-module subobjects $E_1, E_2 \subseteq F$ with equivalences
    \[
    E_1 \otimes_R \Frac(R) \simeq F \simeq E_2 \otimes_R \Frac(R) \ .
    \]
\end{itemize}
	Moreover, \cref{pushforward_ST} shows that the unique $\tau_{\ST}$-flat extension can be explicitly described by
\[
\pi_0j_{\C, \ast}F =\bigoplus_n (E_1 \cap (\pi^{-n}E_2))t^n \subseteq \bigoplus_n F t^n,
\]
	Since $E_1$ is pseudo-perfect and $E_1 \cap (\pi^{-n}E_2)$ is a subobject of $E_1$, each graded piece is pseudo-perfect by \cref{Assumption_subobject}.
	Since
\[
F \simeq E_2 \otimes_R K \simeq E_2 \otimes_R \colim (\cdots \xrightarrow{\pi} R \xrightarrow{\pi} R \xrightarrow{\pi} \cdots),
\]
the union of the filtration 
\[
\cdots \subseteq E_1 \cap(\pi^{-n}
E_2) \subseteq E_1 \cap(\pi^{-n-1}
E_2) \subseteq \cdots
\]
is $E_1 \cap (E_2 \otimes_R K) \simeq E_1 \cap F = E_1$.
	Because $E_1$ is pseudo-perfect (thus compact by \cref{pseudo_perfect_vs_compact}), this union must stabilize.
	Similarly by inverting the roles of $E_1$ and $E_2$, the filtration also stabilizes on the left.
	Thus $\pi_0j_{\C, \ast}F$ satisfies (i)-(ii) of \cref{flat_perf_ST}, hence defining an unique extension
\[
\begin{tikzcd}[sep=small]
	{\ST \setminus \left\lbrace 0 \right\rbrace} && {\M_\C^\heartsuit} \\
	\\
	{\ST}
	\arrow["F", from=1-1, to=1-3]
	\arrow["j"', hook, from=1-1, to=3-1]
	\arrow["{\exists ! \ \pi_0j_{\C, \ast}F}"', dashed, from=3-1, to=1-3]
\end{tikzcd}
\]
\end{proof}

	\cref{S_complete} allows us to characterize closed points of $\M_\C^\heartsuit$. Indeed we show in \cref{closed_point_is_semisimple} that closed points are in bijection with semisimple objects.


\begin{lemma}\label{specialization}
	Let $\kappa$ be a field of characteristic zero and $\X \hookrightarrow \Y$ be an immersion in $\St_\kappa$.
	\begin{enumerate}\itemsep=0.2cm
	    \item If $\Y$ is $\mathrm{S}$-complete, $\X$ is of finite type and $\kappa$ is algebraically closed, then any specialization of $\kappa$-points $x \rightsquigarrow y$ in $\X$, where $y$ is closed in $\X$, is realized by a morphism $\Theta_\kappa \to \X$.
	    \item If $\Y$ is $\Theta$-reductive and $\mathrm{S}$-complete, then the closure of any $\kappa$-point of $\Y$ contains a unique closed point.
	\end{enumerate}
\end{lemma}

\begin{proof}
	Since immersions induces equivalence between automorphisms groups of points and we are in characteristic zero, it follows from \cref{affine_diagonal}, \cref{S_complete} and \cite[Remark 2.6 \& Proposition 3.47]{AHLH} that we can apply \cite[Lemma 3.24 \& Lemma 3.25]{AHLH}.
\end{proof}

\begin{recollection}\label{nu_substacks}
	Let $\C \in \Prlo_k$ be of finite type equipped with an admissible $t$-structure $\tau$ universally satisfying openness of flatness.
	Let $E \in \C$ be a compact generator and $\nu: \ZZ \to \N$ be a function with finite support.
	Define a presheaf
\[
\M_\C^\nu: \dAff_k^{op} \to \Spc
\]
that assigns to $\Spec (A)$ the maximal $\infty$-subgroupoid
\[
\C^\nu(A) \subset (C_A^{pp})^{\simeq}
\]
spanned by those $F \in \Fun^{st}((\C^\omega)^{op}, \Perf(A))$ for which the following property holds: for any field $K$ and any $A \to \pi_0(A) \to K$ we have $\dim_K \pi_i(F(E)\otimes_A K) \leq \nu(i)$.
	This defines a substack of $\M_\C$ that depends on the choice of $nu$ and of the compact generator $E$.
	By \cite[Proposition 3.20]{TV}, the derived stack $\M_\C^\nu$ is an open substack of $\M_\C$ of finite presentation.
	We denote by $\M_\C^{[0,0], \nu}$ the intersection of the open substacks $\M_\C^{[0,0]}$ and $\M_\C^\nu$ in $\M_\C$, and by $\M_\C^{\heartsuit, \nu}$ its truncation.
	Since $\M_\C^{\heartsuit, \nu} \subset t_0\M_\C^\nu$ is an open substack of a finitely presented stack over the noetherian ring $k$, it follows that $\M_\C^{\heartsuit, \nu}$ is locally of finite presentation and quasi-compact.
	Hence $\M_\C^{\heartsuit, \nu}$ is of finite presentation.
\end{recollection}

\begin{lemma}\label{max_nu}
	In the setting of \cref{nu_substacks}, let $\kappa$ be a field over $k$ and $x,y \colon \Spec(\kappa) \to \M_\C^\heartsuit$.
	There exists a function $\nu \colon \ZZ \to \N$ with finite support such that both $x$ and $y$ factor through $\M_\C^{\heartsuit, \nu}$.
\end{lemma}

\begin{proof}
	Let $F,G \in \C_\kappa^\heartsuit \subset \Fun^{st}((\C^\omega)^{op}, \Perf(\kappa))$  be the pseudo-perfect objects corresponding respectively to $x,y$.
	For $i \in \ZZ$, define $h^i(F) \coloneqq \dim_\kappa H^i(F(E))$ and define $h^i(G)$ similarly.
	Since $F(E),G(E) \in \Perf(\kappa)$, the function
\[
\nu\colon \ZZ \to \N
\]
\[
i \mapsto \mathrm{max}\left\lbrace h^i(F), h^i(G) \right\rbrace
\]
has finite support.
	By definition of $\nu$, the morphisms $x,y \colon \Spec(\kappa) \to \M_\C^\heartsuit$ factor through $\M_\C^{\heartsuit, \nu}$.
\end{proof}

\begin{corollary}\label{closed_point_is_semisimple}
	Let $\C \in \Prlo_k$ of finite type equipped with an admissible $t$-structure $\tau$ satisfying \cref{Assumption_subobject} and universally satisfying openness of flatness. 
	Let $\Spec(\kappa) \to \Spec(k)$ be a closed point.
	There is a bijection between closed $\kappa$-points of $\M_\C^\heartsuit$ and pseudo-perfect semisimple objects in $\C_{\kappa}^\heartsuit$.
\end{corollary}

\begin{proof}
	Since the (closed) $\kappa$-points of $\M_\C^\heartsuit$ are in bijection with the (closed) $\kappa$-points of $\M_\C^\heartsuit \times_{\Spec(k)} \Spec(\kappa)$, we can assume that $\M_\C^\heartsuit \in \St_\kappa$.
	Notice that since $\kappa$ is a field, any family in $\C_\kappa^\heartsuit$ is $\tau_\kappa$-flat (\cref{tau_flat_in_heart}).
	Hence every pseudo-perfect family $F \in \C_\kappa^\heartsuit$ corresponds uniquely to a morphism $x_F \colon \Spec(\kappa) \to \M_\C^\heartsuit$.\\ \indent
	By \cref{closed_point_is_semisimple_preliminar}, any closed $\kappa$-point must parametrize a pseudo-perfect semisimple object of $\C_{\kappa}^\heartsuit$.
	Let now $F \in \C_\kappa^\heartsuit$ be a pseudo-perfect semisimple object and $x_F \colon \Spec(\kappa) \to \M_\C^\heartsuit$ be the corresponding morphism.
	By \cref{specialization}-(2), the closure of $x_F$ contains a unique closed point $x_G$.
	Up to replace $\kappa$ with an extension, we can suppose that $\kappa$ is algebraically closed and that $x_F, x_G$ are $\kappa$-rational.
	Let $G \in C^\heartsuit_\kappa$ be the pseudo-perfect family corresponding to $x_G$.
	We wish to show that $x_F$ and $x_G$ agree, that is $F \simeq G$.
	By \cref{max_nu}, there exists a function $\nu \colon \ZZ \to \N$ with finite support such that $x_F, x_G$ factor through $\M_\C^{\heartsuit, \nu}$.
	By \cref{specialization}-(1) and \cref{nu_substacks}, the specialization $x_F \rightsquigarrow x_G$ in $\M_\C^{\heartsuit, \nu}$ is realized by a morphism $\Theta_{\kappa} \to \M_\C^{\heartsuit, \nu} \hookrightarrow \M_\C^\heartsuit$.
	Explicitly this means that $G$ is identified with the associated graded of a filtration of $F$ (\cref{flat_perf_Theta} and \cref{QCoh_Theta}-(4)).
	Since $F$ is supposed semisimple, it follows that $F \simeq G$.
\end{proof}

\subsection{Existence of good moduli spaces for $\M_\C^\heartsuit$}\label{existence_good_section}

	We know put all the results together to obtain the main statement of the article.

\begin{definition}
	Let $\C \in \Prlo_k$ be of finite type equipped with an admissible $t$-structure $\tau$.
	Let $A$ be a discrete commutative algebra over $k$ and $F \in \C^\heartsuit_A$ be a pseudo-perfect $\tau_A$-flat object.
	We say that $F$ lies over a substack $\X \subset \M_\C^\heartsuit$ if the associated morphism $x_F \colon \Spec(A) \to \M_\C^\heartsuit$ factors through $\X$.
\end{definition}

    For $\Spec(\kappa) \to \Spec(k)$ a closed point with $\kappa$ algebraically closed, we have the following
\begin{theorem}\label{good_general}
	Let $\C \in \Prlo_k$ of finite type equipped with an admissible $t$-structure $\tau$ satisfying \cref{Assumption_subobject} and universally satisfying openness of flatness. 
	Then $\M_\C^\heartsuit$ is $\Theta$-reductive and $\mathrm{S}$-complete.
	In particular, any closed quasi-compact substack $\X \subset \M_\C^\heartsuit$ admits a separated good moduli space $X$ whose $\kappa$-points parametrize pseudo-perfect semisimple objects of $\C_\kappa^\heartsuit$ lying over $\X$.
	Moreover, $X$ admits a natural derived enhancement if $\X$ is the truncation of a derived substack of $\M_\C^{[0,0]}$.
\end{theorem}

\begin{proof}
	By \cref{Theta_reductive} and \cref{S_complete}, $\M_\C^\heartsuit$ is $\Theta$-reductive and $\mathrm{S}$-complete.
	Moreover $\M_\C^\heartsuit$ has affine diagonal by \cref{affine_diagonal}.
	Since this properties are preserved under taking closed substacks, the hypothesis of \cref{existence_good_separated} are satisfied for $\X$.
	The characterization of $\kappa$-points of $X$ follows from \cref{good_moduli_properties}-(4) and \cref{closed_point_is_semisimple}.
	The natural derived enhancement is provided by \cref{derived_good}.
\end{proof}

\begin{corollary}\label{good_general_qc_components}
	In the setting of \cref{good_general}, assume that the connected components of $\M_\C^\heartsuit$ are quasi-compact.
	Then $\M_\C^\heartsuit$ admits a separated good moduli space $\mathrm{M}_\C^\heartsuit$.
	Moreover, the $\kappa$-points of $\mathrm{M}_\C^\heartsuit$ parametrize pseudo-perfect semisimple objects of $\C_\kappa^\heartsuit$.
	Furthermore, $\mathrm{M}_\C^\heartsuit$ admits a natural a derived enhancement $\mathrm{M}_\C^{[0,0]}$ which is a derived good moduli space for $\M_\C^{[0,0]}$.
\end{corollary}

\begin{proof}
By \cite[Proposition 4.7-(ii)]{Alp}, it is enough to prove that every connected component of $\M_\C^\heartsuit$ admits a separated good moduli space.
	The claims then follow at once by our assumption that the connected components of $\M_\C^\heartsuit$ are quasi-compact and \cref{good_general}.	
\end{proof}

	In the next paragraph, we give applications of \cref{good_general}.

\subsubsection{Quasi-compact substacks of extensions}\label{subsection_extension}
	In this subsection we want to prove that there are natural closed substacks of $\M_\C^\heartsuit$ that are quasi-compact.
	This substacks roughly arise by fixing a finite number of objects in $\C^\heartsuit$ and looking at all elements that are extensions of the given objects.

\begin{recollection}
	Let $\C \in \Prlo_k$ of finite type equipped with an admissible $t$-structure $\tau$ universally satisfying openness of flatness.
	Following \cite[Section I.5.6]{HDR}, we introduce the derived stack of extensions $\M_\C^{[0,0], ext}$.
	We will not give the precise construction, but we will limit ourselves to describe its points: a morphism $\Spec (A) \to \M_\C^{[0,0], ext}$ corresponds to a fiber/cofiber sequence
\[
F_1 \to F_2 \to F_3
\]
with the $F_i$'s pseudo-perfect and $\tau_A$-flat.
	Notice that if $\Spec (A)$ is underived, the above is a short exact sequence in $\C_A^\heartsuit$.
	The stack $\M_\C^{[0,0], ext}$ comes with a morphism
\[
(\ev_1, \ev_2, \ev_3)\colon \M_\C^{[0,0], ext} \to \M_\C^{[0,0]} \times \M_\C^{[0,0]} \times \M_\C^{[0,0]}
\]
that sends a fiber/cofiber sequence to its components.
\end{recollection}

\begin{notation}
	We will denote by $\M_\C^{\heartsuit, ext}$ the classical truncation of the derived stack $\M_\C^{[0,0], ext}$.
\end{notation}

\begin{proposition}\label{proper_ev}
	Let $\C \in \Prlo_k$ of finite type equipped with an admissible $t$-structure $\tau$ satisfying \cref{Assumption_subobject} and universally satisfying openness of flatness.
	The morphism of stacks over $k$
\[
\ev_2\colon \M_\C^{\heartsuit, ext} \to \M_\C^\heartsuit
\]
\[
(0 \to F_1 \to F_2 \to F_3 \to 0) \mapsto F_2
\]
satisfies the valuative criterion for properness.
\end{proposition}

\begin{proof}
	Let $R$ be a DVR and $K$ be its fraction field.
	Consider the lifting problem
\[\begin{tikzcd}[sep=small]
	{\Spec (K)} && {\M_\C^{\heartsuit, ext}} \\
	\\
	{\Spec (R) } && {\M_\C^{\heartsuit}}
	\arrow["t", from=1-1, to=1-3]
	\arrow[from=1-1, to=3-1]
	\arrow[from=1-3, to=3-3]
	\arrow[dashed, from=3-1, to=1-3]
	\arrow["s"', from=3-1, to=3-3]
\end{tikzcd}\]
	We want to prove that a lift exist and it is unique.
	The above diagram corresponds to a short exact sequence
\begin{equation}\label{s.e.s.}
    0 \to F_1 \to \widetilde{F}_2 \otimes_R K  \to F_3 \to 0
\end{equation}
in $\C_K^\heartsuit$, where $F_1, F_3$ are pseudo-perfect $\tau$-flat families over $K$, and $F_2$ is a pseudo-perfect $\tau$-flat families over $R$.
	In virtue of \cref{induced_t_structure} we can see \eqref{s.e.s.} as a short exact sequence in $\C_R^\heartsuit$.
	Notice that the families in \eqref{s.e.s.} are not necessarily pseudo-perfect over $R$.
	Define $\widetilde{F}_3$ to be the image of the composite map $\widetilde{F}_2 \to \widetilde{F}_2\otimes_R K \to F_3$.
	By construction $\widetilde{F}_3$ is a quotient of $\widetilde{F}_2$.
	Hence $\widetilde{F}_3$ is pseudo-perfect over $R$ by \cref{Assumption_subobject}.
	Since $F_3$ is $\tau_R$-flat and $\widetilde{F}_3$ is a subobject of $F_3$, it follows from \ref{flatness_DVR} that $\widetilde{F}_3$ is $\tau_R$-flat.
	By the two-out-of-three property for perfect complexes, it follows that 
\[
\widetilde{F}_1 \coloneqq \mathrm{Ker}(\widetilde{F}_2 \to \widetilde{F}_3) \simeq \mathrm{fib}(\widetilde{F}_2 \to \widetilde{F}_3)
\]
is pseudo-perfect.
	Since $\widetilde{F}_1$ is a subobject of $\widetilde{F}_2$, it follows again from \cref{flatness_DVR} that $\widetilde{F}_1$ is also $\tau_R$-flat.
	Thus we have constructed a lift
\[
0 \to \widetilde{F}_1 \to \widetilde{F}_2 \to \widetilde{F}_3 \to 0
\]
of \eqref{s.e.s.}.
	The uniqueness follows by the uniqueness of the epi-mono factorization.
\end{proof}

	We will also need the following lemma:

\begin{lemma}\label{qc_map}
	Let $\C \in \Prlo_k$ of finite type equipped with an admissible $t$-structure $\tau$ universally satisfying openness of flatness.
	There exists a perfect complex $\E \in \Perf(\M_\C^{[0,0]} \times\M_\C^{[0,0]})$ and a commutative
\[\begin{tikzcd}[sep=small]
	{\M_\C^{[0,0], ext}} &&&& {\mathbb{V}(\E)} \\
	\\
	&& {\M_\C^{[0,0]} \times\M_\C^{[0,0]}}
	\arrow["\phi", from=1-1, to=1-5]
	\arrow["{(\ev_1, \ev_3)}"', from=1-1, to=3-3]
	\arrow["p", from=1-5, to=3-3]
\end{tikzcd}\]
where $\mathbb{V}(\E)\coloneqq \Spec(\mathrm{Sym}(\E))$ and $\phi$ is an equivalence.
	In particular $(ev_1, ev_3)$ is a quasi-compact morphism.
\end{lemma}

\begin{remark}
	\cref{qc_map} should not be surprising, indeed above any point parametrizing a pair of objects $F_1, F_3$, the (connected components of) fiber of $(\ev_1, \ev_3)$ consists on the possible extensions $\Ext^1(F_1, F_3)$.
\end{remark}

\begin{proof}
	Let $\mathcal{F}$ be the universal family of $\M_\C^{[0,0]}$, i.e., the $\tau$-flat pseudo-perfect family 
\[
\mathcal{F} \in \Fun((\C^{\omega})^{op}, \Perf(\M_\C^{[0,0]}))
\]
corresponding to the identity of $\M_\C^{[0,0]}$.
	Let $p_i \colon \M_\C^{[0,0]} \times \M_\C^{[0,0]} \to \M_\C^{[0,0]}$ be the canonical projections and set $\mathcal{F}_i \coloneqq p_i^\ast\mathcal{F}$.
	By definition, $\mathcal{F}_1$ and $\mathcal{F}_2$ are pseudo-perfect.
	Thus
\[
\E \coloneqq \Hom(\mathcal{F}_1, \mathcal{F}_2)^\vee[-1] \in \Perf(\M_\C^{[0,0]} \times \M_\C^{[0,0]})
\]
by \cref{pseudo_perf_Hom} and we can define the linear stack over $\M_\C^{[0,0]} \times \M_\C^{[0,0]}$ associated to $\E$:
\[
\mathbb{V}(\E) \coloneqq \Spec(\mathrm{Sym}(\E)).
\] 
	For $x \colon \Spec (A) \to \M_\C^{[0,0]} \times \M_\C^{[0,0]}$ parametrizing a couple of pseudo-perfect $\tau_A$-flat families $F_1, F_2$, we have
\begin{equation}\label{Hom_lin}
\begin{aligned}
    \Map_{\big/(\M_\C^{[0,0]} \times \M_\C^{[0,0]})} & (\Spec (A), \mathbb{V}(\E))  \simeq \tau_{\geq 0}\Hom_{\Mod_A}(x^\ast\E, A) \simeq\\
    & \simeq \tau_{\geq 0}\Hom_{\Mod_A}(\Hom(F_1, F_2)^\vee[-1], A) \simeq \\
    & \simeq \tau_{\geq 0}\Hom_{\Mod_A}(A, \Hom(F_1, F_2)[1]) \simeq\\
    & \simeq \tau_{\geq 0}(\Hom(F_1, F_2)[1]),
\end{aligned}    
\end{equation}
where we used that $x^\ast\E \simeq \Hom(x^\ast\mathcal{F}_1, x^\ast\mathcal{F}_2)$  (see \cite[Lemma I.4.9]{HDR}).
	In particular, any  point of $\M_\C^{ext, [0,0]}$ defines a point of $\mathbb{V}(\E)$, so that we have an induced morphism $\phi \colon \M_\C^{ext, [0,0]} \to \mathbb{V}(\E)$.
	This morphism induces an equivalence
\[
\Map_{\big/(\M_\C^{[0,0]} \times \M_\C^{[0,0]})}(\Spec (A), \M_\C^{ext, [0,0]}) \to \Map_{\big/(\M_\C^{[0,0]} \times \M_\C^{[0,0]})}(\Spec (A), \mathbb{V}(\E))
\]
for any $\Spec (A) \to \M_\C^{[0,0]} \times \M_\C^{[0,0]}$ as shown by \eqref{Hom_lin}.
	Thus we have a commutative diagram
\[\begin{tikzcd}[sep=small]
	{\M_\C^{[0,0], ext}} &&&& {\mathbb{V}(\E)} \\
	\\
	&& {\M_\C^{[0,0]} \times\M_\C^{[0,0]}}
	\arrow["\phi", from=1-1, to=1-5]
	\arrow["{(\ev_1, \ev_3)}"', from=1-1, to=3-3]
	\arrow["p", from=1-5, to=3-3]
\end{tikzcd}\]
where the horizontal arrow is an equivalence, proving the statement.
\end{proof}

\begin{construction}\label{construction_ext}
	Let $\C \in \Prlo_k$ of finite type equipped with an admissible $t$-structure $\tau$ satisfying \cref{Assumption_subobject} and universally satisfying openness of flatness.
	Let $\Spec(\kappa) \to \Spec(k)$ be a closed point with $\kappa$ algebraically closed and let $r, N \in \N$.
	Let $\underline{F} = \left\lbrace F_1, \cdots F_r\right\rbrace \subset \C_\kappa^\heartsuit$ be a finite set of pseudo-perfect semisimple objects. 
	Notice that the objects in $\underline{F}$ are automatically $\tau_{\kappa}$-flat (\cref{tau_flat_in_heart}) and that the corresponding points of $\M_\C^\heartsuit$ are closed (\cref{closed_point_is_semisimple}).
	Denote by $\mathcal{G}_{F_i}$ the associated residual gerbes and let $q_i \colon \mathcal{G}_{F_i} \hookrightarrow \M_\C^\heartsuit$ be the closed immersion (\cite[\href{https://stacks.math.columbia.edu/tag/0H27}{Lemma 0H27}]{stacks-project}). By \cite[Théorème 11.3]{LMB}, the stacks $\mathcal{G}_{F_i}$ are of finite type over a field.
	Consider the diagram with pullback square
\[\begin{tikzcd}[sep=small]
	{\M_\C^{\heartsuit, ext, (F_{n_1}, F_{n_2})}} && {\M_\C^{\heartsuit, ext}} && {\M_\C^{\heartsuit}} \\
	\\
	{\mathcal{G}_{F_{n_1}} \times \mathcal{G}_{F_{n_2}}} && {\M_\C^{\heartsuit} \times \M_\C^{\heartsuit}}
	\arrow["{q'}", from=1-1, to=1-3]
	\arrow["{\ev_2^{(2)}}", bend left, from=1-1, to=1-5]
	\arrow["p"', from=1-1, to=3-1]
	\arrow["{\ev_2}", from=1-3, to=1-5]
	\arrow["{(\ev_1, \ev_3)}", from=1-3, to=3-3]
	\arrow["{(q_1,q_2)}"', from=3-1, to=3-3]
\end{tikzcd}\]
	The projection $p$ is quasi-compact as it is the pullback of a quasi-compact morphism (\cref{qc_map}).
	Hence $\M_\C^{\heartsuit, ext,  (F_{n_1}, F_{n_2})}$ is an algebraic stack of finite presentation.
	Since $\M_\C^\heartsuit$ has affine diagonal by \cref{affine_diagonal}, it follows from \cite[\href{https://stacks.math.columbia.edu/tag/075S}{Lemma 075S}]{stacks-project} that the map $\ev_2^{(2)} \colon \M_\C^{\heartsuit, ext, (F_{n_1}, F_{n_2})} \to \M_\C^\heartsuit$ is quasi-compact.
	We then define the closed substack
\[
\M_\C^{\heartsuit, (F_{n_1}, F_{n_2})} \xhookrightarrow{i^{(2)}} \M_\C^\heartsuit
\]
to be the scheme-theoretic image of $\ev_2'$ (this is well defined by \cite[\href{https://stacks.math.columbia.edu/tag/081I}{Lemma 081I}]{stacks-project}).
	Since $q'$ is the pullback of a closed immersion, it is proper.
	Since moreover $\ev_2$ satisfies the valuative criterion for properness (\cref{proper_ev}), it follows that $\ev_2^{(2)}$ is a quasi-compact morphism satisfying the valuative criterion for properness, thus it is proper.
	In particular $\ev_2^{(2)}$ is (universally) closed. Hence $\M_\C^{\heartsuit, (F_{n_1}, F_{n_2})}$ is quasi-compact by \cite[\href{https://stacks.math.columbia.edu/tag/02JQ}{Lemma 02JQ}]{stacks-project}.\\ \indent
	We then proceed by induction.
	Suppose we have defined the closed substack
\[
i^{(s-1)} \colon \M_\C^{\heartsuit, \underline{F}''} \to \M_\C^\heartsuit
\]
where $F''$ is the ordered set $(F_{n_1}, \cdots F_{n_{s-1}})$. We define the closed substack $\M_\C^{\heartsuit, \underline{F}'}$ to be the image of the morphism $\ev_2^{(s)}$ defined by the diagram with pullback square
\[\begin{tikzcd}[sep=small]
	{\M_\C^{\heartsuit, \underline{F}'}} && {\M_\C^{\heartsuit, ext}} && {\M_\C^{\heartsuit}} \\
	\\
	{\M_\C^{\heartsuit, \underline{F}''} \times \mathcal{G}_{F_{n_s}}} && {\M_\C^{\heartsuit} \times \M_\C^{\heartsuit}}
	\arrow["{q''}", from=1-1, to=1-3]
	\arrow["{\ev_2^{(s)}}", bend left, from=1-1, to=1-5]
	\arrow["p"', from=1-1, to=3-1]
	\arrow["{\ev_2}", from=1-3, to=1-5]
	\arrow["{(\ev_1, \ev_3)}", from=1-3, to=3-3]
	\arrow["{(i', q_3)}"', from=3-1, to=3-3]
\end{tikzcd}\]
	The same proof as above shows that $\M_\C^{\heartsuit, \underline{F}'}$ is a well-defined quasi-compact closed substack of $\M_\C^\heartsuit$. Finally, we define the substack
\[
\M_\C^{\heartsuit, \underline{F}, N} \subset \M_\C^\heartsuit
\]
to be the union of the $\M_\C^{\heartsuit, \underline{F}'}$ constructed above.
\end{construction}

\begin{remark}\label{quasi_compact_ext}
	In the setting of \cref{construction_ext}, the stack $\M_\C^{\heartsuit, \underline{F}, N}$ is a finite union of quasi-compact closed substack of $\ \M_\C^\heartsuit$.
	Hence $\M_\C^{\heartsuit, \underline{F}, N}$ is itself a quasi-compact closed substack of $ \M_\C^\heartsuit$.
	Moreover, its closed $\kappa$-points are in bijection with families of $\C_\kappa^\heartsuit$ that are successive extensions of length at most $N$ of elements of $\underline{F}$.
\end{remark}

\begin{remark}
	In \cref{construction_ext} we suppose that $\kappa$ is algebraically closed field only to ensure that semisimple objects corresponds to closed points of $\M_\C^\heartsuit$.
	If we only suppose that $\kappa$ is a field, we should ask that the objects of $\underline{F}$ define closed points of $\M_\C^\heartsuit$. Note that a closed point must corresponds to a semisimple pseudo-perfect object by \cref{closed_point_is_semisimple_preliminar}.
\end{remark}

\begin{theorem}\label{good_ext}
	In the setting of \ref{construction_ext}, the stack $\M_\C^{\heartsuit, \underline{F}, N}$ admits a separated good moduli space $\mathrm{M}_\C^{\heartsuit, \underline{F}, N}$. Moreover, the $\kappa$-points of $\mathrm{M}_\C^{\heartsuit, \underline{F}, N}$ are in bijection with direct sum of at most $N$ objects in $\underline{F}$, possibly with repetitions.
\end{theorem}

\begin{proof}
	The statement follows by combining \cref{quasi_compact_ext} with \cref{good_general} and \cref{closed_point_is_semisimple}.
\end{proof}

\section{Good moduli spaces for perverse sheaves}\label{section_good_Perv}

    In this section we specialize the results of \cref{existence_good_section} to the moduli of perverse sheaves.
    We start our discussion by recalling the definition of perverse sheaves following \cite{BBDG} and the representability results from \cite{Exodromy, exodromyconicality}.

\begin{notation}
	Throughout this section we will work with hypersheaves instead of sheaves.
	In the application we are interested in, there is no difference between the two notions.
	Since there is no risk of confusion, we will not make any reference to this in our notation.
\end{notation}

\begin{definition}
	A stratified space is the data of a continuous map $\rho \colon X \to P$ where $X$ is a topological space and $P$ is a poset endowed with the Alexandroff topology.
	A morphism of stratified spaces $(f,r) \colon (X,P) \to (Y,Q)$ is a commutative diagram
\[\begin{tikzcd}[column sep=scriptsize,row sep=small]
	X && Y \\
	\\
	P && Q
	\arrow["f", from=1-1, to=1-3]
	\arrow["{\rho_x}"', from=1-1, to=3-1]
	\arrow["{\rho_Y}", from=1-3, to=3-3]
	\arrow["r"', from=3-1, to=3-3]
\end{tikzcd}\]
	Abusing notations, we will often refer to a stratified space as a pair $(X,P)$.
	For $p \in P$, the $p$-stratum of $X$ is $X_p \coloneqq \rho^{-1}(p) \subset X$.
\end{definition}
	 
\begin{definition}[{\cite[Definition A.5.3]{HA}}]
	Let $(X,P)$ be a stratified space.
	The cone $(C(X), P \sqcup \left\lbrace - \infty \right\rbrace$) is the stratified space defined by:
	\begin{enumerate}\itemsep=0.2cm
    	\item as a set, $C(X) = \left\lbrace \ast \right\rbrace \sqcup (X \times \mathbb{R}_{>0}) $;
		\item a subset $U \subset C(X)$ is open if and only if $U \cap X \times \mathbb{R}_{>0}$ is open, and if $\ast \in U$ then there exist $\epsilon >0$ such that $X \times (0, \epsilon) \subset U$;
		\item the stratification of $C(X)$ is given by the map $\overline{\rho} \colon X \to P \sqcup \left\lbrace - \infty \right\rbrace$ defined by $\overline{\rho}(\ast) = - \infty$ and $\overline{\rho}(x,t) = \rho(x)$ for $(x,t) \in X \times \mathbb{R}_{>0}$.
	\end{enumerate}
\end{definition}

\begin{definition}[{\cite[Definition A.5.5]{HA}}]
	Let $(X,P)$ be a stratified space.
	We say that $(X,P)$ is conically stratified if for every $p \in P$ and $x \in X_p$, there exists a stratified space $(Y, P_{>p})$ and a topological space $Z$ such that there exists an open immersion $j \colon C(Y) \times Z \to X$ for which $x$ lies in the image of $j$ and the following diagram commute:
\[\begin{tikzcd}[sep=small]
	{C(Y) \times Z} && X \\
	{C(Y)} \\
	{P_{\geq p}} && P
	\arrow["j", hook, from=1-1, to=1-3]
	\arrow["{\pi_1}"', from=1-1, to=2-1]
	\arrow["{\rho_X}", from=1-3, to=3-3]
	\arrow["{\overline{\rho}_Y}"', from=2-1, to=3-1]
	\arrow[from=3-1, to=3-3]
\end{tikzcd}\]
\end{definition}

\begin{definition}
    Let $(X,P)$ be a stratified space.
	We say that $(X,P)$ is conically refineable if there exists a conical stratification $(X,Q)$ with locally weakly contractible strata equipped with a surjective map of posets $r \colon Q \to P$ such that $(id_X, r) \colon (X,Q) \to (X,P)$ is a morphism of stratified spaces.
\end{definition}

\begin{definition}
	Let $(X,P)$ be a stratified space and let $\E \in \Cat_\infty$. 
	For $p \in P$, denote by $i_p \colon X_p \to X$ the inclusion.
	An object $F \in \Sh(X; \E)$ is constructible if $i_p^\ast F$ is locally constant for every $p \in P$.
	We denote by
	\[
	\Cons_P(X; \E) \subset \Sh(X ; \E)
	\]
	the full subcategory spanned by constructible sheaves.
\end{definition}

\begin{definition}
	Let $(X,P)$ be a stratified space and let $\E \in \Cat_\infty$.
	We let 
\[
\Cons_{P, \omega}(X; \E) \subset \Cons_P(X; \E)
\]
be the full subcategory spanned by constructible sheaves  with compact stalks.
\end{definition}

\begin{recollection} \cite[Corollary 4.1.15]{exodromyconicality}\label{conical_vs_exodromic}
	Let $(X,P)$ be a conically refineable stratified space and let  $\Pi_\infty(X,P)$ be the opposite of the full subcategory of objects $x\in \Cons_P(X; \Spc) $ such that the functor
\[
\Map(x,-) \colon \Cons_P(X; \Spc) \to \Spc
\]
commutes with colimits.
	Then \cite[Corollary 4.1.15]{exodromyconicality} shows that for every  $\E \in \Prlo$, there is a canonical equivalence 
\[
\Cons_P(X; \E) \simeq \Fun(\Pi_\infty(X,P), \E).
\]
\end{recollection}

\begin{example}\label{example_exodromic}
	By \cite[Theorem 5.4.1]{Exodromy}, if $(X,P)$ is a conically stratified space with locally weakly contractible strata, then $\Pi_\infty(X,P)$ agrees with Lurie's simplicial set of exit paths (see \cite[Definition A.6.2 \& Theorem A.6.4]{HA}).
	By \cite[Theorem 0.3.1-(3)]{exodromyconicality}, if $(X,P)$ admits a conical refinement $(X,Q)$ with locally weakly contractible strata, then $\Pi_\infty(X,P)$ is a localization of $\Pi_\infty(X,Q)$.
	In particular, the objects of $\Pi_\infty(X,P)$ are exactly the points of $X$.
\end{example}

\begin{definition}[{\cite[Definition 2.2.1]{Exodromy}}]\label{finite_strat}
	Let $(X,P)$ be a conically refineable stratified space.
\begin{itemize}\itemsep=0.2cm
    \item[(i)] We say that $(X,P)$ is categorically compact if $\Pi_\infty(X,P) \in \Cat_\infty^\omega$.
    \item[(ii)] We say that $(X,P)$ is locally categorically compact if $X$ admits a fundamental system of open subsets $\left\lbrace U_i \right\rbrace_{i \in I}$ such that each $(U_i,P)$ is categorically compact.
\end{itemize}
\end{definition}

\begin{example}\label{algebraic_strat}
	Real algebraic varieties with finite stratification by Zariski locally closed subsets are conically refineable, categorically compact, locally categorically compact and have locally weakly contractible strata by \cite[Theorem 2.4.12 \& Remark 2.4.13]{Perv_moduli} (see also \cite[Theorem 5.3.13]{exodromyconicality}).
\end{example}

\begin{definition}
  Let $X$ be a topological space.
  We say that a stratification $X \to P$ is locally finite if for every point $x \in X$, there is an open neighborhood $U$ of $x$ such that $U$ intersects only finitely many strata of $(X,P)$. 
\end{definition} 

\begin{example}\label{subanalytic_strat}
	Real compact subanalytic spaces with locally finite stratification by subanalytic subsets are conically refineable, categorically compact, locally categorically compact and have locally weakly contractible strata by \cite[Theorem 2.4.10 \& Remark 2.4.13]{Perv_moduli} (see also \cite[Theorem 5.3.9]{exodromyconicality}).
\end{example}

\begin{definition}
	A complex subanalytic stratified space is the data of a triple $(M,X,P)$ where $M$ is a smooth complex analytic space, $X \subset M$ is a locally closed complex subanalytic subset, and $X \to P$ is a locally finite stratification by complex subanalytic subsets of $M$.
\end{definition}

\begin{definition}
	A complex algebraic stratified space is the data of a stratified space $(X,P)$ where $X$ is (the complex points of) an algebraic variety over $\mathbb{C}$ and $X \to P$ is a finite stratification by Zariski locally closed subsets.
\end{definition}

\begin{definition}\label{Whitney_strat_def}
	Let $A$ be a smooth manifold. 
	A stratified space $(X, P)$ on a subset $X$ of $A$ is said to be a Whitney stratified space if the following conditions hold:
	\begin{enumerate}\itemsep=0.2cm
    \item $P$ is finite;
    \item for every $p \in P$, the stratum $X_p$ is smooth;
    \item each point has a neighborhood intersecting only a finite number of strata;    	\item if $p \in P$ and $\overline{X_p}$ is the closure of $X_p$ in $A$, the intersection $(X_p \setminus \overline{X_p}) \cap X$ is a union of strata;
    \item for every $p,q \in P$, $p < q$, the pair $(X_p, X_q)$ satisfies Whitney's conditions $B$ (see e.g. \cite[Section 1.2]{GM}).
\end{enumerate}
\end{definition}

\begin{theorem}[{\cite[Theorem 1.1 \& Proposition 2.2]{NV}}]\label{Whitney_vs_conical}
	Whitney stratified spaces are conically stratified.
\end{theorem}

\begin{definition}\label{Withney_manifold_def}
	A complex Whitney stratified space is a stratified space $(X,P)$ such that:
\begin{enumerate}\itemsep=0.2cm
	\item $(X,P)$ is a complex algebraic stratified space or a compact complex subanalytic stratified space;
	\item $(X,P)$ is a Whitney stratified space.
\end{enumerate}
	If moreover $X$ is smooth, we say that $(X,P)$ is a complex Whitney stratified manifold.
\end{definition}

\begin{theorem}\label{Whitney_manifold}
	Complex Whitney stratified spaces have locally weakly contractible strata and are conically stratified, locally categorically compact and categorically compact.
\end{theorem}

\begin{proof}
		Complex Whitney stratified spaces are conically stratified by \cref{Whitney_vs_conical}.
	By \cite[Theorem 2.4.10 \& Remark 2.4.13]{Perv_moduli} (see also \cite[Theorem 5.3.9]{exodromyconicality}), complex subanalytic stratified spaces have locally weakly contractible strata and are locally categorically compact, and categorically compact if $X$ is compact.
	Suppose now that $(X,P)$ is a complex algebraic stratified space with $X$ affine.
	Then $(X,P)$ is a complex subanalytic stratified space, hence locally categorically compact with locally weakly contractible strata.
	By \cite[Proposition 2.2 \& Proposition 2.6]{Weil_restr_note}, we can suppose that $X$ is the (real points of) an affine variety over $\mathbb{R}$ and $X \to P$ is a finite stratification by Zariski locally closed subsets.
	In this case $(X,P)$ is categorically compact by \cite[Theorem 5.3.13]{exodromyconicality}.
	In general, a complex algebraic stratified space $(X,P)$ can be covered by affine charts.
	Hence $(X,P)$ have locally weakly contractible strata and is locally categorically by the above discussion.
	Since $X$ admits a finite cover by affine subsets whose iterated intersections are again affine, hence categorically compact, \cite[Theorem 5.1.7-(5)]{exodromyconicality} shows that $(X,P)$ is categorically compact.
\end{proof}

\subsection{Perverse $t$-structures}

    In this paragraph we recall the definition of perverse $t$-structures and prove that these are admissible (\cref{Perv_accessible}) and satisfy \cref{Assumption_subobject} (\cref{Perv_subobjects}).

\begin{definition}\label{def_perv}
	Let $(X,P)$ be a stratified space and $\mathfrak{p} \colon P \to \ZZ$ be a function. 
	Let $R$ be a simplicial commutative ring.
	Consider the pair of full subcategories $\Sh(X;\Mod_R)$ defined by
\[
{}^{\mathfrak p} \Sh(X;\Mod_R)_{\geq 0} \coloneqq \big\{F \in \Sh(X;\Mod_R) \mid \ \forall p \in P \ , \ \pi_i( i_p^{\ast}(F) ) = 0 \text{ for every } i \leq \mathfrak p(p) \big\},
\]
\[
{}^{\mathfrak p} \Sh(X;\Mod_R)_{\leq 0} \coloneqq \big\{F \in \Sh(X;\Mod_R) \mid \ \forall p \in P \ , \ \pi_i( i_p^{!}(F) ) = 0 \text{ for every } i \geq \mathfrak p(p) \big\}.
\]
	The $\infty$-category of perverse sheaves is
\[
{}^\mathfrak{p} \Perv(X; \Mod_R) \coloneqq {}^{\mathfrak p} \Sh(X;\Mod_R)_{\leq 0} \cap {}^{\mathfrak p} \Sh(X;\Mod_R)_{\geq 0}.
\]
\end{definition}

\begin{definition}
    In the setting of \cref{def_perv}, we set
\[
{}^{\mathfrak p} \Cons_P(X;\Mod_R)_{\geq 0} \coloneqq {}^{\mathfrak p} \Sh(X;\Mod_R)_{\geq 0} \cap \Cons_P(X;\Mod_R) \subset \Cons_P(X;\Mod_R).
\]
    We define analogously ${}^{\mathfrak p} \Cons_P(X;\Mod_R)_{\leq 0}$, ${}^{\mathfrak p} \Perv_P(X;\Mod_R)$, ${}^{\mathfrak p} \Cons_{P, \omega}(X;\Mod_R)_{\geq 0}$, ${}^{\mathfrak p} \Cons_{P, \omega}(X;\Mod_R)_{\leq 0}$, ${}^{\mathfrak p} \Perv_{P, \omega}(X;\Mod_R)$.
\end{definition}
	
\begin{recollection}\label{recollements}
	Let $(X,P)$ be a conically stratified space with locally weakly contractible strata and let $S \subset P$ be a closed subset and let $U \coloneqq P \setminus S$.
	Consider the corresponding open and closed immersions $j \colon X_U \hookrightarrow X$ and $i \colon X_S \to X$.
	Let $R$ be a simplicial commutative ring.
	By \cite[Corollary 2.18 \& Proposition 2.26]{Haine} we have a recollement of $\infty$-categories
\[
i_\ast \colon \Sh(X_S; \Mod_R) \rightarrow \Sh_P(X; \Mod_R) \leftarrow \Sh(X_U; \Mod_R) \colon j_\ast
\]
	By \cite[Corollary 6.8.2]{Exodromy}, the above restricts to a recollement of $\infty$-categories
\[
i_\ast \colon \Cons_S(X_S; \Mod_R) \rightarrow \Cons_P(X; \Mod_R) \leftarrow \Cons_U(X_U; \Mod_R) \colon j_\ast
\]
	and, if furthermore $(X,P)$ is locally categorically compact, then the above further restricts to a recollement
\[
i_\ast \colon \Cons_{S, \omega}(X_S; \Mod_R) \rightarrow \Cons_{P, \omega}(X; \Mod_R) \leftarrow \Cons_{U, \omega}(X_U; \Mod_R) \colon j_\ast
\]
\end{recollection}

\begin{recollection}
	Let $\D$ be a stable $\infty$-category and let 
\[
D_F \xhookrightarrow{i_\ast} \D \xhookleftarrow{j_\ast} D_U
\]
be a pair of full subcategories, respectively equipped with $t$-structures $\tau_F, \tau_U$.
	Suppose that $\D$ is a recollement of $\D_U$ and $D_F$, so that there are adjunctions $i^\ast \dashv i_\ast \dashv i^!$ and $j^\ast \dashv j_\ast$. Since the notion of $t$-structure depends only on the underlying homotopy category, it follows from \cite[Theorem 4.1.10]{BBDG} that the pair of full subcategories
\[
\D_{\geq 0} \coloneqq \lbrace d \in \D \mid \ j^\ast d \in (\D_U)_{\geq 0}, i^\ast d \in (\D_F)_{\geq 0} \rbrace,
\]
\[
\D_{\leq 0} \coloneqq \lbrace d \in \D \mid \ j^\ast d \in (\D_U)_{\leq 0}, i^! d \in (\D_F)_{\leq 0}\rbrace,
\]
define a $t$-structure on $\D$.
\end{recollection}

\begin{proposition}\label{perverse_t_structure}
	Let $(X,P)$ be a stratified space with $P$ finite and let $\mathfrak{p} \colon P \to \ZZ$ be a function.
\begin{enumerate}\itemsep=0.2cm
    \item The pair of $\infty$-categories $({}^{\mathfrak p} \Sh(X;\Mod_R)_{\leq 0}, {}^{\mathfrak p} \Sh(X;\Mod_R)_{\geq 0})$ of \cref{def_perv} define a $t$-structure ${}^{\mathfrak{p}} \tau$ on $\Sh(X;\Mod_R)$.
    \item If $(X,P)$ is conically stratified with locally weakly constructible strata, the $t$-structure ${}^{\mathfrak{p}} \tau$ restricts to a $t$-structure on $\Cons_P(X;\Mod_R)$.
    \item If $(X,P)$ is conically stratified with locally weakly constructible strata and locally categorically compact, and if $R$ is a discrete regular noetherian ring, the $t$-structure ${}^{\mathfrak{p}} \tau$ restricts to a $t$-structure on $\Cons_{P, \omega}(X; \Mod_R)$.
\end{enumerate}
\end{proposition}

\begin{proof}
	Point $(1)$ is proven in \cite[Lemma 3.2.2]{Perv_moduli}, where it is shown that the pair of $\infty$-categories $({}^{\mathfrak p} \Sh(X;\Mod_R)_{\leq 0}, {}^{\mathfrak p} \Sh(X;\Mod_R)_{\geq 0})$ is the $t$-structure obtained by successive recollement of the shifted standard $t$-structure on $\Sh(X_p;\Mod_R)$, $p \in P$. 
    Point $(2)$ then follows by \cref{recollements} and the fact that the standard $t$-structure on $\Sh(X_p;\Mod_R)$ restricts to a $t$-structure on $\Loc(X_p;\Mod_R)$, $p \in P$.
	If $R$ is a discrete regular noetherian, the $t$-structure on $\Mod_R$ restricts to a $t$-structure on $\Mod_R^\omega = \Perf(R)$ by \cite[\href{https://stacks.math.columbia.edu/tag/07LT}{Proposition 07LT}]{stacks-project} and \cite[\href{https://stacks.math.columbia.edu/tag/066Z}{Lemma 066Z}]{stacks-project}.
	Hence the standard $t$-structure on $\Loc(X_p; \Mod_R)$ restricts to a $t$-structure on $\Loc_{\omega}(X_p; \Mod_R)$, $p \in P$. 
	Point $(3)$ then follows again by \cref{recollements}.
\end{proof}

\begin{lemma}\label{Perv_accessible}
	Let $(X,P)$ be a locally categorically compact conically stratified space with locally weakly contractible strata.
	Let $\mathfrak{p} \colon P \to \ZZ$ be a function and let $A \in \dAff_k$.
	Then the perverse $t$-structure on $\Cons_P(X; \Mod_A)$ is $\omega$-accessible and non-degenerate.
\end{lemma}

\begin{proof}
	We begin by proving that ${}^{\mathfrak{p}}\tau$ is $\omega$-accessible.
	Let $p \in P$.
	Our assumptions guarantee that $i_p^! \colon \Cons_P(X; \Mod_k) \to \Loc(X_p; \Mod_k)$ commutes with filtered colimits (\cite[Proposition 6.9.2]{Exodromy}). 
	Moreover the standard $t$-structure on $\Loc(X_p; \Mod_k)$ is clearly $\omega$-accessible.
	In particular, for every filtered diagram $\left\lbrace F_i \right\rbrace_{i \in I} \in \Cons_P(X;\Mod_k)$ and every $n \in \ZZ$ we have equivalences
	\[
	\pi_n(i_p^!\colim_{i \in I} F_i) \simeq \pi_n(\colim_{i \in I} i_p^! F_i) \simeq \colim_{i \in I} \pi_n(i_p^!F_i).
	\]
	The result then follows by definition of ${}^{\mathfrak{p}}\Cons_P(X; \Mod_k)_{\leq 0}$.\\ \indent
    We now prove that ${}^{\mathfrak{p}}\tau$ is non-degenerate. Again by \cite[Lemma 7.3.9]{Exodromy} we can assume that $P$ is finite.
    Since the standard $t$-structure on locally constant sheaves is clearly left-complete, it follows that the perverse $t$-structure is also left-complete.
    Moreover, the argument at the bottom of page 56 of \cite{BBDG} implies that the perverse $t$-structure is also right-complete, thus non-degenerate.
\end{proof}

\begin{notation}
	Let $(Y,Q)$ be stratified spaces.
	Let $A \in \dAff_k$.
	Given a functor 
	\[
	f \colon \Cons_Q(Y; \Mod_A) \to \Cons_P(X; \Mod_A),
	\]
	we set
	\[
	{}^{\mathfrak{p}} f \coloneqq {}^{\mathfrak{p}} \pi_0 \circ f \colon \Cons_Q(Y; \Mod_A) \to {}^{\mathfrak{p}}\Perv_P(X; \Mod_A).
	\]
\end{notation}

\begin{proposition}\label{Perv_subobjects}
	Let $R$ be a regular noetherian ring.
	Let $(X,P)$ be a locally categorically compact conically stratified space with locally weakly contractible strata and $P$ finite.
	Let $F\in {}^{\mathfrak{p}}\Perv_{P,\omega}(X)$.
	Then every subobject of $F$ lies in ${}^{\mathfrak{p}}\Perv_{P,\omega}(X)$.
\end{proposition}

\begin{proof}
	Since perfect complexes satisfy the two-out-of-three properties, it is equivalent to show that the quotients of $F$ lie in ${}^{\mathfrak{p}}\Perv_{P,\omega}(X)$.
	Up to replace $X$ by an open subset, we can suppose that $(X,P)$ is categorically compact conical with locally weakly contractible strata.
	In particular, $\Pi_\infty(X,P)$ admits a finite number of objects up to equivalence.
	We argue by induction on the cardinality of $P$.	
    If $P=\left\lbrace \ast \right\rbrace$, the category of perverse sheaves is up to a shift identified with  $\Loc(X; \Mod_R^\heartsuit)$.
	In this case, \cref{Perv_subobjects} boils down to the fact that a quotient of a finitely presented module over a noetherian ring is again finitely presented. We now proceed with the inductive step.\\ \indent
	Let $G$ be a quotient of $F$.
	Consider a maximal element $p\in P$ and let $j \colon X_p \hookrightarrow X$ and $i \colon X\setminus X_p \hookrightarrow X$ be the inclusions.
	The localization sequence \cite[Lemma 1.4.19]{BBDG} attached to the open subset $X_p\subset X$ gives rise to a morphism of exact sequences of perverse sheaves 
\begin{equation}\label{fundamental_sequence}
\begin{tikzcd}[sep=small]
	0 & {i_{\ast}{}^\mathfrak{p}\pi_1(i^\ast F)} & {{}^{\mathfrak{p}} j_!j^\ast F} & F & {i_{\ast}{}^{\mathfrak{p}}i^\ast F} & 0 \\
	0 & {i_{\ast}{}^\mathfrak{p}\pi_1(i^\ast G)} & {{}^{\mathfrak{p}} j_!j^\ast G} & G & {i_{\ast}{}^{\mathfrak{p}}i^\ast G} & 0
	\arrow[from=1-1, to=1-2]
	\arrow[from=1-2, to=1-3]
	\arrow[from=1-2, to=2-2]
	\arrow["\psi_1",from=1-3, to=1-4]
	\arrow[from=1-3, to=2-3]
	\arrow["\varphi_1", from=1-4, to=1-5]
	\arrow[two heads, from=1-4, to=2-4]
	\arrow[from=1-5, to=1-6]
	\arrow[from=1-5, to=2-5]
	\arrow[from=2-1, to=2-2]
	\arrow[from=2-2, to=2-3]
	\arrow["{\psi_2}"', from=2-3, to=2-4]
	\arrow["{\varphi_2}"', from=2-4, to=2-5]
	\arrow[from=2-5, to=2-6]
\end{tikzcd}   
\end{equation}
	Since $R$ is regular noetherian, \cref{perverse_t_structure} implies that the perverse sheaves in the top row of \eqref{fundamental_sequence} have perfect stalks.
	Since $F \to G$ and  $G  \to i_{\ast}{}^{\mathfrak{p}}i^\ast G$   are epimorphisms, so is the right vertical arrow.
	Applying the inductive hypothesis we get that $i_{\ast}{}^{\mathfrak{p}}i^{\ast} G$ has perfect stalks.
	Since $j^\ast$ is ${}^{\mathfrak{p}}\tau$-exact, it preserves epimorphisms.
	In particular $j^\ast F \to j^\ast G$ is an epimorphism.
	Hence, $j^\ast G$ has perfect stalks by the unstratified case already treated.
	By \cref{perverse_t_structure}, we deduce that ${}^{\mathfrak{p}} j_!j^\ast G$ has perfect stalks.
	To conclude, we are left to show that $i_{\ast}{}^\mathfrak{p}\pi_1(i^\ast G)$ has perfect stalks.
	By \cite[1.4.17.1]{BBDG}, we know that $i_{\ast}{}^\mathfrak{p}\pi_1(i^\ast G)$ is a subobject of $i_\ast {}^\mathfrak{p}i^!{}^\mathfrak{p}j_!j^\ast G$.
	Since  ${}^\mathfrak{p}j_!j^\ast G$ has perfect stalks, it follows from \cref{perverse_t_structure} that $i_\ast{}^\mathfrak{p}i^!{}^\mathfrak{p}j_!j^\ast G$ has perfect stalks. 
	Hence $i_{\ast}{}^\mathfrak{p}\pi_1(i^\ast G)$ has perfect stalks by our inductive hypothesis.
\end{proof}

\begin{example}\label{good_refinement_example}
	By \cref{Whitney_manifold}, complex Whitney stratified space satisfy the assumptions of \cref{Perv_subobjects}.
\end{example}

\begin{example}
	By \cref{algebraic_strat}, \cref{subanalytic_strat} and \cref{Whitney_vs_conical}, real algebraic varieties and real compact analytic spaces equipped with Whitney stratification satisfy the assumptions of \cref{Perv_subobjects}.
\end{example}

\subsection{Existence of good moduli spaces for ${}^{\mathfrak{p}}\mathbf{Perv}_P(X)$}\label{subsection_good_Perv}

	In this subsection we specialize the results of \cref{existence_good_section} to moduli spaces of perverse sheaves.
	We also prove that for complex Whitney stratified manifolds, the entire algebraic stack of perverse sheaves admits a good moduli space.

\begin{recollection}
	Let $(X,P)$ be a stratified space and $\mathfrak{p} \colon X \to P$ be a function.
	The derived prestack of constructible sheaves is defined by the assignment
\[
\mathbf{Cons}_P(X) \colon \dAff_k \to \mathrm{Spc}
\]
\[
\Spec (A) \mapsto Cons_{P,\omega}(X; \Mod_A)^{\simeq}
\]
with functoriality given by extension of scalars. 
	If $P$ is finite, we let
\[
{}^{\mathfrak{p}}\mathbf{Perv}_P(X) \subset \mathbf{Cons}_P(X)
\]
the sub-prestack over $k$ spanned by those constructible sheaves that are ${}^{\mathfrak{p}}\tau$-flat as objects of $\Sh(X; \Mod_R)$.
	We will refer to ${}^{\mathfrak{p}}\mathbf{Perv}_P(X)$ as the derived prestack of perverse sheaves.
\end{recollection}

\begin{remark}
    If the perverse $t$-structure restricts to a $t$-structure on the $\infty$-category of constructible sheaves (see e.g. \cref{perverse_t_structure}), then a constructible sheaf is ${}^{\mathfrak{p}}\tau$-flat as a sheaf if and only if it is so as a constructible sheaf.
\end{remark}

\begin{theorem}\label{openness_flatness_perv}
    Let $(X, R)$ be a conically stratified space with locally weakly contractible strata, let $\phi\colon R \to P$ be a map of posets, let $\mathfrak{p} \colon P\to \ZZ$ be a function. 
    Assume that $(X, R)$ is categorically compact.
    Then:
\begin{enumerate}\itemsep=0.2cm
	\item $\Cons_P(X; \Mod_k) \in \Prlo_k$ is of finite type.
    \item There is a canonical equivalence
\[
\M_{\Cons_P(X; \Mod_k)} \simeq \mathbf{Cons}_P(X).
\]
    In particular $\mathbf{Cons}_P(X)$ is a locally geometric derived stack, locally of finite presentation over $k$.
    \item If furthermore $(X, R)$ is locally categorically compact and $R,P$ are finite posets, then the morphism of derived stacks
\[
{}^{\mathfrak{p}}\mathbf{Perv}_P(X) \hookrightarrow \mathbf{Cons}_P(X)
\]
is representable by an open immersion.
    In particular ${}^{\mathfrak{p}}\mathbf{Perv}_P(X)$ is a $1$-Artin stack locally of finite presentation over $k$.
\end{enumerate}
\end{theorem}

\begin{proof}
    This is \cite[Lemma 4.1.12, Lemma 4.1.14 \& Theorem 4.2.10]{Perv_moduli}.
\end{proof}

For $\Spec(\kappa) \to \Spec(k)$ a closed point with $\kappa$ algebraically closed, we have the following

\begin{theorem}\label{good_Perv_non_smooth}
	Let $(X,P)$ be a categorically compact, locally categorically compact conically stratified space with locally weakly contractible strata and $P$ finite.
	Then the algebraic $t_0{}^{\mathfrak{p}}\mathbf{Perv}_P(X)$ is $\Theta$-reductive and $\mathrm{S}$-complete.
	In particular, any closed quasi-compact substack $\X \subset t_0{}^{\mathfrak{p}}\mathbf{Perv}_P(X)$ admits a separated good moduli space $X$ whose $\kappa$-points parametrize semisimple perverse sheaves with perfect stalks lying over $\X$.
	Moreover $X$ admits a natural derived enhancement if $\X$ is the truncation of a substack of ${}^{\mathfrak{p}}\mathbf{Perv}_P(X)$.
\end{theorem}

\begin{proof}
	By \cref{Perv_accessible}, \cref{Perv_subobjects} and \cref{openness_flatness_perv}, the algebraic stack ${}^{\mathfrak{p}}\mathbf{Perv}_P(X)$ satisfies the assumptions of \cref{good_general}.
\end{proof}

\begin{example}
	By \cref{Whitney_manifold}, complex Whitney stratified spaces satisfy the assumptions of \cref{good_Perv_non_smooth}.
\end{example}

\begin{example}
	By \cref{algebraic_strat}, \cref{subanalytic_strat} and \cref{Whitney_vs_conical}, real algebraic varieties and real compact analytic varieties equipped with Whitney stratification satisfy the assumptions of \cref{good_Perv_non_smooth}.
\end{example}

\subsubsection{Good moduli for perverse sheaves on complex Whitney stratified manifolds}
	In this paragraph we construct a good moduli space for the entire algebraic stack perverse sheaves for the middle perversity on a complex Whitney stratified manifolds.
	
\begin{remark}
    Notice that \cref{good_ext} provide a class closed substacks of the stack of perverse sheaves admitting a good moduli space.
    Roughly, these closed substacks parametrize perverse sheaves that can be written as extension of a fixed finite set of semisimple perverse sheaves.
    The good moduli spaces arising in this way are much smaller than the good moduli of perverse sheaves constructed in \cref{good_Perv_algebraic} and define closed algebraic subspace of it by \cite[Lemma 4.14]{Alp}.
\end{remark}

\begin{definition}\label{middle_perversity_def}
    Let $(X,P)$ be a complex Whitney stratified space. 
    The middle perversity is the function
\[
\mathfrak{p} \colon P \to \ZZ
\]
\[
p \mapsto \dim_\mathbb{C} X_p \ .
\]
\end{definition}
	
\begin{definition}
	For $K$ be a field, $P \in \Perf(K)$ and $i \in \N$, we set
\[
h_i(P) \coloneqq \dim_K \pi_i(P), \qquad \chi(F) \coloneqq \sum_{j \in \N} (-1)^j h_j(P).
\]
\end{definition}

\begin{definition}
	Let $n \in \ZZ$.
	Let $\mathbf{Perf}_k^{\left\lbrace \chi=n \right\rbrace} \subset \mathbf{Perf}_k$ be the sub-presheaf
\[
\mathbf{Perf}_k^{\left\lbrace \chi=n \right\rbrace} \colon \dAff_k^{op} \to \mathrm{Spc}
\]
defined by sending $\Spec(A) \in \dAff_k$ to the maximal $\infty$-subgroupoid of $\Perf(A)$ spanned by those perfect complexes $F$ for which the following property holds: for every field $K$ and every $A \to \pi_0(A) \to K$ we have $\chi(F \otimes _A K) = n$.
\end{definition}

\begin{lemma}\label{Perf_chi_representability}
	Let $n \in \ZZ$.
	The the inclusion
\[
\mathbf{Perf}_k^{\left\lbrace \chi=n \right\rbrace} \subset \mathbf{Perf}_k
\]
is representable by an open and closed immersion.
	In particular, the presheaf $\mathbf{Perf}_k^{\left\lbrace \chi=n \right\rbrace}$ is a locally geometric derived stack, locally of finite presentation over $k$.
\end{lemma}

\begin{proof}
	The first claim follows from locally constancy of Euler characteristic (\cite[\href{https://stacks.math.columbia.edu/tag/0BDJ}{Lemma 0BDJ}]{stacks-project}).
	Since $\mathbf{Perf}_k$ is a derived stack over $k$, locally geometric and locally of finite presentation over $k$ by \cite[Proposition 3.7]{TV}, the second claim follows the first.
\end{proof}

\begin{definition}
	Let $X$ be a topological space and $f \colon X \to \ZZ$.
	We say that $f$ is constructible if there exists a stratification $(X,P)$ of $X$ such that $f$ is constant on the strata of $(X,P)$.
\end{definition}

\begin{definition}\label{recollection_euler_CC}
	Let $(X, P)$ be a stratified space and $K$ a field.
	For $F \in \Cons_{P, \omega}(X; \Mod_K)$, define the local Euler-Poincaré index of $F$ by 
\[
\chi(F) \colon X \to \ZZ
\]
\[
x \mapsto \chi(F_x).
\]
\end{definition}

\begin{definition}\label{Cons_chi}
	Let $(X, P)$ be a stratified space and $\chi \colon X \to \ZZ$ be a constructible function.
	Let $K$ be a field.
	We denote by
\[
\Cons_{P, \omega}^{\chi}(X, \Mod_K) \subset \Cons_{P, \omega}(X, \Mod_K)
\]
	the full subcategory spanned by those constructible sheaves $F \in \Cons_{P, \omega}(X, \Mod_K)$ for which $\chi(F) = \chi$.
\end{definition}

\begin{definition}\label{Perv_chi}
	In the setting of \cref{Cons_chi}, let $\mathfrak{p}\colon P \to \ZZ$ be a function.
	We define
\[
{}^{\mathfrak{p}}\Perv_{P, \omega}^{\chi}(X, \Mod_K) \coloneqq {}^{\mathfrak{p}}\Perv_{P, \omega}(X, \Mod_K) \cap \Cons_{P, \omega}^{\chi}(X, \Mod_K) \subset {}^{\mathfrak{p}}\Perv_{P, \omega}(X, \Mod_K).
\]
\end{definition}

\begin{definition}\label{Cons_nu}
	Let $(X, P)$ be a stratified space admitting a categorically compact conical refinement.
	Consider a finite set of representatives $\left\lbrace x_i \right\rbrace_{i \in I}$ for the equivalence classes of $\Pi_\infty(X,P)$.
	Let $\nu \colon \ZZ \to \N$ be function and $K$ a field.
	We denote by
\[
\Cons_{P, \omega}^{\nu}(X, \Mod_K) \subset \Cons_{P, \omega}(X, \Mod_K)
\]
	the full subcategory spanned by those constructible sheaves $F$ with perfect stalks for which $\sum_{i \in I} h^j(F_{x_i}) \leq \nu(j)$ for every $j \in \ZZ$.
\end{definition}

\begin{definition}\label{Perv_nu}
	In the setting of \cref{Cons_nu}, let $\mathfrak{p}\colon P \to \ZZ$ be a function.
	We define
\[
{}^{\mathfrak{p}}\Perv_{P, \omega}^{\nu}(X, \Mod_K) \coloneqq {}^{\mathfrak{p}}\Perv_{P, \omega}(X, \Mod_K) \cap \Cons_{P,\omega}^{\nu}(X, \Mod_K) \subset {}^{\mathfrak{p}}\Perv_{P, \omega}(X, \Mod_K).
\]
\end{definition}

\begin{remark}
    By \cref{conical_vs_exodromic} and \cref{example_exodromic}, \cref{Cons_nu} and \cref{Perv_nu} do not depend on the choice of  representatives for the equivalence classes of $\Pi_\infty(X,P)$.
\end{remark}
	
\begin{remark}
	Let $(X,P)$ be a stratified space with $X$ an analytic manifold and let $F \in \Cons_{P, \omega}(X; \Mod_\mathbb{C})$.
	Attached to $F$ there is a lagrangian cycle $CC(F) \subset T^\ast(X)$, the characteristic cycle of $F$, which roughly measure how far is $F$ from being locally constant.
	In this paragraph, we will need to invoke the theory of characteristic cycles.
	We will blackbox it and refer to \cite[Chapter IX]{KS} for an introductory exposition.
\end{remark}

\begin{remark}\label{rem_vanishing_cohom}
	Let $(X, P)$ be a Whitney stratified manifold and $\mathfrak{p}$ the middle perversity.
	Let $K$ be a field and $F \in \Perv_P(X; \Mod_K)$.
	Then for every $x \in X$ we have $h_i(F_x)= 0$ for every $i<0$ and $i >\dim_{\mathbb{C}} X$.
	Indeed for $i < 0$ this follows directly by the definition of the connective part of $\Perv_P(X; \Mod_K)$.
	For $i > \dim_{\mathbb{C}} X$, see the argument given in \cite{BBDG}, bottom of page 56.
\end{remark}
	
\begin{theorem}[{Massey}]\label{bound_cohomologies}
	Let $(X, P)$ be a Whitney stratified manifold and let $\mathfrak{p} \colon P \to \ZZ$ be the middle perversity.
	Let $\chi \colon X \to \ZZ$ be a constructible function.
	There exists a function $\nu \colon \ZZ \to \N$ with finite support such that
\[
{}^{\mathfrak{p}}Perv_{P, \omega}^{\chi}(X; \Mod_\mathbb{C}) \subset {}^{\mathfrak{p}}\Perv_{P, \omega}^{\nu}(X; \mathbb{C}).
\]
\end{theorem}

\begin{proof}
	By \cref{rem_vanishing_cohom}, it is enough to show that a choice of $\chi$ bounds the dimensions of the cohomologies $h_i$ of the stalks in the range $0 \leq i \leq \dim_{\mathbb{C}} X$.
	By \cite[Theorem 9.7.11]{KS}, fixing $\chi$ is equivalent to fixing the characteristic cycle.
	Up to taking an open neighborhood of $X$, we can assume $X \subset \mathbb{C}^n$.
	Up to shrinking again $X$, \cite[Proposition 2.7]{Massey} shows that the hypothesis of \cite[Corollary 5.5 \& Theorem 7.5]{Massey} are satisfied.
	Together with \cref{rem_vanishing_cohom}, this shows that a choice of $\chi$ bounds the dimension of the cohomologies of the stalks, hence proving the claim.
\end{proof}

\begin{definition}
	Let $(X, P)$ be a stratified space admitting a categorically compact, locally categorically compact conical refinement with locally weakly contractible strata and let $\mathfrak{p} \colon P \to \ZZ$ be a function.
	Let $\chi \colon X \to \ZZ$ be a constructible function.
	We define the sub-presheaf of ${}^{\mathfrak{p}}\mathbf{Perv}_P(X)$
\[
{}^{\mathfrak{p}}\mathbf{Perv}_P^{\chi}(X) \colon \dAff_k^{op} \to \mathrm{Spc}
\]
by sending $\Spec(A) \in \dAff_k$ to the maximal $\infty$-subgrupoid of $\Cons_{P, \omega}(X; \Mod_A)$ spanned by those objects $F$ for which the following property holds: $F$ is ${}^{\mathfrak{p}}\tau$-flat over $A$ and for every field $K$ and every $A \to \pi_0(A) \to K$ we have $\chi(F_x \otimes_A K) = \chi(x)$ for every $x \in X$.
\end{definition}

\begin{definition}\label{underline_chi_definition}
	Let $(X, P)$ be a stratified space.
	Consider a set of representants $\left\lbrace x_i \right\rbrace_{i \in I}$ of the equivalence classes of $\Pi_\infty(X,P)$ and a set of integers $\left\lbrace \chi_i \right\rbrace_{i \in I}$.
	These data determine a constructible function 
\[
\underline{\chi} \colon X \to \ZZ
\]
\[ x \mapsto \chi_i \ \ \textit{if $[x]= [x_i]$ in $\pi_0(\Pi_\infty(X,P))$.}
\]
\end{definition}

\begin{remark}
	If $(X,P)$ is a conically refineable stratified space, it follows by \cref{conical_vs_exodromic} and \cref{example_exodromic} that the constructible function $\underline{\chi}$ of \cref{underline_chi_definition} do not depend on the choice of the $x_i$'s, but only on the choice of the $\chi_i$'s.
\end{remark}

\begin{lemma}\label{pullback_Perv_chi}
	In the setting of \cref{underline_chi_definition}, suppose that $(X,P)$ is conically refineable.
	Then we have a pullback square of derived prestack over $k$
\[\begin{tikzcd}[sep=small]
	{{}^{\mathfrak{p}}\mathbf{Perv}_P^{\underline{\chi}}(X)} && {{}^{\mathfrak{p}}\mathbf{Perv}_P(X)} \\
	\\
	{\prod_{i \in I} \mathbf{Perf}_k^{\left\lbrace \chi=\chi_i \right\rbrace}} && {\prod_{i \in I} \mathbf{Perf}_k}
	\arrow[from=1-1, to=1-3]
	\arrow[from=1-1, to=3-1]
	\arrow["{\prod_{i \in I} i^\ast_{x_i}}", from=1-3, to=3-3]
	\arrow[from=3-1, to=3-3]
\end{tikzcd}\]
\end{lemma}

\begin{proof}
	By \cref{conical_vs_exodromic} and \cref{example_exodromic}, the local Euler-Poincaré index is constant on the equivalent classes of $\Pi_\infty(X,P)$.
	Hence the presheaf ${}^{\mathfrak{p}}\mathbf{Perv}_P^{\underline{\chi}}(X)$ can be described as follows.
	For $\Spec(A) \in \dAff_k$, we have that $F \in {}^{\mathfrak{p}}\mathbf{Perv}_P^{\underline{\chi}}(X)(\Spec(A))$ if and only if the following property holds: $F \in \Cons_{P, \omega}(X; \Mod_A)$ is ${}^{\mathfrak{p}}\tau_A$-flat and for every field $K$ and every $A \to \pi_0(A) \to K$ we have $\chi(F_{x_i} \otimes _A K) = \chi_i$ for every $i \in I$.
	The claim follows.
\end{proof}

\begin{lemma}\label{representability_Perv_chi}
	In the setting of \cref{underline_chi_definition}, let $(X, P)$ be a stratified space admitting a categorically compact, locally compact refinement with locally weakly contractible strata.
	Then the structural morphism
\[
{}^{\mathfrak{p}}\mathbf{Perv}_P^{\underline{\chi}}(X) \to{}^{\mathfrak{p}}\mathbf{Perv}_P(X)
\]
is representable by an open and closed immersion.
	In particular, the presheaf ${}^{\mathfrak{p}}\mathbf{Perv}_P^{\underline{\chi}}(X)$ is a $1$-Artin stack locally of finite presentation over $k$.
\end{lemma}

\begin{proof}
	By \cite[Theorem 0.3.1-(5)]{exodromyconicality}, the stratified space $(X,P)$ is categorically compact.
	Hence the $\infty$-category $\Pi_\infty(X,P)$ has a finite number of equivalence classes.
	It then follows by \cref{Perf_chi_representability} and \cref{pullback_Perv_chi} that the structural morphism ${}^{\mathfrak{p}}\mathbf{Perv}_P^{\underline{\chi}}(X) \to {}^{\mathfrak{p}}\mathbf{Perv}_P(X)$ is (the pullback of) an open and closed immersion.
	The last statement follows since ${}^{\mathfrak{p}}\mathbf{Perv}_P(X)$ is a $1$-Artin stack locally of finite presentation over $k$ by \cref{openness_flatness_perv}.
\end{proof}

\begin{example}
    By \cref{Whitney_manifold}, \cref{representability_Perv_chi} applies to complex Whitney stratified spaces.
\end{example}

\begin{recollection}\label{weak_gen_Cons}
	Let $(X,P)$ be a stratified space admitting a categorically compact conical refinement $(X,Q)$ with locally weakly contractible strata. 
	By \cite[Theorem 0.3.1-(5)]{exodromyconicality}, the stratified space $(X,P)$ is categorically compact.
	Denote by $I$ the finite set of equivalence classes of $\Pi_\infty(X,P)$ and choose a point $x_i$ for every $i \in I$.
	Let $A \in \dAff_k$ and denote by $j_{x_i} \colon \left\lbrace x_i \right\rbrace \hookrightarrow X$ the inclusions. Consider the adjunction
\[
j_{x_i, \#} \colon \Mod_A \rightleftarrows \Cons_P(X;\Mod_A) \colon j_{x_i}^\ast.
\]
	It follows from \cite[Corollary 6.5.4]{Exodromy} and \cite[Theorem 5.1.7-(2)]{exodromyconicality} that $\bigoplus_{i \in I} j_{x_i, \#} A$ is a compact generator of $\Cons_P(X;\Mod_A)$.
\end{recollection}

\begin{recollection}
	Let $(X, P)$ be a stratified space admitting a categorically compact, locally categorically compact conical refinement with locally weakly contractible strata and let $\mathfrak{p}\colon P \to \ZZ$ be a function.
	Consider a finite set of representatives $\left\lbrace x_i \right\rbrace_{i \in I}$ for the equivalence classes of $\Pi_\infty(X,P)$ and the associated compact generator $\bigoplus_{i \in I} j_{x_i, \#} k$ of $\Cons_P(X; \Mod_k)$ (\cref{weak_gen_Cons}).
	Let $\nu \colon \ZZ \to \N$ be a function with finite support.
	Consider the algebraic stack ${}^{\mathfrak{p}}\mathbf{Perv}_P^{\nu}(X)$ of 
	\cref{nu_substacks} and let $\Spec(A) \in \dAff_k$.
	Unwinding the definitions, we have that $F \in {}^{\mathfrak{p}}\mathbf{Perv}_P^{\nu}(X)$ if and only if the following property holds: $F \in \Cons_{P, \omega}(X; \Mod_A)$ is ${}^{\mathfrak{p}}\tau_A$-flat and for every field $K$ and every $A \to \pi_0(A) \to K$ we have $\sum_{i \in I} h_j(F_{x_i} \otimes _A K) \leq \nu(j)$ for every $j \in \ZZ$.
\end{recollection}

\begin{proposition}\label{Perv_qc_connected_components}
	In the setting of \cref{underline_chi_definition}, let $(X, P)$ be a Whitney stratified manifold and $\mathfrak{p} \colon P \to \ZZ$ the middle perversity.
	There exists a function $\nu \colon \ZZ \to \N$ with finite support such that 
\[
{}^{\mathfrak{p}}\mathbf{Perv}_P^{\underline{\chi}}(X) \subset {}^{\mathfrak{p}}\mathbf{Perv}_P^{\nu}(X).
\]
	In particular, the algebraic stack $t_0{}^{\mathfrak{p}}\mathbf{Perv}_P^{\underline{\chi}}(X)$ is of finite presentation over $k$.
\end{proposition}

\begin{proof}
	Let us assume for the moment that the first claim holds.
	By \cref{representability_Perv_chi}, the algebraic stack $t_0{}^{\mathfrak{p}}\mathbf{Perv}_P^{\underline{\chi}}(X)$ is locally of finite presentation over $k$.
	Hence we are reduced to prove that $t_0{}^{\mathfrak{p}}\mathbf{Perv}_P^{\underline{\chi}}(X)$ is quasi-compact.
	Since by \cref{nu_substacks} the algebraic stack $t_0{}^{\mathfrak{p}}\mathbf{Perv}_P^{\nu}(X)$ is of finite presentation over the noetherian ring $k$, it is noetherian.
	In particular, $t_0{}^{\mathfrak{p}}\mathbf{Perv}_P^{\underline{\chi}}(X)$ is an open substack of a Noetherian algebraic stack.
	Since quasi-compactness of an algebraic stack is a topological property by \cite[\href{https://stacks.math.columbia.edu/tag/04YC}{Lemma 04YC}]{stacks-project}, it follows that $t_0{}^{\mathfrak{p}}\mathbf{Perv}_P^{\underline{\chi}}(X)$ quasi-compact. \\ \indent
	We now show the first claim. 
	Let $\Spec(A) \in \dAff_k$ and consider a morphism $\Spec(A) \to {}^{\mathfrak{p}}\mathbf{Perv}_P^{\underline{\chi}}(X)$ corresponding to a ${}^{\mathfrak{p}}\tau_A$-flat constructible sheaf $F \in \Cons_{P, \omega}(X; \Mod_A)$.
	Consider a morphism $A \to \pi_0(A) \to K$ with $K$ a field.
	Since $k$ contains $\mathbb{Q}$, we can write $K$ as a filtered colimits of its sub-fields $K_j$ that are finitely generated over $\mathbb{Q}$.
	Since ${}^{\mathfrak{p}}\mathbf{Perv}_P^{\underline{\chi}}(X)$ is locally of finite presentation, the composition
\[
\Spec(K) \to \Spec(A) \to {}^{\mathfrak{p}}\mathbf{Perv}_P^{\underline{\chi}}(X)
\]
factors through $\Spec (K_j)$ for some $j$.
	Hence we can assume that $K$ is finitely generated over $\mathbb{Q}$.
	Up to choosing an embedding $K \hookrightarrow \mathbb{C}$, we can assume that $K = \mathbb{C}$.
	The claim then follows directly by \cref{bound_cohomologies}.
\end{proof}

\begin{corollary}\label{Perv_good_components}
	In the setting of \cref{Perv_qc_connected_components}, the connected components of $t_0 {}^{\mathfrak{p}}\mathbf{Perv}_P^{\underline{\chi}}(X)$ are quasi-compact.
\end{corollary}

\begin{proof}
	Keeping the notations of \cref{underline_chi_definition}, we have a decomposition in open and closed substacks
\[
t_0{}^{\mathfrak{p}}\mathbf{Perv}_P(X) = \bigsqcup_{\underline{\chi}} t_0{}^{\mathfrak{p}}\mathbf{Perv}_P^{\underline{\chi}}(X).
\]
	In particular, every connected component is a closed substack of $t_0{}^{\mathfrak{p}}\mathbf{Perv}_P^{\underline{\chi}}(X)$ for some $\underline{\chi}$.
	The claim then follows by \cref{Perv_qc_connected_components}.
\end{proof}

    For $\Spec(\kappa) \to \Spec(k)$ a closed point with $\kappa$ algebraically closed, we have the following

\begin{theorem}\label{good_Perv_algebraic}
    Let $(X, P)$ be a Whitney stratified manifold and $\mathfrak{p} \colon P \to \ZZ$ the middle perversity.
    The algebraic stack $t_0{}^{\mathfrak{p}}\mathbf{Perv}_P(X)$ admits a separated good moduli space $t_0{}^{\mathfrak{p}}\Perv_P(X)$.
	Moreover, the $\kappa$-points of $t_0{}^{\mathfrak{p}}\Perv_P(X)$ parametrize semisimple perverse sheaves with perfect stalks.
	Furthermore, $t_0{}^{\mathfrak{p}}\Perv_P(X)$ admits a natural derived enhancement ${}^{\mathfrak{p}}\Perv_P(X)$ which is a derived good moduli space for ${}^{\mathfrak{p}}\mathbf{Perv}_P(X)$.
\end{theorem}

\begin{proof}
    By \cref{openness_flatness_perv}, \cref{Perv_accessible}, \cref{Perv_subobjects}, \cref{Whitney_manifold} and \cref{Perv_good_components}, the algebraic stack $t_0{}^{\mathfrak{p}}\mathbf{Perv}_P(X)$ satisfies the assumptions of \cref{good_general_qc_components}.
\end{proof}

\appendix

\section{}\label{Appendix_QCoh_ST_Theta}

	In this Appendix we give a description of the derived $\infty$-category of quasi-coherent sheaves on $\ST$.
	We recall that $\ST$ is defined as
\[
\ST\coloneqq \left[\faktor{\Spec (R[s,t] / (st-\pi))}{\G_{m, R}}\right],
\]
where $R$ is a DVR, $\pi$ a uniformizer for $R$ and the action of $\G_m$ is induced by the one on $\A^2$ (where $s$ has weight $1$ and $t$ has weight $-1$).
	The statements in \cref{explicit_description_ST} are of course highly inspired by \cite[Lemma 7.14]{AHLH}, and they summarize all the contents of this Appendix.
	The proof is instead inspired \cite{Moulinos}, where the equivalent statements about $\QCoh([\A^1 / \G_m])$ are proven.

\subsection{A description of $\QCoh(\Theta)$}

	We recall the main results of \cite{Moulinos} for reader's convenience.
	The results are stated over the sphere spectrum $\mathbb{S}$, but they hold for any $\mathbb{E}_\infty$-ring spectrum.
	We denote by $\ZZ$ the poset of integers and by $\ZZ^{ds}$ the underlying set.
	In the following we use the convention introduced in \cref{notation_Rep}.

\begin{theorem}\label{QCoh_Theta}
\hfill
\begin{enumerate}\itemsep=0.2cm
	\item (\cite[Theorem 4.1]{Moulinos} \& \cite[Proposition 3.1.6]{rotation_invariance})
	There is an equivalence
\[
\QCoh(\mathrm{B}\G_m) \simeq \Rep(\ZZ) \simeq \mathrm{CoMod}_{\mathbb{S}[\ZZ]}
\]
of symmetric monoidal $\infty$-categories.
    \item (\cite[Theorem 1.1]{Moulinos}) We have an equivalence of symmetric monoidal categories 
\[
\QCoh([\A^1 / \G_m]) \simeq \Mod_{\mathbb{S}[x]}(\Rep(\ZZ^{ds})) \simeq \Rep(\ZZ) \ .
\]
    \item (\cite[Proposition 6.3]{Moulinos}) The pullback along the open immersion $\Spec (\mathbb{S}) = [\G_m / \G_m] \xhookrightarrow{j} [\A^1 / \G_m]$ induces the functor
    \[
    \colim \colon \Rep(\ZZ) \to \mathrm{Sp}
    \]
    \[
    (\cdots \xrightarrow{x} F_n \xrightarrow{x} F_{n-1} \xrightarrow{x} \cdots) \mapsto \colim(\cdots \xrightarrow{x} F_n \xrightarrow{x} F_{n-1} \xrightarrow{x} \cdots)
    \]
    under the equivalence of point (2).
    \item (\cite[Proposition 6.5]{Moulinos}) The pullback along the closed immersion $\mathrm{B}\G_m = [\Spec(\mathbb{S}) / \G_m] \xhookrightarrow{i} [\A^1 / \G_m]$ induces the functor
    \[
    \Gr \colon \Rep(\ZZ) \to \Rep(\ZZ^{ds})
    \]
    \[
    (\cdots \xrightarrow{x} F_n \xrightarrow{x} F_{n-1} \xrightarrow{x} \cdots) \mapsto \bigoplus_{n \in \ZZ}\cof(F_n \xrightarrow{x} F_{n-1})
    \]
    under the equivalence of points (1) and (2).
    \item (\cite[Proposition 8.2]{Moulinos}) The natural $t$-structure on $\QCoh([\A^1 / \G_m])$ corresponds via the equivalence of point (2) to the standard $t$-structure on $\Rep(\ZZ)$.
    That is, $F_\bullet \in \Rep(\ZZ)$ is (co)connective if and only if $F_n$ is (co)connective for every $n$.
\end{enumerate}
\end{theorem}

\begin{remark}\label{pullback_atlas_Theta}
	Let $q \colon \Spec(\mathbb{S}) \to \mathrm{B}\G_m$ and $p \colon \A^1 \to [\A^1 / \mathrm{G}_m]$ be the atlas maps.
	Via the equivalences of \cref{QCoh_Theta}-(1)-(2), there are commutative diagrams
\[\begin{tikzcd}[sep=small]
	{\QCoh(\mathrm{B\mathrm{G}_m})} && {\Rep(\ZZ^{ds})} & {\QCoh(\A^1 / \G_m)} && {\Mod_{\mathbb{S}[x]}(\Rep(\ZZ^{ds}))} \\
	\\
	{\mathrm{Sp}} && {\mathrm{Sp}} & {\QCoh(\A^1)} && {\Mod_{\mathbb{S}[x]}}
	\arrow["\simeq", from=1-1, to=1-3]
	\arrow["{p^\ast}"', from=1-1, to=3-1]
	\arrow["{\pi_!}", from=1-3, to=3-3]
	\arrow["\simeq", from=1-4, to=1-6]
	\arrow["{p^\ast}"', from=1-4, to=3-4]
	\arrow["{\widetilde{\pi}_!}", from=1-6, to=3-6]
	\arrow["\Id"', from=3-1, to=3-3]
	\arrow["\simeq"', from=3-4, to=3-6]
\end{tikzcd}\]
where the right vertical arrow is given by (respectively, induced by)
\[
\pi_! \colon \Rep(\ZZ^{ds}) \to \mathrm{Sp}
\]
\[
F_\bullet \mapsto \oplus_{n \in \ZZ} F_n \ .
\]
\end{remark}

\begin{notation}\label{notations}
	In what follows we will make extensive use of the following diagram:
\[\begin{tikzcd}[sep=small]
	{\Spec(R[s,t] / (st-\pi))} && {\ST} \\
	\\
	{\A^2_R} && {[\A^2_R / \G_{m,R}]} \\
	\\
	{\Spec(R)} && {\mathrm{B}\G_{m,R}}
	\arrow["{p'}", two heads, from=1-1, to=1-3]
	\arrow["{i'}"', hook, from=1-1, to=3-1]
	\arrow["{\phi'}"', bend right = 45, from=1-1, to=5-1]
	\arrow["i", hook, from=1-3, to=3-3]
	\arrow["\phi", bend left = 60, from=1-3, to=5-3]
	\arrow["{p''}", two heads, from=3-1, to=3-3]
	\arrow["{\phi_1'}"', from=3-1, to=5-1]
	\arrow["{\phi_1}", from=3-3, to=5-3]
	\arrow["p"', two heads, from=5-1, to=5-3]
\end{tikzcd}\]
	All the squares are pullbacks, the horizontal maps are atlases, the top vertical maps are closed immersions and the bottom vertical maps are the natural projections.
\end{notation}

\begin{lemma}
	The pushforward $\phi_{\ast} \colon \QCoh(\ST) \to \QCoh(\mathrm{B}\G_{m,R})$ induces an equivalence
\[
\QCoh(\ST) \simeq \Mod_{\phi_{\ast} \OO_{\ST}}(\mathrm{B}\G_{m,R})
\]
of symmetric monoidal $\infty$-categories.
\end{lemma}

\begin{remark}
	The proof is a translation in our setting of the discussion at the beginning of \cite[Section 5]{Moulinos}.
\end{remark}

\begin{proof}
	By \cite[Proposition 3.2.5]{Tannaka_QCoh}, it is enough to show that 
\[
\Spec(R[s,t] / (st-\pi)) \to \ST
\]
is a quasi-affine morphism.
	By \cite[Remark 3.1.29]{Tannaka_QCoh}, it is enough to show it after pulling back along atlas of $\mathrm{B}\G_{m,R}$.
	The conclusion follows from the outer pullback diagram of \cref{notations}.
\end{proof}

	We now identify the graded $R$-module of $\phi_{\ast} \OO_{\ST}$ under the equivalence of \cref{QCoh_Theta}-(1).
	This will identify $\QCoh(\ST)$ as a category of modules in $\Rep(\ZZ^{ds}; \Mod_R)$.
	We begin by identifying the underlying $R$-module.

\begin{lemma}\label{underlying_R_mod}
	The underlying $R$-module of $\phi_{\ast} \OO_{\ST}$ is $R[s,t] / (st-\pi)$.
\end{lemma}

\begin{proof}
	Consider the diagram of \cref{notations}.
	By \cite[Corollary 3.2.6]{Tannaka_QCoh}, the Beck-Chevalley transformation $p^\ast\phi_{\ast} \to (\phi')_{\ast}(p')^\ast$ is an equivalence.
	It then follows that the underlying $R$-module of $\phi_{\ast} \OO_{\ST}$ is identified with
\[
p^\ast\phi_{\ast} \OO_{\ST} \simeq (\phi')_{\ast}(p')^{\ast} \OO_{\ST} \simeq (\phi')_{\ast}\OO_{\Spec(R[s,t] / (st-\pi))} \simeq R[s,t] /(st-\pi),
\]
where the last equivalence follows by $\Spec(R[s,t] / (st-\pi))$ being affine.
\end{proof}

	Now we compute the graded structure.
	In order to do so, we notice that by \cref{notations}, the graded structure is induced by the one on $\A^2$.
	Thus it is enough to compute it there and this can be done already at the non-commutative level.

\begin{lemma}\label{graded_structure}
	Via the equivalence of \cref{QCoh_Theta}-(1) the spectral quasi-coherent sheaf $(\phi_1)_{\ast} \OO_{\A^2}$ is identified with $\mathbb{S}[t,s]$, where $t$ has weight $-1$ and $s$ has weight $1$.
\end{lemma}

\begin{proof}
	By definition, $\mathbb{S}[t,s]$ is the $\mathbb{S}[\ZZ]$-comodule 
\[
\mathbb{S}[\N \times \N] \to \mathbb{S}[\N \times \N] \otimes_{\mathbb{S}} \mathbb{S}[\ZZ]
\]
arising from the natural map of monoids
\[
\N \times \N \to \N \times \N \times \ZZ
\]
\[
(n,m) \mapsto (n,m, m-n).
\]
	The comodule structure of on $\phi_{\ast}\OO_{\A^2}$ arises from the action map $\A^2 \times \G_m \to \A^2$, which is induced by the same map of monoids.
	Thus the equivalence of \eqref{underlying_R_mod} can be promoted to an equivalence of $\mathbb{S}[\ZZ]$-comodules as stated.
\end{proof}

\begin{corollary}\label{QCoh_ST}
	There is an equivalence
\[
\QCoh(\ST) \simeq \Mod_{R[s,t] / (st-\pi)} (\Rep(\ZZ^{ds}; \Mod_R))
\]
of symmetric monoidal $\infty$-categories.
\end{corollary}

\begin{proof}
	It follows formally from \cref{underlying_R_mod} and \cref{graded_structure}.
	See \cite[Theorem 5.1]{Moulinos}, for a detailed proof.
\end{proof}

\begin{remark}\label{explicit_description}
	Let us describe objects of $\Mod_{R[s,t] / (st-\pi)}(\Rep(\ZZ^{ds}; \Mod_R))$ informally: we have to give ourselves a graded $R$-modules, with an action of $s$ that makes the weight increase by $1$ and an action of $t$ that makes the weight decrease by $1$.
	Moreover, we have the relation $st = \pi = ts$.
	Thus we can represent such an object by
\[
F_\bullet \colon \ \
\begin{tikzcd}[sep=small]
	\cdots & {F_{n+1}} & {F_{n}} & {F_{n-1}} & \cdots
	\arrow["t", shift left, from=1-1, to=1-2]
	\arrow["s", shift left, from=1-2, to=1-1]
	\arrow["t", shift left, from=1-2, to=1-3]
	\arrow["s", shift left, from=1-3, to=1-2]
	\arrow["t", shift left, from=1-3, to=1-4]
	\arrow["s", shift left, from=1-4, to=1-3]
	\arrow["t", shift left, from=1-4, to=1-5]
	\arrow["s", shift left, from=1-5, to=1-4]
\end{tikzcd}
\]
with the relation that composing $s$ with $t$ (and vice-versa) is the endomorphism induced by multiplication by $\pi$.
	This recovers the description of the abelian category $\QCoh(\ST)$ provided by \cite[Lemma 7.14]{AHLH}.
	Moreover the proof of \cite[Proposition 8.2]{Moulinos} carries on in our setting showing under the equivalence of \cref{QCoh_ST} an object $F_\bullet$ is (co)connective if and only if $F_n$ is (co)connective for every $n \in \ZZ$.
\end{remark}

\begin{remark}\label{pullback_atlas_ST}
	Let $p \colon \Spec(R[s,t]/(st-\pi)) \to \ST$ be the atlas map.
	Via the equivalence of \cref{QCoh_ST} there is a commutative diagram
\[\begin{tikzcd}[sep=small]
	{\QCoh(\ST)} && {\Mod_{R[s,t] / (st-\pi)} (\Rep(\ZZ^{ds}; \Mod_R))} \\
	\\
	{\QCoh(\Spec(R[s,t] / (st-\pi))} && {\Mod_{R[s,t] / (st-\pi)}}
	\arrow["\sim", from=1-1, to=1-3]
	\arrow["{p^\ast}"', from=1-1, to=3-1]
	\arrow["{\pi_!}", from=1-3, to=3-3]
	\arrow["\sim"', from=3-1, to=3-3]
\end{tikzcd}\]
where the right vertical arrow is induced by
\[
\pi_! \colon \Rep(\ZZ^{ds}) \to \mathrm{Sp}\]
\[
F_\bullet \mapsto \bigoplus_{n \in \ZZ} F_n \ .
\]
\end{remark}

\begin{corollary}\label{computation_pullback_atlas_ST}
	Let $\ev_0 \colon \Mod_R(\Rep(\ZZ^{ds})) \to Mod_R$ be the evaluation at $0$ functor and let $\ev_{0,!} \Mod_R \to \Mod_R(\Rep(\ZZ^{ds}))$ be its left adjoint.
	In the setting of \cref{pullback_atlas_ST}, we have a commutative diagram
\[\begin{tikzcd}[sep=small]
	{\QCoh(\ST)} & {\Mod_{R[s,t] / (st-\pi)} (\Rep(\ZZ^{ds}; \Mod_R))} && {\Mod_R(\Rep(\ZZ^{ds}))} \\
	\\
	{\QCoh(\Spec(R[s,t] / (st-\pi))} & {\Mod_{R[s,t] / (st-\pi)}} && {\Mod_R}
	\arrow["\sim", from=1-1, to=1-2]
	\arrow["{{p^\ast}}"', from=1-1, to=3-1]
	\arrow["{{\pi_!}}", from=1-2, to=3-2]
	\arrow["{- \otimes_R R[s,t]/(st-\pi)}"', from=1-4, to=1-2]
	\arrow["\sim"', from=3-1, to=3-2]
	\arrow["{\ev_{0, !}}"', from=3-4, to=1-4]
	\arrow["{- \otimes_R R[s,t]/(st-\pi)}", from=3-4, to=3-2]
\end{tikzcd}\]
\end{corollary}

\begin{proof}
	By  \cref{pullback_atlas_ST}, it is enough to show that the right square commutes. Passing to right adjoints, it is enough to show that the square
\[\begin{tikzcd}[sep=small]
	{\Mod_{R[s,t] / (st-\pi)} (\Rep(\ZZ^{ds}; \Mod_R))} && {\Mod_R(\Rep(\ZZ^{ds}))} \\
	\\
	{\Mod_{R[s,t] / (st-\pi)}} && {\Mod_R}
	\arrow["{U'}", from=1-1, to=1-3]
	\arrow["{\ev_0^\ast}", from=1-3, to=3-3]
	\arrow["{\pi^\ast}", from=3-1, to=1-1]
	\arrow["U"', from=3-1, to=3-3]
\end{tikzcd}\]
commutes, where $U$ and $U'$ are the forgetful functors.
	Let $M \in \Mod_{R[s,t]/(st-\pi)}$.
	The composition $U' \circ \pi^\ast(M)$ is identified with the diagonal functor $\underline{M}$ defined by $n \mapsto U(M)$ for all $n \in \ZZ^{ds}$.
	The claim follows by evaluating $\underline{M}$ at zero.
\end{proof}

\subsection{Operations on $\QCoh(\ST)$}

	We now want to describe pullback along the open and closed immersion defined by the loci $\left\lbrace t \neq 0 \right\rbrace$, $\left\lbrace s \neq 0 \right\rbrace$, $\left\lbrace t = 0 \right\rbrace$ and $\left\lbrace s = 0 \right\rbrace$ of $\ST$.

\begin{proposition}\label{generic_points}
	Let $j_{\left\lbrace s \neq 0\right\rbrace}, j_{\left\lbrace t \neq 0\right\rbrace} \colon \Spec (R)  \hookrightarrow \ST$ be the open immersions.
	We have commutative diagrams
\[\begin{tikzcd}[column sep=tiny, row sep=small]
	{\QCoh(\ST)} & {\QCoh(\Spec (R) )} & {\QCoh(\ST)} & {\QCoh(\Spec (R) )} \\
	\\
	{\Mod_{R[s,t]/(st-\pi)}(\Rep(\ZZ^{ds}))} & {\Mod_R} & {\Mod_{R[s,t]/(st-\pi)}(\Rep(\ZZ^{ds}))} & {\Mod_R}
	\arrow["{j^\ast_{\left\lbrace s \neq 0 \right\rbrace}}", from=1-1, to=1-2]
	\arrow[from=1-1, to=3-1]
	\arrow[from=1-2, to=3-2]
	\arrow["{j^\ast_{\left\lbrace t \neq 0 \right\rbrace}}", from=1-3, to=1-4]
	\arrow[from=1-3, to=3-3]
	\arrow[from=1-4, to=3-4]
	\arrow["{\colim_s}"', from=3-1, to=3-2]
	\arrow["{\colim_t}"', from=3-3, to=3-4]
\end{tikzcd}\]
where the vertical arrows are equivalences and the bottom horizontal arrows are defined respectively by
\[
F_\bullet \mapsto \colim(\cdots \xleftarrow{s} F_n \xleftarrow{s} F_{n-1} \xleftarrow{s} \cdots)
\]
\[
F_\bullet \mapsto \colim(\cdots \xrightarrow{t} F_n \xrightarrow{t} F_{n-1} \xrightarrow{t} \cdots )
\]
\end{proposition}

\begin{proof}
	We prove the commutativity of the second diagram, the other being analogue.
	Let $F_\bullet \in \Mod_{R[s,t]/(st-\pi)}(\Rep(\ZZ^{ds}))$ and set $\overline{R} \coloneqq R[s,t]/(st-\pi)$.
	Then we have
	\[
	R[t^{\pm 1}] \simeq \overline{R}_t \simeq \colim (\cdots \overline{R} \xrightarrow{t} \overline{R} \xrightarrow{t} \cdots)
	\]
	It follows that under the equivalences
	\[
	\QCoh(\ST|_{\left\lbrace t \neq 0 \right\rbrace}) \simeq \QCoh([\G_{m,R}/ \G_{m,R}]) \simeq \QCoh(\Spec (R) ),
	\]
	the pullback $j^\ast_{\left\lbrace t \neq 0 \right\rbrace} \colon \QCoh(\ST) \to \QCoh(\ST|_{\left\lbrace t \neq 0 \right\rbrace}) $ then corresponds to weight $0$ part of
	\[
	F_\bullet \otimes_{\overline{R}} \overline{R}_t \simeq \colim (\cdots F_\bullet \xrightarrow{t} F_\bullet \xrightarrow{t} \cdots).
	\]
	The claim follows.
\end{proof}

\begin{proposition}\label{closed_points}
	Let $\kappa = R / (\pi)$ and let $i_{\left\lbrace s = 0\right\rbrace}, i_{\left\lbrace t = 0\right\rbrace} \colon \mathrm{B}\G_{m, \kappa} \hookrightarrow \ST$ be the closed immersions.
	We have commutative diagrams
\[\begin{tikzcd}[sep=tiny]
	{\QCoh(\ST)} & {\QCoh(\mathrm{B}\G_{m, \kappa})} & {\QCoh(\ST)} & {\QCoh(\mathrm{B}\G_{m, \kappa})} \\
	\\
	{\Mod_{R[s,t]/(st-\pi)}(\Rep(\ZZ^{ds}))} & {\Mod_R} & {\Mod_{R[s,t]/(st-\pi)}(\Rep(\ZZ^{ds}))} & {\Mod_R}
	\arrow["{i^\ast_{\left\lbrace s = 0 \right\rbrace}}", from=1-1, to=1-2]
	\arrow[from=1-1, to=3-1]
	\arrow[from=1-2, to=3-2]
	\arrow["{i^\ast_{\left\lbrace t = 0 \right\rbrace}}", from=1-3, to=1-4]
	\arrow[from=1-3, to=3-3]
	\arrow[from=1-4, to=3-4]
	\arrow["{\Gr_s}"', from=3-1, to=3-2]
	\arrow["{\Gr_t}"', from=3-3, to=3-4]
\end{tikzcd}\]
where the vertical arrows are equivalences and the bottom horizontal arrows are defined respectively by
\[
F_\bullet \mapsto [ \cdots \xrightarrow{t} F_n / sF_{n-1} \xrightarrow{t} F_{n-1} / sF_{n-2} \xrightarrow{t} \cdots ]
\]
\[
F_\bullet \mapsto [ \cdots \xleftarrow{s} F_n / tF_{n+1} \xleftarrow{s} F_{n-1} / tF_{n} \xleftarrow{s} \cdots ]
\]
\end{proposition}

\begin{proof}
	We prove the commutativity of the second diagram, the other being analogue.
	Let $F_\bullet \in \Mod_{R[s,t]/(st-\pi)}(\Rep(\ZZ^{ds}))$ and set $\overline{R} \coloneqq R[s,t]/(st-\pi)$.
	Then we have
\[
\kappa[s] \simeq \overline{R} / (t) \simeq \cof (\overline{R} \xrightarrow{t} \overline{R})
\]
	It follows that $i^\ast_{\left\lbrace s = 0 \right\rbrace}$ corresponds to
\[
F_\bullet \otimes_{\overline{R}} \kappa[s] \simeq \cof (F_\bullet \xrightarrow{t} F_\bullet).
\]
	The claim follows.
\end{proof}

	Consider now the cover of $\ST \setminus \left\lbrace 0 \right\rbrace$ given by the inclusion of the open substacks
\[
\Spec(R) \xhookrightarrow{j_{\left\lbrace s \neq 0 \right\rbrace}} \ST \setminus \left\lbrace 0 \right\rbrace,
\]
\[
\Spec(R) \xhookrightarrow{j_{\left\lbrace t \neq 0 \right\rbrace}} \ST \setminus \left\lbrace 0 \right\rbrace.\]
	These open substacks intersect in the locus
\[
\Spec(K) \xhookrightarrow{j_{\left\lbrace s,t \neq 0 \right\rbrace}} \ST \setminus \left\lbrace 0 \right\rbrace,
\]
where $K \coloneqq \Frac(R)$.
	Hence we get an equivalence
\begin{equation}\label{QCoh_equiv}
	\QCoh(\ST \setminus \left\lbrace 0 \right\rbrace) \simeq \QCoh(\Spec (R) ) \times_{\QCoh(\Spec(K))} \QCoh(\Spec (R) ).
\end{equation}

	This allows us to prove the following:

\begin{lemma}\label{descent_ST}
	Let $\C \in \Prlo$ equipped with an admissible $t$-structure $\tau$.
	Then \eqref{QCoh_equiv} restricts to an equivalence between $\tau$-flat families over $\ST \setminus \left\lbrace 0 \right\rbrace$ and the full $\infty$-category of $\QCoh(\Spec (R) ) \times_{\QCoh(\Spec (K))} \QCoh(\Spec (R) )$ spanned by the following data:
\begin{itemize}
	\item[(a)] a $\tau$-flat family $F$ over $K$;
	\item[(b)] $\tau$-flat families $E_1, E_2$ over $R$ equipped with injections $E_i \hookrightarrow F$ in $\C_R^\heartsuit$;
	\item[(c)] equivalences $E_i \otimes_R K \simeq F$.
\end{itemize}
\end{lemma}

\begin{proof}
	Since $\tau$-flatness satisfies faithfully flat descent (\cref{descent_flat_affine}), this follows directly from the covering of $\ST \setminus \left\lbrace 0 \right\rbrace$ described above.
	More precisely, let $F_\bullet$  be a $\tau$-flat family over $\ST \setminus \left\lbrace 0 \right \rbrace$.
	Then under the equivalence of \eqref{QCoh_equiv}, the family $F$ corresponds to $(j_{\left\lbrace s,t \neq 0 \right\rbrace})_\C^\ast F_\bullet$, the family $E_1$ to $(j_{\left\lbrace t \neq 0 \right\rbrace})_\C^\ast F_\bullet$ and the family $E_2$ to $(j_{\left\lbrace s \neq 0 \right\rbrace})_\C^\ast F_\bullet$.
\end{proof}

\begin{corollary}\label{pushforward_ST}
	In the setting of \cref{descent_ST}, let $F_\bullet$ be a $\tau$-flat family over $\ST \setminus \left\lbrace 0 \right \rbrace$ and $j \colon \ST \setminus \left\lbrace 0 \right\rbrace \to \ST$ be the open immersion.
	Set $E_1 \simeq (j_{\left\lbrace t \neq 0 \right\rbrace})_\C^\ast F_\bullet$ and $E_2 \simeq (j_{\left\lbrace s \neq 0 \right\rbrace})^\ast F_\bullet$.
	Then $\pi_0 j_{\C, \ast} F_\bullet$ is identified with $\bigoplus_{n \in \ZZ} E_1 \cap (\pi^{-n} E_2)t^n$.
\end{corollary}

\begin{proof}
	Set $F \simeq (j_{\left\lbrace s,t \neq 0 \right\rbrace})_\C^\ast F_\bullet$.
	The underlying $R[s,t]/(st-\pi)$-family of $\pi_0 j_\ast F_\bullet$ is identified with
\[
\mathrm{Ker}(E_1 \oplus E_2 \to F).
\]
	By means of the equivalence $[\G_m / \G_m] \simeq \Spec(R)$, the family $E_1 \in \C_R$ corresponds to the graded object 
\[
E_1 \otimes_R R[t^{\pm 1}] 
\]
	Similarly for $E_2$ and $F$. 
	Hence the pushforward of these families to $\ST$ are identified respectively with:
\[\begin{tikzcd}
	\cdots && F && F && F && \cdots
	\arrow["id", shift left, from=1-1, to=1-3]
	\arrow["\pi", shift left, from=1-3, to=1-1]
	\arrow["id", shift left, from=1-3, to=1-5]
	\arrow["\pi", shift left, from=1-5, to=1-3]
	\arrow["id", shift left, from=1-5, to=1-7]
	\arrow["\pi", shift left, from=1-7, to=1-5]
	\arrow["id", shift left, from=1-7, to=1-9]
	\arrow["\pi", shift left, from=1-9, to=1-7]
\end{tikzcd}\]
\[\begin{tikzcd}
	\cdots && {E_1} && {E_1} && {E_1} && \cdots
	\arrow["id", shift left, from=1-1, to=1-3]
	\arrow["\pi", shift left, from=1-3, to=1-1]
	\arrow["id", shift left, from=1-3, to=1-5]
	\arrow["\pi", shift left, from=1-5, to=1-3]
	\arrow["id", shift left, from=1-5, to=1-7]
	\arrow["\pi", shift left, from=1-7, to=1-5]
	\arrow["id", shift left, from=1-7, to=1-9]
	\arrow["\pi", shift left, from=1-9, to=1-7]
\end{tikzcd}\]
\[\begin{tikzcd}
	\cdots && {E_2} && {E_2} && {E_2} && \cdots
	\arrow["\pi", shift left, from=1-1, to=1-3]
	\arrow["id", shift left, from=1-3, to=1-1]
	\arrow["\pi", shift left, from=1-3, to=1-5]
	\arrow["id", shift left, from=1-5, to=1-3]
	\arrow["\pi", shift left, from=1-5, to=1-7]
	\arrow["id", shift left, from=1-7, to=1-5]
	\arrow["\pi", shift left, from=1-7, to=1-9]
	\arrow["id", shift left, from=1-9, to=1-7]
\end{tikzcd}\]
Using that $E_2 \subseteq F$, the last family is equivalent to
\[\begin{tikzcd}
	\cdots && {\pi E_2} && {E_2} && {\pi^{-1}E_2} && \cdots
	\arrow["id", shift left, from=1-1, to=1-3]
	\arrow["\pi", shift left, from=1-3, to=1-1]
	\arrow["id", shift left, from=1-3, to=1-5]
	\arrow["\pi", shift left, from=1-5, to=1-3]
	\arrow["id", shift left, from=1-5, to=1-7]
	\arrow["\pi", shift left, from=1-7, to=1-5]
	\arrow["id", shift left, from=1-7, to=1-9]
	\arrow["\pi", shift left, from=1-9, to=1-7]
\end{tikzcd}\]
	The pushforward $j_{\C, \ast} F_\bullet$ is then identified with
\[\begin{tikzcd}
	\cdots && {\pi E_2 \cap E_1} && {E_2 \cap E_1} && {\pi^{-1}E_2 \cap E_1} && \cdots
	\arrow["id", shift left, from=1-1, to=1-3]
	\arrow["\pi", shift left, from=1-3, to=1-1]
	\arrow["id", shift left, from=1-3, to=1-5]
	\arrow["\pi", shift left, from=1-5, to=1-3]
	\arrow["id", shift left, from=1-5, to=1-7]
	\arrow["\pi", shift left, from=1-7, to=1-5]
	\arrow["id", shift left, from=1-7, to=1-9]
	\arrow["\pi", shift left, from=1-9, to=1-7]
\end{tikzcd}\]
as claimed.
\end{proof}

\begin{lemma}\label{local_criterion_ST}
	Let $\C \in \Prlo_k$ equipped with an admissible $t$-structure $\tau$.
	Let $F_\bullet \in \C_{\ST}$.
	If the restrictions $(j_{\left\lbrace s\neq 0 \right\rbrace})_\C^\ast F_\bullet, (j_{\left\lbrace s\neq 0 \right\rbrace})_\C^\ast F_\bullet, (i_{\left\lbrace s = t = 0 \right\rbrace})_\C^\ast F_\bullet$ are $\tau$-flat, then $F_\bullet$ is $\tau$-flat.
\end{lemma}

\begin{proof}
	By \cref{descent_flat_affine} it is enough to show that $p_\C^\ast F_\bullet$ is $\tau$-flat.
	Consider the pullback diagram
\[\begin{tikzcd}[sep=small]
	{\G_{m} =\Spec(R[s^{\pm 1}])} && {\Spec(R[s,t]/(st-\pi))} \\
	\\
	{[\G_{m} / \G_{m}] =\Spec(R)} && \ST
	\arrow["{j_{ \left\lbrace s \neq 0 \right\rbrace}}", hook, from=1-1, to=1-3]
	\arrow["q"', from=1-1, to=3-1]
	\arrow["p", from=1-3, to=3-3]
	\arrow["{j_{\left\lbrace s \neq 0 \right\rbrace}}"', hook, from=3-1, to=3-3]
\end{tikzcd}\]
where the vertical maps are the atlas maps and the horizontal are open immersions.
	By hypothesis, $(j_{\left\lbrace s \neq 0 \right\rbrace})_\C^\ast F_\bullet$ is $\tau_R$-flat, which in turn implies that $q_\C^\ast (j_{\left\lbrace s \neq 0 \right\rbrace})_\C^\ast F_\bullet \simeq (j_{\left\lbrace s \neq 0 \right\rbrace})_\C^\ast p^\ast F_\bullet$ is $\tau_{R[s^{\pm 1}]}$-flat.
	Similarly, the restriction of $p_\C^\ast F_\bullet$ to the locus $\left\lbrace t \neq 0 \right\rbrace$ is $\tau_{R[t^{\pm 1}]}$-flat.
	Consider now the pullback diagram
\[\begin{tikzcd}[sep=small]
	{\Spec(\kappa)} && {\Spec(R[s,t]/(st-\pi))} \\
	\\
	{\mathrm{B}\G_{m}} && \ST
	\arrow["{i_{\left\lbrace s = t = 0 \right\rbrace}}", hook, from=1-1, to=1-3]
	\arrow["g"', from=1-1, to=3-1]
	\arrow["p", from=1-3, to=3-3]
	\arrow["{i_{\left\lbrace s = t = 0 \right\rbrace}}"', hook, from=3-1, to=3-3]
\end{tikzcd}\]
	Arguing as above we find that $(i_{\left\lbrace s = t = 0 \right\rbrace})_\C^\ast p_\C^\ast F_\bullet$ is $\tau_{\kappa}$-flat.
	Hence $p_\C^\ast F_\bullet$ is $\tau$-flat by \cref{local_criterion_flatness}.
\end{proof}

\bibliographystyle{alpha}
\bibliography{good_bib_arxiv}

\newcommand{\etalchar}[1]{$^{#1}$}
\begin{thebibliography}{BLM{\etalchar{+}}21}

\bibitem[AG14]{AG}
Benjamin Antieau and David Gepner.
\newblock Brauer groups and {\'e}tale cohomology in derived algebraic geometry.
\newblock {\em Geom. Topol.}, 18(2):1149--1244, 2014.

\bibitem[AHLH23]{AHLH}
Jarod Alper, Daniel Halpern-Leistner, and Jochen Heinloth.
\newblock Existence of moduli spaces for algebraic stacks.
\newblock {\em Inventiones mathematicae}, 234(3):949–1038, August 2023.

\bibitem[AHPS23]{derived_good}
Eric Ahlqvist, Jeroen Hekking, Michele Pernice, and Michail Savvas.
\newblock Good {Moduli} {Spaces} in {Derived} {Algebraic} {Geometry}.
\newblock Preprint, {arXiv}:2309.16574 [math.{AG}] (2023), 2023.

\bibitem[AHR25]{AHR}
Jarod Alper, Jack Hall, and David Rydh.
\newblock The \'etale local structure of algebraic stacks.
\newblock Preprint, {arXiv}:1912.06162 [math.{AG}] (2025), 2025.

\bibitem[Alp13]{Alp}
Jarod Alper.
\newblock Good moduli spaces for {Artin} stacks.
\newblock {\em Ann. Inst. Fourier}, 63(6):2349--2402, 2013.

\bibitem[AOV08]{AOV}
Dan Abramovich, Martin Olsson, and Angelo Vistoli.
\newblock Tame stacks in positive characteristic.
\newblock {\em Ann. Inst. Fourier}, 58(4):1057--1091, 2008.

\bibitem[AP06]{AP}
Dan Abramovich and Alexander Polishchuk.
\newblock Sheaves of {{\(t\)}}-structures and valuative criteria for stable
  complexes.
\newblock {\em J. Reine Angew. Math.}, 590:89--130, 2006.

\bibitem[AZ01]{AZ}
M.~Artin and J.~J. Zhang.
\newblock Abstract {Hilbert} schemes. {I}.
\newblock {\em Algebr. Represent. Theory}, 4(4):305--394, 2001.

\bibitem[BBDG18]{BBDG}
Alexander Beilinson, Joseph Bernstein, Pierre Deligne, and Ofer Gabber.
\newblock {\em Faisceaux pervers. {Actes} du colloque ``{Analyse} et
  {Topologie} sur les {Espaces} {Singuliers}''. {Partie} {I}}, volume 100 of
  {\em Ast{\'e}risque}.
\newblock Paris: Soci{\'e}t{\'e} Math{\'e}matique de France (SMF), 2nd edition
  edition, 2018.

\bibitem[BDN{\etalchar{+}}25]{BDNIKP}
Chenjing Bu, Ben Davison, Andr{\'e}s~Ib{\'a}{\~n}ez N{\'u}{\~n}ez, Tasuki
  Kinjo, and Tudor P{\u{a}}durariu.
\newblock Cohomology of symmetric stacks.
\newblock Preprint, {arXiv}:2502.04253 [math.{AG}] (2025), 2025.

\bibitem[BLM{\etalchar{+}}21]{stability_in_families}
Arend Bayer, Mart{\'{\i}} Lahoz, Emanuele Macr{\`{\i}}, Howard Nuer, Alexander
  Perry, and Paolo Stellari.
\newblock Stability conditions in families.
\newblock {\em Publ. Math., Inst. Hautes {\'E}tud. Sci.}, 133:157--325, 2021.

\bibitem[BZFN10]{BZFN}
David Ben-Zvi, John Francis, and David Nadler.
\newblock Integral transforms and {Drinfeld} centers in derived algebraic
  geometry.
\newblock {\em J. Am. Math. Soc.}, 23(4):909--966, 2010.

\bibitem[BZNP17]{BZNP}
David Ben-Zvi, David Nadler, and Anatoly Preygel.
\newblock Integral transforms for coherent sheaves.
\newblock {\em J. Eur. Math. Soc. (JEMS)}, 19(12):3763--3812, 2017.

\bibitem[Dav24]{Davison}
Ben Davison.
\newblock Purity and 2-{Calabi}-{Yau} categories.
\newblock {\em Invent. Math.}, 238(1):69--173, 2024.

\bibitem[DHM22]{DHM}
Ben Davison, Lucien Hennecart, and Sebastian~Schlegel Mejia.
\newblock {BPS} {Lie} algebras for totally negative 2-{Calabi}-{Yau} categories
  and nonabelian {Hodge} theory for stacks.
\newblock Preprint, {arXiv}:2212.07668 [math.{RT}] (2022), 2022.

\bibitem[FHLM25]{FHLM}
Andres Fernandez~Herrero, Emmett Lennen, and Svetlana Makarova.
\newblock Moduli of objects in finite length abelian categories.
\newblock Preprint, {arXiv}:2305.10543 [math.{AG}] (2025), 2025.

\bibitem[GM88]{GM}
Mark Goresky and Robert MacPherson.
\newblock {\em Stratified {Morse} theory}, volume~14 of {\em Ergeb. Math.
  Grenzgeb., 3. Folge}.
\newblock Berlin etc.: Springer-Verlag, 1988.

\bibitem[Hai25]{Haine}
Peter~J. Haine.
\newblock From nonabelian basechange to basechange with coefficients.
\newblock {\em J. Pure Appl. Algebra}, 229(7):30, 2025.
\newblock Id/No 107993.

\bibitem[Hen24]{Hennecart}
Lucien Hennecart.
\newblock Cohomological integrality for symmetric quotient stacks.
\newblock Preprint, {arXiv}:2408.15786 [math.{AG}] (2024), 2024.

\bibitem[HL14]{HL}
Daniel Halpern-Leistner.
\newblock On the structure of instability in moduli theory.
\newblock Preprint, {arXiv}:1411.0627 [math.{AG}] (2014), 2014.

\bibitem[HPT24]{exodromyconicality}
Peter~J. Haine, Mauro Porta, and Jean-Baptiste Teyssier.
\newblock Exodromy beyond conicality.
\newblock Preprint, {arXiv}:2401.12825 [math.{AT}] (2024), 2024.

\bibitem[HPT25]{Perv_moduli}
Peter~J. Haine, Mauro Porta, and Jean-Baptiste Teyssier.
\newblock The derived moduli of perverse sheaves.
\newblock In preparation, 2025.

\bibitem[JLMS22]{Weil_restr_note}
Lena Ji, Shizhang Li, Patrick McFaddin, and Matthew Stevenson.
\newblock Weil restriction for schemes and beyond.
\newblock In {\em Stacks project expository collection (SPEC)}, pages 194--221.
  Cambridge: Cambridge University Press, 2022.

\bibitem[Kas84]{Kashiwara}
Masaki Kashiwara.
\newblock The {Riemann}-{Hilbert} problem for holonomic systems.
\newblock {\em Publ. Res. Inst. Math. Sci.}, 20:319--365, 1984.

\bibitem[KM97]{KM}
Se{\'a}n Keel and Shigefumi Mori.
\newblock Quotients by groupoids.
\newblock {\em Ann. Math. (2)}, 145(1):193--213, 1997.

\bibitem[KS90]{KS}
Masaki Kashiwara and Pierre Schapira.
\newblock {\em Sheaves on manifolds. {With} a short history ``{Les} d{\'e}buts
  de la th{\'e}orie des faisceaux'' by {Christian} {Houzel}}, volume 292 of
  {\em Grundlehren Math. Wiss.}
\newblock Berlin etc.: Springer-Verlag, 1990.

\bibitem[Lam25]{good_Stokes}
Enrico Lampetti.
\newblock Good moduli spaces for stokes functors and constructible sheaves.
\newblock Available at the authors' webpage
  \url{https://sites.google.com/view/enricolampettimath}, 2025.

\bibitem[Lie06]{Lieblich}
Max Lieblich.
\newblock Moduli of complexes on a proper morphism.
\newblock {\em J. Algebr. Geom.}, 15(1):175--206, 2006.

\bibitem[LMB00]{LMB}
G{\'e}rard Laumon and Laurent Moret-Bailly.
\newblock {\em Champs alg{\'e}briques}, volume~39 of {\em Ergeb. Math.
  Grenzgeb., 3. Folge}.
\newblock Berlin: Springer, 2000.

\bibitem[Lur09]{HTT}
Jacob Lurie.
\newblock {\em Higher topos theory}, volume 170 of {\em Ann. Math. Stud.}
\newblock Princeton, NJ: Princeton University Press, 2009.

\bibitem[Lur11a]{Tannaka_QCoh}
Jacob Lurie.
\newblock Dag viii: Quasi-coherent sheaves and tannaka duality theorems.
\newblock Available at the authors' webpage
  \url{https://www.math.ias.edu/~lurie/}, 2011.

\bibitem[Lur11b]{DAG_descent}
Jacob Lurie.
\newblock Dag xi: Descent theorems.
\newblock Available at the authors' webpage
  \url{https://www.math.ias.edu/~lurie/}, 2011.

\bibitem[Lur14]{rotation_invariance}
Jacob Lurie.
\newblock Rotation invariance in algebraic {$K$}-theory.
\newblock Available at the authors' webpage
  \url{https://www.math.ias.edu/~lurie/}, 2014.

\bibitem[Lur17]{HA}
Jacob Lurie.
\newblock Higher algebra.
\newblock Available at the authors' webpage
  \url{https://www.math.ias.edu/~lurie/}, 2017.

\bibitem[Lur18]{SAG}
Jacob Lurie.
\newblock Spectral algebraic geometry.
\newblock Available at the authors' webpage
  \url{https://www.math.ias.edu/~lurie/}, 2018.

\bibitem[Mas94]{Massey}
David~B. Massey.
\newblock Numerical invariants of perverse sheaves.
\newblock {\em Duke Math. J.}, 73(2):307--369, 1994.

\bibitem[Meb84]{Mebkhout}
Z.~Mebkhout.
\newblock Une {\'e}quivalence de cat{\'e}gories. {Une} autre {\'e}quivalence de
  cat{\'e}gories.
\newblock {\em Compos. Math.}, 51:51--62, 63--88, 1984.

\bibitem[MFK94]{GIT}
D.~Mumford, J.~Fogarty, and F.~Kirwan.
\newblock {\em Geometric invariant theory.}, volume~34 of {\em Ergeb. Math.
  Grenzgeb.}
\newblock Berlin: Springer-Verlag, 3rd enl. ed. edition, 1994.

\bibitem[Mou21]{Moulinos}
Tasos Moulinos.
\newblock The geometry of filtrations.
\newblock {\em Bull. Lond. Math. Soc.}, 53(5):1486--1499, 2021.

\bibitem[NV23]{NV}
Guglielmo Nocera and Marco Volpe.
\newblock Whitney stratifications are conically smooth.
\newblock {\em Sel. Math., New Ser.}, 29(5):20, 2023.
\newblock Id/No 68.

\bibitem[Pol07]{Polishchuk}
A.~Polishchuk.
\newblock Constant families of {{\(t\)}}-structures on derived categories of
  coherent sheaves.
\newblock {\em Mosc. Math. J.}, 7(1):109--134, 2007.

\bibitem[Por25]{HDR}
Mauro Porta.
\newblock {\em Derived methods in moduli theory}.
\newblock Habilitation {\`a} diriger des recherches, {Universit{\'e} de
  Strasbourg}, 2025.

\bibitem[PT22]{Exodromy}
Mauro Porta and Jean-Baptiste Teyssier.
\newblock Topological exodromy with coefficients.
\newblock Preprint, {arXiv}:2211.05004 [math.{AT}] (2022), 2022.

\bibitem[SS03]{SS}
Stefan Schwede and Brooke Shipley.
\newblock Stable model categories are categories of modules.
\newblock {\em Topology}, 42(1):103--153, 2003.

\bibitem[{Sta}25]{stacks-project}
The {Stacks project authors}.
\newblock The stacks project.
\newblock \url{https://stacks.math.columbia.edu}, 2025.

\bibitem[TV05]{HAG-I}
Bertrand To{\"e}n and Gabriele Vezzosi.
\newblock Homotopical algebraic geometry. {I}: {Topos} theory.
\newblock {\em Adv. Math.}, 193(2):257--372, 2005.

\bibitem[TV07]{TV}
Bertrand To{\"e}n and Michel Vaqui{\'e}.
\newblock Moduli of objects in dg-categories.
\newblock {\em Ann. Sci. {\'E}c. Norm. Sup{\'e}r. (4)}, 40(3):387--444, 2007.

\bibitem[TV08]{HAG-II}
Bertrand To{\"e}n and Gabriele Vezzosi.
\newblock {\em Homotopical algebraic geometry. {II}: {Geometric} stacks and
  applications}, volume 902 of {\em Mem. Am. Math. Soc.}
\newblock Providence, RI: American Mathematical Society (AMS), 2008.

\end{thebibliography}

\end{document}